\documentclass[11pt]{article}

\usepackage{etoolbox}

\newtoggle{SIOPT}

\newtoggle{easyexamples}

\usepackage{jcd}

\usepackage{mathtools}
\usepackage{color}
\usepackage{comment}
\usepackage{xspace}
\usepackage{tikz}

\usepackage{overpic}

\usepackage[title]{appendix}

\definecolor{darkblue}{rgb}{0,0,.75}
\usepackage{graphicx}
\usepackage[colorlinks=true,linkcolor=darkblue,%
  citecolor=darkblue,urlcolor=darkblue]{hyperref}
\usepackage[numbers]{natbib}

\usepackage{enumitem}
\usepackage[top=2.54cm,right=3cm,left=3cm,bottom=2.54cm]{geometry}

\newcommand{\statrv}{S}
\newcommand{\statval}{s}
\newcommand{\statdomain}{\mc{S}}
\newcommand{\stepsize}{\alpha}
\newcommand{\opt}{^\star}

\newcommand{\lambdamin}{\lambda_{\min}}
\newcommand{\lambdamax}{\lambda_{\max}}

\newcommand{\gapfunc}{\Gamma}

\newcommand{\growfunc}{\mathsf{G}_{\rm big}}

\providecommand{\steppow}{\beta}

\newcommand{\lipobj}{M}
\newcommand{\lipgrad}{L}
\newcommand{\poissondist}{\mathsf{Poisson}}

\newcommand{\aProx}{\textsc{aProx}\xspace}

\newcommand{\stabilityfunc}{C}
\newcommand{\sphere}{\mathbb{S}}
\newcommand{\ball}{\mathbb{B}}

\begin{document}

\begin{center}
  {\Large Stochastic (Approximate) Proximal Point Methods: \\ Convergence,
    Optimality, and Adaptivity\footnote{Research partially supported by
      NSF-CAREER Award 1553086, ONR-YIP N00014-19-1-2288, and the Sloan
      Foundation.}}
  
  \vspace{.3cm}
  
  \begin{tabular}{cc}
    Hilal Asi & John C.\ Duchi \\
    \texttt{asi@stanford.edu} &
    \texttt{jduchi@stanford.edu} \\
    \multicolumn{2}{c}{
      Stanford University
    } \\
  \end{tabular}
\end{center}

\begin{abstract}
  We develop model-based methods for solving stochastic convex optimization
  problems, introducing the approximate-proximal point, or \aProx, family,
  which includes stochastic subgradient, proximal point, and bundle
  methods. When the modeling approaches we propose are appropriately
  accurate, the methods enjoy stronger convergence and robustness guarantees
  than classical approaches, even though the model-based methods typically
  add little to no computational overhead over stochastic subgradient
  methods.  For example, we show that improved models converge with
  probability 1 and enjoy optimal asymptotic normality results under weak
  assumptions; these methods are also adaptive to a natural class of what we
  term easy optimization problems, achieving linear convergence under
  appropriate strong growth conditions on the objective. Our substantial
  experimental investigation shows the advantages of more accurate modeling
  over standard subgradient methods across many smooth and non-smooth
  optimization problems.
\end{abstract}



\section{Introduction}
\label{sec:intro}

In this paper, we develop and analyze a family of model-based methods,
moving beyond naive stochastic gradient methods, for solving the stochastic
convex optimization problem
\begin{equation}
  \label{eqn:objective}
  \begin{split}
    \minimize ~ & F(x) = \E_P[f(x;\statrv)]
    = \int_\statdomain f(x; \statval) dP(\statval) \\
    \subjectto ~ & x \in \mc{X}.
  \end{split}
\end{equation}
In problem~\eqref{eqn:objective}, the set $\statdomain$ is a sample space,
and for each $\statval \in \statdomain$, the function $f(\cdot; \statval) :
\R^n \to \R$ is a closed convex function, subdifferentiable
on the closed convex set $\mc{X} \subset \R^n$.

Stochastic minimization problems, in which an optimizer has
access to samples $\statrv_i$ drawn independently and identically
distributed from $P$ and uses these samples to minimize $F$, have
applications in numerous fields, including machine learning, statistical
estimation, and simulation-based optimization~\cite{Zhang04,
  HastieTiFr09, ShapiroDeRu09}. The current accepted
methodology for such problems is the stochastic (sub)gradient
method~\cite{Zinkevich03, Zhang04, NemirovskiJuLaSh09, BottouBo07,
  ShalevSiSrCo11}, which \citet{RobbinsMo51} originally developed for smooth
stochastic approximation problems, which iterates as follows: beginning
at
an initial point $x_1$, iteratively draw $\statrv_k \simiid P$
and update
\begin{equation}
  \label{eqn:sgm}
  x_{k + 1} \defeq x_k - \stepsize_k g_k ~~ \mbox{for~some}~
  g_k \in \partial f(x_k; \statrv_k).
\end{equation}
The stochastic gradient method enjoys convergence 
guarantees~\cite{Zinkevich03, NemirovskiJuLaSh09} and widespread empirical
success in large-scale convex and non-convex stochastic
optimization~\cite{Zhang04, BottouBo07, ShalevSiSrCo11,
  DeanCoMoChDeMaRaSeTuYaNg12, LeCunBeHi15}.  In spite of this success, there
are notable difficulties with the stochastic subgradient
method~\eqref{eqn:sgm}: it is sensitive to stepsize selection; it can
diverge on objectives, such as $F(x) = x^4$, that do not obey its
convergence criteria; and it is rarely adaptive to nuanced aspects
of problem difficulty. Engineers thus waste time and computation dealing
with these issues and finding appropriate stepsizes, which cascades into
additional practical challenges.

An alternative to treating the stochastic gradient method~\eqref{eqn:sgm}
(SGM) as a noisy approximation to gradient descent is to view it as
minimizing a sequence of random \emph{models} of the functions $F$ and $f$,
and we leverage this view here. In this context, SGM
makes a linear approximation to the instantaneous function $f$
around the point $x_k$, setting
\begin{equation*}
  f_{x_k}(x; \statrv_k) \defeq
  f(x_k; \statrv_k) + \<g_k, x - x_k\>
\end{equation*}
and choosing $x_{k+1}$ to minimize the regularized model $f_{x_k}(x;
\statrv_k) + \frac{1}{2 \stepsize_k} \ltwos{x - x_k}^2$. More sophisticated
models are plausible. Most familiar is the stochastic proximal point
method~\cite{Rockafellar76, KulisBa10, Bertsekas11, KarampatziakisLa11,
  Bianchi16} (the least-mean-squares algorithm~\cite{WidrowHo60} for
quadratic $f$), which makes no approximation, using $f_{x_k}(x; \statval) =
f(x; \statval)$ and iterating
\begin{equation}
  \label{eqn:sppm}
  x_{k+1} = \argmin_{x \in \mc{X}}
  \left\{ f(x ; \statrv_k) + \frac{1}{2 \stepsize_k}
  \ltwo{x - x_k}^2 \right\}.
\end{equation}
This modeling perspective is important in non-stochastic optimization, where
(for example) Newton, Gauss-Newton, bundle, and trust-region
methods~\cite[e.g.][]{HiriartUrrutyLe93ab, BoydVa04, Nesterov04,
  NocedalWr06} explicitly build sequences of easier-to-minimize models while
minimizing the global function $F$.  A substantial body of work investigates
this modeling perspective~\cite{Burke85,
  DrusvyatskiyLe18}, and recent work by \citet{DuchiRu18c} and
\citet{DavisDr19} demonstrates
convergence for appropriate models in weakly convex stochastic
optimization, motivating our approach.

We show how to extend this modeling perspective to stochastic convex
optimization problems, leveraging it to build a new family of algorithms for
solving problem~\eqref{eqn:objective}, which, in homage to the stochastic
proximal point iteration~\eqref{eqn:sppm}, we call the \aProx (approximate
proximal point) algorithms, with substantially better theoretical guarantees
and empirical performance than naive stochastic subgradient methods. The
\aProx algorithms iterate as follows: for $k = 1, 2, \ldots$, we draw a
random $\statrv_k \simiid P$, then update the iterate $x_k$ by minimizing a
regularized approximation to $f(\cdot; \statrv_k)$, setting
\begin{equation}
  \label{eqn:model-iteration}
  x_{k+1} \defeq \argmin_{x \in \mc{X}}
  \left\{ f_{x_k}(x ; \statrv_k) + \frac{1}{2 \stepsize_k}
  \ltwo{x - x_k}^2 \right\}.
\end{equation}
The function $f_x(\cdot ; \statval)$ is a \emph{model} of $f(\cdot;
\statval)$ at the point $x$, meaning that $f_x$ satisfies the
following conditions on its structure and local approximation
properties for $f$:
\begin{enumerate}[label=(C.\roman*),leftmargin=*]
\item \label{cond:convex-model}
  The function $y \mapsto f_x(y; \statval)$ is convex
  and subdifferentiable on $\mc{X}$.
\item \label{cond:lower-model}
  The model $f_x$ satisfies the equality
  $f_x(x; \statval) = f(x; \statval)$ and
  \begin{equation*}
    f_x(y; \statval) \le f(y; \statval) ~~ \mbox{for~all~} y.
  \end{equation*}
\end{enumerate}
\noindent
By the first-order conditions for convexity, for any $g \in \partial_y
f_x(y; \statval)|_{y=x}$ Conditions~\ref{cond:convex-model}
and~\ref{cond:lower-model} imply $f(y; \statval) \ge f_x(y; \statval) \ge
f_x(x; \statval) + \<g, y - x\> = f(x; \statval) + \<g, y - x\>$, yielding
the containment
\begin{equation}
  \label{eqn:subgrad-containment}
  \left.\partial_y f_x(y; \statval)\right|_{y = x}
  \subset \partial_x f(x; \statval).
\end{equation}
\citet{DavisDr19} and
\citet{DuchiRu18c} consider similar modeling conditions, and they inspire
our treatment here. See Section~\ref{sec:aProx} and
Figure~\ref{fig:model-illustrations} for examples.

The \aProx methodology~\eqref{eqn:model-iteration} is flexible in that it
allows many possible modeling choices.  As we shall see, though stochastic
gradient~\eqref{eqn:sgm} and proximal point~\eqref{eqn:sppm} methods are
both special cases, they possess quite different behavior.  Thus, it is
interesting to provide conditions on the accuracy of the models $f_x$, in
addition to \ref{cond:convex-model}--\ref{cond:lower-model}, that imply
stronger convergence guarantees than those available for stochastic gradient
and other simple methods. To describe our contributions at a high level, we
list two assumptions that we frequently make. As $f(\cdot; \statval)$ is
subdifferentiable on $\mc{X}$, there exist measurable selections $f'(x;
\statval) \in \partial f(x; \statval)$, and $\E[\partial f(x; \statrv)] =
\partial F(x)$ for $x \in \mc{X}$ (cf.~\cite[Sec.~2]{Bertsekas73}).

\begin{assumption}
  \label{assumption:convex}
  The set $\mc{X}\opt \defeq \argmin_{x \in \mc{X}} \{F(x)\}$ is non-empty,
  and there exists $\sigma^2 < \infty$
  such that $\E[\ltwo{f'(x\opt;\statrv)}^2] \leq \sigma^2$
  for $x\opt \in \mc{X}\opt$ 
  and all measurable selections $f'(x\opt; \statval) \in \partial
  f(x\opt; \statval)$.
\end{assumption}

\begin{assumption}
  \label{assumption:very-weak-moment}
  There exists a non-decreasing function $\growfunc : \R_+ \to
  \openright{0}{\infty}$ such that for all $x \in
  \mc{X}$ and measurable selections $f'(x; \statval) \in \partial
  f(x; \statval)$,
  $\E[\norm{f'(x; \statrv)}^2] \le \growfunc(\dist(x, \mc{X}\opt))$.
\end{assumption}
\noindent
Assumption~\ref{assumption:very-weak-moment} makes no restrictions on the
growth of the function $\growfunc$, so the second moment of the subgradient
$f'(x; \statrv)$ may grow arbitrarily.  This contrasts with typical
assumptions for stochastic subgradient methods~\cite[e.g.][]{Zinkevich03,
  NemirovskiJuLaSh09} which assume uniform boundedness or second-moment
conditions on subgradients.  Within this context, we take three
thrusts.
\begin{enumerate}[label=\arabic*.,leftmargin=*]
\item First, in Section~\ref{sec:stability}, we develop conditions for the
  \emph{stability} of iterates~\eqref{eqn:model-iteration} under
  Assumption~\ref{assumption:convex}, meaning that they remain in a
  bounded neighborhood of the optimal solution set $\mc{X}\opt$ of
  problem~\eqref{eqn:objective}.  We leverage this stability to prove
  convergence for the \aProx
  iteration~\eqref{eqn:model-iteration} even for functions with substantial
  variation in their gradient estimates (e.g.\ the gradient may grow
  super-exponentially in $\norm{x}$) to which standard
  results do not apply.  As a consequence, we
  extend \citeauthor{PolyakJu92}'s analysis of
  averaged stochastic gradient methods to all \aProx
  models~\eqref{eqn:model-iteration}---showing asymptotic
  normality with optimal covariance under weaker conditions than
  those necessary for classical situations---so long as the
  iterates are bounded, highlighting the importance of this stability.
\item In our second thrust, in Section~\ref{sec:aProx-adv-fast-conv}, we
  study the performance of \aProx methods~\eqref{eqn:model-iteration} for
  what we term \emph{easy} problems. In these problems there exists a shared
  minimizer $x\opt$ common to all the sampled functions. This assumption is
  strong, yet many problems are easy: in statistical learning,
  \citet{BelkinRaTs18, BelkinHsMi18} show that functions that
  \emph{perfectly interpolate} the observed data (suffering zero loss on the
  observations) can achieve optimal statistical convergence;
  Kaczmarz algorithms solve consistent over-parameterized linear
  systems~\cite{StrohmerVe09,NeedellWaSr14, NeedellTr14};
  the problem of finding a point in the intersection of convex sets assumes
  there exists a point in each of them~\cite{LeventhalLe10,
    BauschkeBo96}.  By incorporating a simple lower bound
  condition---basically, that if $f$ is non-negative, any model $f_x$ should
  also be non-negative---in addition to
  conditions~\ref{cond:convex-model}--\ref{cond:lower-model}, we show how
  \aProx adapts to these easy problems and achieves (near) linear
  convergence, even in stochastic settings,
  using methods with no additional computational complexity
  beyond stochastic gradient methods.
\item Finally, in Section~\ref{sec:everything-is-fine}, we present
  representative non-asymptotic convergence guarantees.
  Any \aProx method~\eqref{eqn:model-iteration} recovers standard
  convergence guarantees of stochastic gradient
  and proximal point methods~\cite{NemirovskiJuLaSh09, Bertsekas11}
  (see \cite{DavisDr19} for the basic convergence guarantee).
  We show how stochastic proximal point methods enjoy fast convergence
  under (restricted) strong convexity with only weak moment conditions on
  $\partial f(x;\statval)$, further emphasizing the advantages of accurate
  modeling.
\end{enumerate}

In addition to our theoretical results, we perform substantial simulations.
Our experiments consider a wide range
of smooth, non-smooth, and super-polynomially-growing convex problems:
regression with squared and absolute losses, logistic and poisson
regression, and projection problems onto intersections of halfspaces
(relating these to classification problems). The common refrain in each
of these is that even slightly improved \aProx models are much
more robust to stepsize choice than stochastic
gradient methods, and more careful modeling allows fast convergence in
a much broader range of problems, including those with moderately poor
conditioning where stochastic gradient methods fail.

\subsection{Related work}

We situate our paper in relation to classical and modern
work on stochastic optimization problems.  Stochastic gradient methods are
classical, beginning with the development by \citet{RobbinsMo51}
in the 1950s~\cite{Polyak87, PolyakJu92, Zinkevich03,
  NemirovskiJuLaSh09, Zhang04, KushnerYi03}. A number of authors recognize
the challenges associated with stepsize selection and instability of
stochastic gradient methods: in the case of smooth strongly convex
minimization, \citet{NemirovskiJuLaSh09} show how a slightly mis-specified
stepsize can cause arbitrarily slow convergence guarantees. More
recent work (e.g.~\cite{BachMo11}) shows
that even when assumptions sufficient for
convergence hold, stochastic gradient methods can exhibit transient
divergent behavior.

In effort to alleviate some of these issues, there is recent
work on more careful approaches to stochastic optimization problems.  Of
most relevance to our work are \emph{stochastic proximal point
  methods}~\eqref{eqn:sppm}, which use the true function $f_x(y; \statval) =
f(y; \statval)$ in the iteration~\eqref{eqn:model-iteration}.
\citet{Bertsekas11} analyzes stochastic proximal point algorithms in an
incremental framework (when $\statdomain = \{1, \ldots, m\}$ is a
finite set), showing convergence results similar to subgradient methods,
while \citet{KulisBa10} and \citet{KarampatziakisLa11} give theoretical and
empirical results in online convex optimization settings, demonstrating
regret bounds similar to classical results~\cite{Zinkevich03}. Toulis
\emph{et al.}~\cite{ToulisAi17,ToulisTrAi16} study stochastic proximal point
algorithms and convergence guarantees for their final iterates, a different
approach than we take. Their results, however, assume that the functions
under consideration are both globally Lipschitz and globally strongly
convex, a contradiction, severely limiting the applicability of
their results. (Their analysis explicitly and
frequently uses both assumptions;
under weaker assumptions, their convergence
guarantees exhibit the same potential for exponential divergence
of stochastic gradient methods.)
\citet{PatrascuNe17} also analyze stochastic proximal point
algorithms, providing non-asymptotic convergence results under the
assumption that each function $f(\cdot; \statval)$ is Lipschitz or
strongly convex~\cite[Assumptions 1 \& 9]{PatrascuNe17}; these assumptions
fail for many problems (including linear regression with
$f(x; (a, b)) = \half (\<a, x\> - b)^2$), though their
results also apply to interesections of sets
$\mc{X} = \mc{X}_1 \cap \cdots \cap \mc{X}_m$.

\citet{RyuBo14} also investigate the stochastic proximal point method,
making arguments on its stability stronger than classical results for
stochastic gradient methods.  Under Assumption~\ref{assumption:convex},
\citeauthor{RyuBo14} show that the stochastic proximal point method
guarantees $\E[\ltwo{x_{k+1} - x\opt}] \le \E[\ltwo{x_1 - x\opt}] + \sigma
\sum_{i = 1}^k \stepsize_i$, so that the iterates do not diverge
exponentially. Yet this result is not enough to explain or
provide empirical or theoretical stability, boundedness, or convergence
guarantees.


\paragraph{Notation}
For a convex function $f$, $\partial f(x)$ denotes its subgradient set at
$x$, and $f'(x) \in \partial f(x)$ denotes an arbitrary element of
the subdifferential.  Throughout, $x\opt$ denotes a minimizer of
problem~\eqref{eqn:objective} and $\mc{X}\opt = \argmin_{x \in \mc{X}} F(x)$
its optimal set.  We let $\mc{F}_k \defeq \sigma(\statrv_1, \ldots,
\statrv_k)$ be the $\sigma$-field generated by the first $k$ random
variables $\statrv_i$, so $x_k \in
\mc{F}_{k-1}$ for all $k$ under iteration~\eqref{eqn:model-iteration}.


\section{Methods}
\label{sec:aProx}

\begin{figure}[t]
  \begin{center}
    \begin{tabular}{cc}
      \hspace{-.5cm}
      \begin{overpic}[width=.5\columnwidth]{
          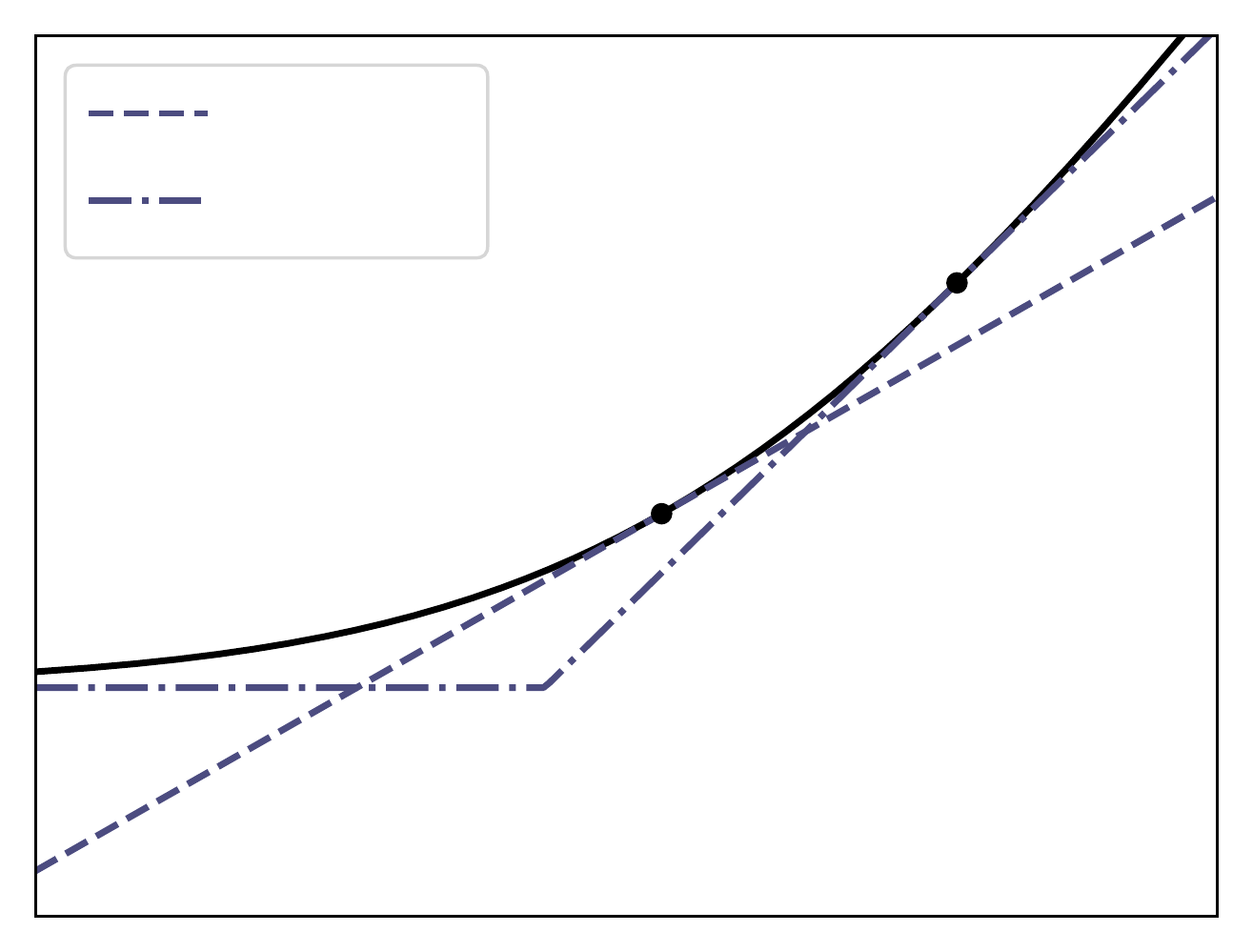}
        \put(18,65.5){\footnotesize{Linear}}
        \put(18,58.5){\footnotesize{Truncated}}
        \put(49.4,37){\footnotesize{$x_0$}}
        \put(73,55.5){\footnotesize{$x_1$}}
      \end{overpic} &
      \includegraphics[width=.5\columnwidth]{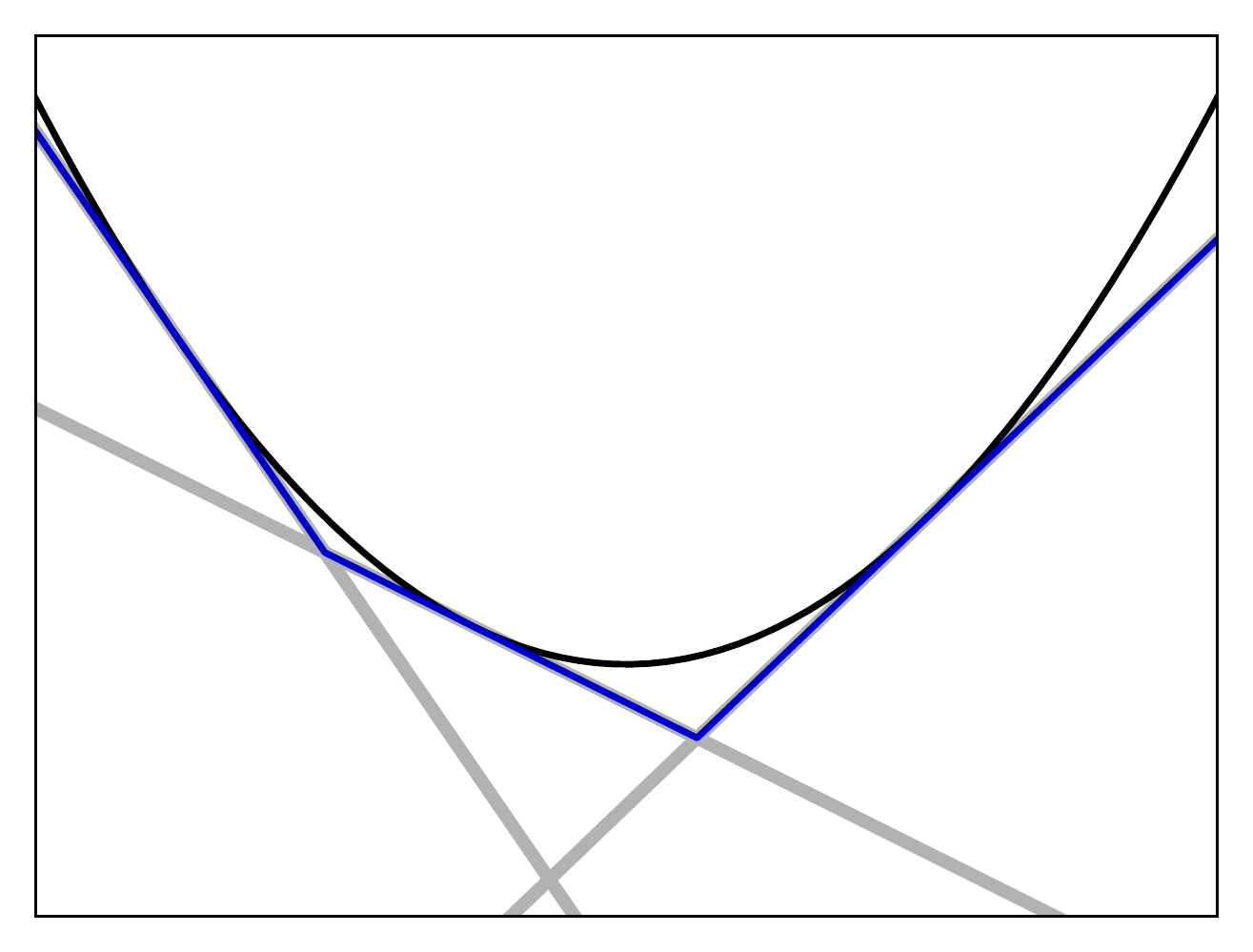}
      \\
      (a) & (b)
    \end{tabular}
    \caption{\label{fig:model-illustrations} (a) Models of the function
      $f(x) = \log(1 + e^{-x})$: a linear
      model~\eqref{eqn:dumb-linear-model} built around the point $x_0$ and
      truncated model~\eqref{eqn:trunc-model} built around the point
      $x_1$. (b) The bundle model, maximum of linear
      functions, as in the iteration~\eqref{eqn:bundling}. The
      lighter lines represent individual linear approximations, the
      darker line their maximum.}
  \end{center}
\end{figure}

We begin our contributions by introducing different natural models for
stochastic convex optimization problems, as well as a few conditions in
addition to \ref{cond:convex-model}--\ref{cond:lower-model} that we can
use to demonstrate new aspects of convergence for the \aProx family. While
the stochastic proximal point method~\eqref{eqn:sppm} satisfies all the
conditions in the paper, in some situations it may be expensive or
challenging to implement exactly.  With that in mind, we provide a catalogue
of a few models to serve as a reference for the remainder of the paper.

\paragraph{Stochastic subgradient methods:}
The starting point for any model-based methods are the simple first-order
models. As we discuss in the introduction, the stochastic subgradient method
uses the model
\begin{equation}
  \label{eqn:dumb-linear-model}
  f_{x}(y; \statval) \defeq f(x; \statval) + \<f'(x; \statval), y - x\>,
\end{equation}
where $f'(x; \statval) \in \partial f(x; \statval)$ is an arbitrary
element of the subdifferential.
The model~\eqref{eqn:dumb-linear-model}
satisfies conditions~\ref{cond:convex-model}--\ref{cond:lower-model}
by convexity.

\paragraph{Proximal point methods:}
The stochastic proximal point method uses the ``model''
\begin{equation}
  \label{eqn:prox-model}
  f_x(y; \statval) \defeq f(y; \statval),
\end{equation}
that is, the true function.
The model~\eqref{eqn:prox-model} satisfies all the conditions
we provide.

\paragraph{Truncated models:}


The first condition beyond~\ref{cond:convex-model}--\ref{cond:lower-model}
builds out of the simple observation that, if
one is minimizing a nonnegative function (for example, in most
machine learning and statistical applications with a loss function), then
\emph{a priori} a model of the function that takes
negative values cannot be accurate.  If $f(x; \statval) \ge 0$
for all $x$, a better approximation to $f$ than the linear
model~\eqref{eqn:dumb-linear-model} is to
take $f'(x;\statval) \in \partial f(x; \statval)$ and define
\begin{equation*}
  f_x(y; \statval) \defeq \hinge{f(x; \statval) + \<f'(x; \statval),
    y - x\>},
\end{equation*}
where $f'(x; \statval) \in \partial f(x; \statval)$. More generally, we may
consider models that provide a lower guarantee:
\begin{enumerate}[label=(C.\roman*),leftmargin=*]
  \setcounter{enumi}{2}
\item \label{cond:lower-by-optimal}
  For all $\statval \in \statdomain$, the models $f_x(\cdot; \statval)$ satisfy
  \begin{equation*}
    f_x(y; \statval) \ge \inf_{z \in \mc{X}}
    f(z; \statval).
  \end{equation*}
\end{enumerate}

\noindent
Thus, if we are given an oracle that, for each fixed $\statval \in \statdomain$,
can compute the minimal value $\inf_{z \in \mc{X}} f(z; \statval)$,
we may consider the truncated models
\begin{equation}
  \label{eqn:trunc-model}
  f_x(y; \statval) \defeq \max\left\{
  f(x; \statval) + \<f'(x;\statval), y - x\>,
  \inf_{z \in \mc{X}} f(z; \statval)\right\}.
\end{equation}
See Figure~\ref{fig:model-illustrations}(a) for an illustration
of this model.

Many statistical, machine learning, and signal-processing examples support
this model, because for any individual sample $\statval$ we have
$\inf_{z \in \mc{X}} f(z; \statval) = 0$.  For example, in linear
regression, $\statval = (a, b) \in \R^n \times \R$, and $\inf_z (\<a,
z\> - b)^2 = 0$ when $a \neq 0$. In logistic regression~\cite{HastieTiFr09},
we have $\statval = (a, b) \in \R^n \times \{-1, 1\}$, and $f(x; (a,
b)) = \log(1 + \exp(-b \<a, x\>))$ satisfies $\inf_z f(z; (a, b)) = 0$
unless $a = 0$. Support vector machines use $f(x; (a, b)) = \hinge{1 - b
  \<a, x\>}$, which again has infimal value 0.  Even in more complex
scenarios, these infimal values may be easy to compute; see, for example,
our discussion of poisson regression in
Section~\ref{sec:asymptotic-normality},
Example~\ref{example:poisson-regression}.

\paragraph{Relatively accurate models:} 
Now we consider an additional condition on accuracy, which allows less
accurate models than the exact model~\eqref{eqn:sppm}; for example, this
allows the bundle model, which approximates $f(\cdot; \statval)$ by the
maximum of affine lower bounds, to come. As we see in the sequel
(Theorem~\ref{theorem:prox-point-bounded}), this condition is sufficient for
strong stability and convergence guarantees for any \aProx method using the
model-based updates~\eqref{eqn:model-iteration}.
We require a bit more
notation. Let $f_{x_0}(\cdot; \statval)$ be a model centered at $x_0$
satisfying Conditions~\ref{cond:convex-model}--\ref{cond:lower-model}.
For $\stepsize > 0$ define
\begin{equation*}
  x_\stepsize \defeq \argmin_{x \in \mc{X}}
  \left\{ f_{x_0}(x; \statval) + \frac{1}{2 \stepsize}
  \ltwo{x - x_0}^2 \right\},
\end{equation*}
which is the result of a single update (leaving dependence
on $\statval$ implicit).  Then we consider
\begin{enumerate}[label=(C.\roman*),leftmargin=*]
  \setcounter{enumi}{3}
\item \label{cond:accurate-model}
  For some $\epsilon > 0$, there exists a function $\stabilityfunc :
  \statdomain \to \R_+$ with $\E[\stabilityfunc(\statrv)] < \infty$
  such that for all $x_0 \in \mc{X}$, the updated point
  $x_\stepsize$ and model $f_{x_0}(\cdot; \statval)$ satisfy
  \begin{equation*}
    f(x_\stepsize; \statval)
    \le f_{x_0}(x_\stepsize; \statval)
    + \frac{1 - \epsilon}{2 \stepsize} \ltwo{x_\stepsize - x_0}^2
    + \stabilityfunc(\statval) \stepsize.
  \end{equation*}
\end{enumerate}
The lower bound condition~\ref{cond:lower-model} guarantees that $f(x;
\statval) \ge f_{x_0}(x; \statval)$ for all $x, x_0$, so
\ref{cond:accurate-model} provides a complementary upper bound.
It is clear that the full (proximal point) model~\eqref{eqn:prox-model}
satisfies Condition~\ref{cond:accurate-model} with
$\epsilon = 1$ and $\stabilityfunc(\statval) = 0$, as $f_x = f$.  We term
Condition~\ref{cond:accurate-model} a ``relative'' accuracy because the
necessary approximation scales with $\stepsize$ so that higher accuracy is
necessary as $\stepsize \downarrow 0$, though $\norm{x_\stepsize - x_0} =
O(\stepsize)$ by standard results~\cite{DrusvyatskiyLe18}.

While aside from proximal-point models, it is not clear \emph{a priori}
how to guarantee that Condition~\ref{cond:accurate-model} holds, one
approach is to use bundle methods~\cite{HiriartUrrutyLe93ab,TeoViSmLe10},
identical to Kelley's cutting plane method~\cite{Kelley60}.  In this
situation, we begin from the linear model $f^0_x(y; \statval) = f(x;
\statval) + \<f'(x; \statval), y - x\>$, iteratively construct the lower
piecewise-linear models
\begin{equation}
  \label{eqn:bundling}
  \begin{split}
    x_\stepsize^i & \defeq \argmin_{y \in \mc{X}}
    \left\{f_x^{i-1}(y) + \frac{1}{2 \stepsize} \ltwo{y - x}^2 \right\} \\
    \mbox{and} ~~
    f_x^i(y) & \defeq \max\left\{f_x^{i-1}(y; \statval),
    f(x_\stepsize^i; \statval) + \<f'(x_\stepsize^i; \statval), y - x_\stepsize^i\>\right\}.
  \end{split}
\end{equation}
Whenever the iterate $x_\stepsize^i$ satisfies
Condition~\ref{cond:accurate-model}, we may terminate the iteration, as
$f_x^{i-1}(y)$ satisfies
Conditions~\ref{cond:convex-model}--\ref{cond:lower-model} by
construction.  While we do not address this in this paper,
the number of iterations to achieve a
solution satisfying Condition~\ref{cond:accurate-model} is
at most $O(\norm{f'(x_0; \statval)}^2)$, as each
step solves a strongly convex optimization problem
(see~\cite[Sec.~2.4]{TeoViSmLe10}).

\section{Stability and its consequences}
\label{sec:stability}

The first of our main thrusts, upon which we focus
in this section, is the stability and boundedness of the iterates for
stochastic proximal point methods and their relatives in the \aProx family
of methods. These stability guarantees are in strong contrast to standard
stochastic subgradient methods, which may diverge for problems on which
\aProx methods converge.

We begin with our definition of stability. In this definition, we let
$\mc{A}$ denote the set of positive stepsize sequences $\{\stepsize_k\}$
with $\sum_k \stepsize_k^2 < \infty$. We call a pair $(\mc{F}, \mc{P})$ a
\emph{collection of problems} if $\mc{P}$ is a collection of probability
measures on a sample space $\statdomain$, and $\mc{F}$ is a collection of
functions $f : \mc{X} \times \statdomain \to \R$, where $f(\cdot; \statval)$
is convex. We make the following definition.
\begin{definition}
  \label{definition:stability}
  An algorithm generating
  iterates $x_k$ according to
  the model-based update~\eqref{eqn:model-iteration}
  is \emph{stable in probability} for the collection of
  problems $(\mc{F}, \mc{P})$ if for all $f \in \mc{F}$ and $P \in \mc{P}$
  defining $F(x) = \E_P[f(x; \statrv)]$
  and $\mc{X}\opt = \argmin_{x \in \mc{X}} F(x)$, and for all 
  stepsize sequences $\{\stepsize_k\} \in \mc{A}$
  \begin{equation}
    \sup_k \dist(x_k, \mc{X}\opt) < \infty
    ~~ \mbox{with~probability~1}.
    \label{eqn:bounded-distance}
  \end{equation}
\end{definition}

Classical results coupled with the Robbins-Siegmund
supermartingale convergence theorem~\cite{RobbinsSi71} guarantee the
stochastic subgradient method satisfies $\sup_k \dist(x_k, \mc{X}\opt) <
\infty$ whenever
\begin{equation*}
  \E\left[\ltwo{f'(x; \statrv)}^2\right] \le C_0 + C_1 \dist(x, \mc{X}\opt)^2
  ~~~ \mbox{for~all~} x \in \mc{X},
\end{equation*}
which typical smoothness or boundedness conditions imply
(cf.~\cite{RobbinsSi71, PolyakJu92, Bertsekas99}). Stochastic (approximate)
proximal point approaches allow us to move beyond these quadratic growth
assumptions; in contrast, for objectives for which the gradients grow more
than quadratically, condition~\eqref{eqn:bounded-distance} typically
fails for gradient methods:

\begin{example}[Divergence for non-quadratics]
  \label{example:non-quadratic-divergence}
  Let $F(x) = \frac{1}{4} x^4$, and consider any sequence of stepsizes
  $\stepsize_k > 0$ satisfying $\stepsize_{k+1} \ge \frac{1}{4}
  \stepsize_k$ for all $k$, and let
  $x_{k+1} = x_k - \stepsize_k F'(x_k)$ be generated by the gradient
  method.
  Then whenever the initial iterate $x_1$ satisfies
  $|x_1| \ge \sqrt{3 / \stepsize_1}$,
  we have $|x_k| \ge 2 |x_{k-1}|$ for all $k$,
  so
  $|x_k| \ge 2^k |x_1|$ for all $k \in \N$.
\end{example}

When the objective grows faster than polynomially, even worse behavior is
possible; for example, for the objective $F(x) = (e^x + e^{-x})$, for any
polynomially decreasing stepsize sequence, if $x_1$ is large enough we have
the super-exponential divergence $|x_k| \ge 2^{2^k} |x_1|$.  Standard
stochastic gradient methods may also exhibit undesirable behavior for easier
problems; even in situations for which the objective is smooth and in which
there is no noise, (sub)gradient methods may suffer transient exponential
growth, as the following example demonstrates.

\begin{example}[Instability for quadratics]
  \label{example:quadratic-instability}
  Let $F(x) = \half x^2$. Then the gradient method iterates $x_{k + 1} = (1
  - \stepsize_k) x_k$. Let us assume that $\stepsize_k = \stepsize_0
  k^{-\steppow}$, and let $\stepsize_0 \ge 3 K^\steppow$ for some $K \in
  \N$. Then assuming $x_1 \neq 0$, for all $k \le K$ we have $|x_{k+1}| = |1
  - \stepsize_k| |x_k| \ge 2 |x_k|$, so that $|x_{k+1}| \ge 2^k$ for $k \le
  K$. Classical guarantees for the (stochastic) gradient method show that
  $x_k$ will converge eventually, but even for smooth quadratics, gradient
  descent suffers from exponential growth behavior if the stepsize is
  mis-specified.
\end{example}

While stylized,
these examples highlight the difficulty of naive application of gradient
methods; in the sequel, we
show how accurate models alleviate the issues in the
examples.

\subsection{Stability of (approximate) proximal methods}

The starting point of almost all of what follows are sufficient conditions
on the models we use to guarantee stability as in
Definition~\ref{definition:stability}. For this first result, in addition to
the two conditions~\ref{cond:convex-model}--\ref{cond:lower-model}, we
assume the models are accurate at their updated points, that is,
Condition~\ref{cond:accurate-model}.  The full model (stochastic proximal
point method) satisfies these conditions, and
so too do bundle models~\eqref{eqn:bundling}.  In the
theorem, recall the $\sigma$-field
$\mc{F}_k \defeq \sigma(\statrv_1, \ldots, \statrv_k)$.

\begin{theorem}
  \label{theorem:prox-point-bounded}
  Let Assumption~\ref{assumption:convex} hold and $x_k$ be generated by the
  iteration~\eqref{eqn:model-iteration} with any model satisfying
  Conditions~\ref{cond:convex-model}--\ref{cond:lower-model}
  and~\ref{cond:accurate-model}.  Then for
  all $x\opt \in \mc{X}\opt$,
  \begin{equation*}
    \E\left[\ltwo{x_{k + 1} - x\opt}^2 \mid \mc{F}_{k-1}\right]
    \le
    \ltwo{x_k - x\opt}^2
    + \stepsize_k^2 \left(\frac{\sigma^2}{\epsilon} +
    \E[\stabilityfunc(\statrv)]\right).
  \end{equation*}
\end{theorem}

\noindent
Before providing the proof of the theorem (see
Section~\ref{sec:proof-prox-point-bounded}), we present a few of its
consequences for stability. By taking $x\opt$ to be the projection of $x_k$ onto
$\mc{X}\opt$, Theorem~\ref{theorem:prox-point-bounded} implies
\begin{equation*}
  \E[\dist(x_{k+1}, \mc{X}\opt)^2 \mid \mc{F}_{k-1}]
  \le \dist(x_k, \mc{X}\opt)^2
    + \stepsize_k^2 \left(\frac{\sigma^2}{\epsilon} +
    \E[\stabilityfunc(\statrv)]\right).
\end{equation*}
A few somewhat more consequential corollaries, at least from the perspective
of our stability Definition~\ref{definition:stability}, follow.
We first have that the expected distance is non-divergent.
\begin{corollary}
  \label{corollary:prox-point-expectation-bounded}
  Let the conditions of Theorem~\ref{theorem:prox-point-bounded} hold.
  For each $k \in \N$,
  \begin{equation*}
    \E\left[\dist(x_{k+1}, \mc{X}\opt)^2 \right]
    \le \E\left[\dist(x_1, \mc{X}\opt)^2\right]
    + \left(\frac{\sigma^2}{\epsilon}
    + \E[\stabilityfunc(\statrv)]\right)
    \sum_{i=1}^k \stepsize_i^2.
  \end{equation*}
\end{corollary}

A second corollary establishes that the iterates
of appropriately accurate \aProx methods are stable.
We require the Robbins-Siegmund
almost supermartingale convergence lemma.
\begin{lemma}[\cite{RobbinsSi71}]
  \label{lemma:robbins-siegmund}
  Let $A_k, B_k, C_k, D_k \ge 0$ be non-negative random variables adapted to
  the filtration $\mc{F}_k$ and satisfying $\E[A_{k + 1} \mid \mc{F}_k] \le
  (1 + B_k) A_k + C_k - D_k$. Then on the event $\{\sum_k B_k < \infty,
  \sum_k C_k < \infty\}$, there is a random $A_\infty < \infty$
  such that $A_k \cas A_\infty$ and $\sum_k D_k < \infty$.
\end{lemma}

\noindent
By applying Theorem~\ref{theorem:prox-point-bounded} with $A_k =
\dist(x_{k+1}, \mc{X}\opt)^2$, $C_k =
\stepsize_{k+1}^2 (\sigma^2 / \epsilon + \E[\stabilityfunc(\statrv)])$,
and $B_k = D_k = 0$ in
Lemma~\ref{lemma:robbins-siegmund}, we have
\begin{corollary}
  \label{corollary:prox-point-probability-bounded}
  Let the conditions of Theorem~\ref{theorem:prox-point-bounded} hold
  and assume $\sum_k \stepsize_k^2 < \infty$. Then
  \begin{equation*}
    \sup_{k \in \N} \dist(x_k, \mc{X}\opt) < \infty
  \end{equation*}
  and $\dist(x_k, \mc{X}\opt)$ converges to some
  finite value with probability 1.
\end{corollary}

Combining Corollaries~\ref{corollary:prox-point-expectation-bounded}
and~\ref{corollary:prox-point-probability-bounded}, we see that the
stochastic proximal point method and its \aProx variants---as long as they
satisfy the accuracy condition~\ref{cond:accurate-model}---are stable
according to Definition~\ref{definition:stability}. This is in strong
contrast to
stochastic gradient methods and their relatives, which can be unstable even
for relatively simple problems.

\subsubsection{Proof of Theorem~\ref{theorem:prox-point-bounded}}
\label{sec:proof-prox-point-bounded}

We now return to the promised proof of
Theorem~\ref{theorem:prox-point-bounded}.  In giving the proof, we present
lemmas on the progress of individual iterates for subsequent use.
These results are
typical of Lyapunov-type arguments for convergence of stochastic gradient
methods~\cite{Zinkevich03, NemirovskiJuLaSh09}.

\begin{lemma}
  \label{lemma:conv-bound}
  Let $h$ be convex and subdifferentiable on a closed convex set $\mc{X}$
  and let $\beta > 0$. Then for all $x_0, x_1, y \in \mc{X}$,
  and $h'(y) \in \partial h(y)$,
  \begin{equation*}
    h(y) - h(x_1)
    \le \<h'(y), y - x_0\>
    + \frac{1}{2 \beta} \norm{x_1 - x_0}^2
    + \frac{\beta}{2} \norm{h'(y)}^2
  \end{equation*}
\end{lemma}
\begin{proof}
  By the first-order conditions for convexity, we have
  \begin{align*}
    h(y) - h(x_1)
    & \le \<h'(y), y - x_1\>
    = \<h'(y), y - x_0\> + \<h'(y), x_0 - x_1\> \\
    & \le \<h'(y), y - x_0\> + \frac{1}{2 \beta} \norm{x_1 - x_0}^2
    + \frac{\beta}{2} \norm{h'(y)}^2,
  \end{align*}
  where the second line uses Young's inequality.
\end{proof}

We also have the following lemma, which gives a one-step progress guarantee
for any algorithm using models satisfying
Conditions~\ref{cond:convex-model}--\ref{cond:lower-model}.
\begin{lemma}
  \label{lemma:single-step-progress}
  Let Condition~\ref{cond:convex-model} hold.
  In each step of the method~\eqref{eqn:model-iteration},
  for any $x \in \mc{X}$,
  \begin{equation*}
    \half \ltwo{x_{k + 1} - x}^2
    \le \half \ltwo{x_k - x}^2
    - \stepsize_k\left[f_{x_k}(x_{k + 1}; \statrv_k)
      - f_{x_k}(x; \statrv_k)\right]
    - \half \ltwo{x_k - x_{k+1}}^2.
  \end{equation*}
\end{lemma}
\begin{proof}
  By the first-order conditions for convex optimization,
  for some $g_k \in \partial f_{x_k}(x_{k+1}; \statrv_k)$ we have that
  $\<\stepsize_k g_k + (x_{k+1} - x_k), y - x_{k+1}\> \ge 0$ for
  all $y \in \mc{X}$.
  Setting $y = x$, we obtain
  \begin{equation*}
    \stepsize_k \<g_k, x_{k+1} - x\>
    \le \<x_{k + 1} - x_k, x - x_{k + 1}\>
    = \half \left[\ltwo{x_k - x}^2
      - \ltwo{x_{k + 1} - x}^2
      - \ltwo{x_{k + 1} - x_k}^2 \right].
  \end{equation*}
  As $f_{x_k}(x; \statrv_k) \ge f_{x_k}(x_{k+1}; \statrv_k) + \<g_k,
  x - x_{k+1}\>$ by Condition~\ref{cond:convex-model}, this gives the
  result.
\end{proof}

With Lemmas~\ref{lemma:conv-bound} and~\ref{lemma:single-step-progress} in
place, we can prove the theorem. Let $x\opt \in \mc{X}\opt$ be an otherwise
arbitrary optimal point.  Applying Lemma~\ref{lemma:single-step-progress}
with $x = x\opt$, we have
\begin{align*}
  \half \ltwo{x_{k + 1} - x\opt}^2
  & \le \half \ltwo{x_k - x\opt}^2
  - \stepsize_k\left[f_{x_k}(x_{k + 1}; \statrv_k)
    - f_{x_k}(x\opt; \statrv_k)\right]
  - \half \ltwo{x_k - x_{k+1}}^2 \\
  & \stackrel{(i)}{\le}
  \half \ltwo{x_k - x\opt}^2 - \stepsize_k \left[f(x_{k+1};
    \statrv_k) - f_{x_k}(x\opt; \statrv_k)\right]
  - \frac{\epsilon}{2} \ltwo{x_k - x_{k+1}}^2
  + \stabilityfunc(\statrv_k) \stepsize_k^2 \\
  & \stackrel{(ii)}{\le}
  \half \ltwo{x_k - x\opt}^2 - \stepsize_k \left[f(x_{k+1};
    \statrv_k) - f(x\opt; \statrv_k)\right]
  - \frac{\epsilon}{2} \ltwo{x_k - x_{k+1}}^2
  + \stabilityfunc(\statrv_k) \stepsize_k^2,
\end{align*}
where inequality~$(i)$ is a consequence of
the accurate model condition~\ref{cond:accurate-model}
and $(ii)$ because $f_{x}(x\opt; \statval) \le f(x\opt; \statval)$
by the lower model condition~\ref{cond:lower-model}.
Now, we apply Lemma~\ref{lemma:conv-bound} with
$x_1 = x_{k+1}$, $x_0 = x_k$, $y = x\opt$, and $\beta = \frac{\stepsize_k}{
  \epsilon}$ to find
\begin{align*}
  \half \ltwo{x_{k + 1} - x\opt}^2
  & \le \half \ltwo{x_k - x\opt}^2
  + \stepsize_k \<f'(x\opt; \statrv_k), x\opt - x_k\>
  + \frac{\stepsize_k^2}{2 \epsilon}
  \ltwo{f'(x\opt; \statrv_k)}^2
  + \stabilityfunc(\statrv_k) \stepsize_k^2
\end{align*}
for all $f'(x\opt; \statrv_k) \in \partial f(x\opt; \statrv_k)$.

For some $F'(x\opt) \in \partial F(x\opt)$, we have $\<F'(x\opt), y -
x\opt\> \ge 0$ for all $y \in \mc{X}$.  As our choice of $f'(x\opt;
\statval) \in \partial f(x\opt; \statval)$ above was arbitrary, we may take
$f'(x\opt; \statrv_k)$ so that $\E[f'(x\opt; \statrv_k)] = F'(x\opt)$ for
any desired $F'(x\opt) \in \partial F(x\opt)$ (cf.~\cite{Bertsekas73}).
Thus, taking expectations with respect to $\mc{F}_{k-1}$,
\begin{equation*}
  \half \E[\ltwo{x_{k + 1} - x\opt}^2 \mid \mc{F}_{k-1}]
  \le \half \ltwo{x_k - x\opt}^2
  + \frac{\stepsize_k^2}{2 \epsilon} \E\left[\ltwo{f'(x\opt; \statrv)}^2\right]
  + \E[\stabilityfunc(\statrv)] \stepsize_k^2
  + \stepsize_k \<F'(x\opt), x\opt - x_k\>.
\end{equation*}
As $\<F'(x\opt), x\opt - x_k\> \le 0$, we obtain the theorem.

\subsection{Convergence of \aProx methods}

The key consequence of Theorem~\ref{theorem:prox-point-bounded} is that the
iterates $x_k$ are stable in probability~\eqref{eqn:bounded-distance},
remaining bounded with probability 1. Iterate boundedness of algorithms does
not guarantee convergence in general; however, in this section, we show that
as a consequence of this boundedness, any algorithm satisfying
Conditions~\ref{cond:convex-model}--\ref{cond:lower-model} is convergent.
Assumptions~\ref{assumption:convex}
and~\ref{assumption:very-weak-moment} are insufficient to guarantee
convergence of subgradient methods, which---as our examples show---may
diverge without uniform boundedness conditions on the subgradients.  In
contrast, any \aProx method that is stable in
probability~\eqref{eqn:bounded-distance} guarantees convergence with
probability 1.  Throughout this section and
Section~\ref{sec:asymptotic-normality}, we without comment make the
assumption that the stepsizes $\stepsize_k$ satisfy the standard summability
conditions
\begin{equation*}
  \stepsize_k > 0 ~ \mbox{for~all~} k,
  ~~~
  \sum_{k=1}^\infty \stepsize_k = \infty,
  ~~~ \mbox{and} ~~~
  \sum_{k=1}^\infty \stepsize_k^2 < \infty.
\end{equation*}
\begin{proposition}
  \label{proposition:convergence-from-boundedness}
  Let Assumptions~\ref{assumption:convex} and
  \ref{assumption:very-weak-moment} hold. Let the iterates $x_k$ be
  generated by any method satisfying
  Conditions~\ref{cond:convex-model}--\ref{cond:lower-model}
  and $F\opt = \inf_{x \in \mc{X}} F(x)$.
  On the event that
  $\sup_k \dist(x_k, \mc{X}\opt) < \infty$, with probability
  one both $\sum_k \stepsize_k (F(x_k) - F\opt) < \infty$,
  and the iterates $x_k$ are convergent: there exists
  $x\opt \in \mc{X}\opt$ such that
  $\norm{x_k - x\opt} \cas 0$ and
  $F(x_k) \cas F(x\opt)$.
\end{proposition}

Proposition~\ref{proposition:convergence-from-boundedness} implies an
asymptotic convergence rate on (weighted averages of) the iterates $x_k$.
Indeed, let $\{\gamma_k\}_{k = 1}^\infty \subset \R_+$ be a non-decreasing
sequence with $\gamma_k > 0$, and $\gamma_k \to \infty$. Defining
the weighted averages $\wb{x}_k
= \sum_{i = 1}^k \gamma_i \stepsize_i x_i / (\sum_{i = 1}^k \gamma_i
\stepsize_i)$, we have
\begin{corollary}
  \label{corollary:kind-of-convergence-rate}
  Let the conditions of
  Proposition~\ref{proposition:convergence-from-boundedness} hold. Then
  with probability 1,
  \begin{equation*}
    \lim_{k \to \infty}
    \frac{1}{\gamma_k}
    \bigg(\sum_{i = 1}^k \gamma_i \stepsize_i\bigg)
    \left[F(\wb{x}_k) - F\opt\right] = 0.
  \end{equation*}
\end{corollary}
\begin{proof}
  We have $(\sum_{i = 1}^k \gamma_i \stepsize_i) (F(\wb{x}_k) - F\opt) \le
  \sum_{i = 1}^k \gamma_i \stepsize_i (F(x_i) - F\opt)$ by Jensen's
  inequality. Kronecker's lemma gives the result.
\end{proof}
\noindent
For example, taking $\gamma_k = \stepsize_k^{-1}$, we obtain
that the average $\wb{x}_k = \frac{1}{k} \sum_{i=1}^k x_i$ satisfies
\begin{equation*}
  k \stepsize_k \left(F(\wb{x}_k) - F\opt\right) \cas 0.
\end{equation*}

To prove Proposition~\ref{proposition:convergence-from-boundedness}, we
present a lemma giving a one-step progress guarantee for any method
satisfying Conditions~\ref{cond:convex-model}--\ref{cond:lower-model}.
\begin{lemma}
  \label{lemma:single-recentered-progress}
  Let Conditions~\ref{cond:convex-model}--\ref{cond:lower-model} hold and
  let $x_k$ be generated by the updates~\eqref{eqn:model-iteration}.
  Then for any $x \in \mc{X}$,
  \begin{equation*}
    \half \ltwo{x_{k + 1} - x}^2
    \le \half \ltwo{x_k - x}^2
    - \stepsize_k [f(x_k; \statrv_k) - f(x; \statrv_k)]
    + \frac{\stepsize_k^2}{2} \ltwo{f'(x_k; \statrv_k)}^2.
  \end{equation*}
\end{lemma}
\begin{proof}
  Using Lemma~\ref{lemma:single-step-progress}, it suffices to
  show that for any $\stepsize > 0$ and $x_0, x_1, x \in \mc{X}$
  \begin{equation*}
    -\stepsize [f_{x_0}(x_1; \statval) - f_{x_0}(x; \statval)]
    - \half \ltwo{x_1 - x_0}^2
    \le -\stepsize [f(x_0; \statval) - f(x; \statval)]
    + \frac{\stepsize^2}{2} \ltwo{f'(x_0; \statval)}^2.
  \end{equation*}
  To see this, recall that
  Conditions~\ref{cond:convex-model}--\ref{cond:lower-model} imply the
  containment~\eqref{eqn:subgrad-containment}, which in turn implies
  \begin{align*}
    -f_{x_0}(x_1; \statval) + f_{x_0}(x; \statval)
    & = -[f_{x_0}(x_0; \statval) - f_{x_0}(x; \statval)]
    + f_{x_0}(x_0; \statval) - f_{x_0}(x_1; \statval) \\
    & \le
    -[f_{x_0}(x_0; \statval) - f_{x_0}(x; \statval)]
    + \<f'(x_0;\statval), x_0 - x_1\> \\
    & \stackrel{\ref{cond:lower-model}}{\le}
    -[f(x_0;\statval) - f(x; \statval)]
    + \<f'(x_0;\statval), x_0 - x_1\>.
  \end{align*}
  Then we use that for any vector $v$,
  $\stepsize \<v, \Delta\> - \half \ltwo{\Delta}^2 \le
  \frac{\stepsize^2}{2} \ltwo{v}^2$, which
  gives the result.
\end{proof}

\begin{proof-of-proposition}[\ref{proposition:convergence-from-boundedness}]
  By Assumption~\ref{assumption:very-weak-moment}, $\E[\norm{f'(x;
      \statrv)}^2] \le \growfunc(r)$ for all $x$ such that $\dist(x,
  \mc{X}\opt) \le r$.  Lemma~\ref{lemma:single-recentered-progress} implies
  that for any $x\opt \in\mc{X}\opt$,
  \begin{align}
    \label{eqn:single-dingle}
    \E[\norms{x_{k+1} - x\opt}^2 \mid \mc{F}_{k-1}]
    & \le \norms{x_k - x\opt}^2 - 2 \stepsize_k (F(x_k) - F\opt)
    + \stepsize_k^2
    \growfunc(\dist(x_k, \mc{X}\opt)).
  \end{align}
  On the event that $\sup_k \dist(x_k, \mc{X}\opt) < \infty$, we have
  $\sum_k \stepsize_k^2 \growfunc(\dist(x_k, \mc{X}\opt)) < \infty$, and so
  the Robbins-Siegmund Lemma~\ref{lemma:robbins-siegmund} implies
  for any $x\opt \in \mc{X}\opt$,
  there is some (random) $V(x\opt) < \infty$ such that
  $\norms{x_k - x\opt} \cas V(x\opt)$, and
  $\sum_k \stepsize_k (F(x_k) - F(x\opt)) < \infty$.
  
  We now show that $\dist(x_k, \mc{X}\opt) \cas 0$.
  Letting $x\opt$ be the
  projection of $x_k$ onto $\mc{X}\opt$,
  inequality~\eqref{eqn:single-dingle} gives $\E[\dist(x_{k+1},
    \mc{X}\opt)^2 \mid \mc{F}_{k-1}] \le \dist(x_k, \mc{X}\opt)^2 - 2
  \stepsize_k(F(x_k) - F\opt) + \stepsize_k^2 \growfunc(\dist(x_k,
  \mc{X}\opt))$. Thus using Lemma~\ref{lemma:robbins-siegmund}, there exists
  a random $D_\infty < \infty$ such that $\dist(x_k, \mc{X}\opt) \cas
  D_\infty$.
  For $\epsilon_i \in \R_+$, define the gap function
  \begin{equation*}
    \gapfunc(x\opt, \epsilon_1, \epsilon_2)
    \defeq \inf_{x \in \mc{X}} \left\{F(x) - F\opt
    \mid \epsilon_1 \le \norm{x - x\opt} \le 4 \epsilon_1,
    \dist(x, \mc{X}\opt) \ge \epsilon_2
    \right\}.
  \end{equation*}
  The set $\{x \in \mc{X} \mid \norm{x - x\opt} \in
  [\epsilon_1, 4\epsilon_1], \dist(x, \mc{X}\opt) \ge \epsilon_2\}$ is
  compact, so the infimum in the definition of the gap $\gapfunc$ is
  attained and $\gapfunc(x\opt, \epsilon_1, \epsilon_2) > 0$ for
  $\epsilon_i > 0$.  When the a.s.\ limits
  $V(x\opt)$ and $D_\infty$ are positive,
  there exists a (random) $K <\infty$ such that
  $\norm{x_k - x\opt} \in [V(x\opt)/2, 2 V(x\opt)]$ and
  $\dist(x_k, \mc{X}\opt) \ge D_\infty / 2$ for all
  $k \ge K$. Thus with probability 1,
  \begin{equation*}
    \infty > \sum_k \stepsize_k (F(x_k) - F\opt)
    \ge \sum_{k \ge K} \stepsize_k \gapfunc(x\opt, V(x\opt) / 2, D_\infty / 2).
  \end{equation*}
  As $\sum_k \stepsize_k = \infty$ and
  $\gapfunc(x\opt, \epsilon_1, \epsilon_2) > 0$ for
  $\epsilon_i > 0$, we have $D_\infty = 0$
  with probability 1.

  Finally, we show the sequence $x_k$ converges.
  We begin by noting
  that $V(x\opt) = \lim_k \norm{x_k - x\opt}$ is $1$-Lipschitz.
  Indeed, let $x_1\opt, x_2\opt \in \mc{X}\opt$. By
  the a.s.\ limit definition of $V$,
  \begin{equation*}
    \left|V(x_1\opt) - V(x_2\opt)\right|
    \le
    \limsup_k \left|\norm{x_k - x_1\opt} - \norm{x_k - x_2\opt}\right|
    \le \norm{x_1\opt - x_2\opt}.
  \end{equation*}
  Let us now show that for some $x\opt \in \mc{X}\opt$, we have $V(x\opt) =
  0$.  Let $\ball = \{x \in \R^n \mid \norm{x} \le 1\}$ be the ball.  As
  $V(x\opt) < \infty$, there exists a (random) $r < \infty$ such that the
  $x_k \in r \ball$ for all $k$. As $\dist(x_k, \mc{X}\opt) \cas 0$, the
  compactness of $\ball$ implies that $r \ball \cap \mc{X}\opt \neq
  \emptyset$.  Let $x\opt \in r\ball \cap \mc{X}\opt$.  For any $x \in
  r\ball$, the projection $\pi(x)$ of $x$ onto $\mc{X}\opt$ satisfies
  $\norm{\pi(x) - x} \le \norm{x\opt - x} \le 2r$, and so $\norm{\pi(x)} \le
  3r$, and defining the set $C = 3r \ball$,
  we have $\dist(x_k, \mc{X}\opt) = \dist(x_k, \mc{X}\opt \cap C)$ for all $k$.
  Now, fix $\epsilon > 0$, and let $\{x_i\opt\}_{i = 1}^N$ be
  an $\epsilon$-net of $C \cap \mc{X}\opt$, where $N < \infty$.
  As $x_k \in C$ for all $k$,
  \begin{equation*}
    \min_{i \in [N]}
    \norm{x_k - x_i\opt} - \epsilon
    \le \dist(x_k, \mc{X}\opt \cap C)
    = \dist(x_k, \mc{X}\opt) \cas 0,
  \end{equation*}
  and $\min_{i \in [N]} \norm{x_k - x_i\opt} \cas \min_{i \in [N]}
  V(x_i\opt)$. Thus for any $\epsilon > 0$,
  there exists $x_\epsilon \in \mc{X}\opt \cap C$ such that
  $V(x_\epsilon) \le \epsilon$, so that
  $\inf_{x \in C \cap \mc{X}\opt} V(x) = 0$. By the continuity of $V$,
  the infimum is attained and there is
  $x\opt$ such that $V(x\opt) = 0$. The sequence $x_k$ thus converges to
  this $x\opt \in \mc{X}\opt$,
  and $F(x_k) \cas F(x\opt) = F\opt$.
\end{proof-of-proposition}

\subsection{Asymptotic normality}
\label{sec:asymptotic-normality}

Without additional conditions, it is challenging to provide more precise
convergence rate guarantees than those of
Corollary~\ref{corollary:kind-of-convergence-rate}, which (at best) provides
an asymptotic rate scaling as $1 / (\stepsize_k k)$ for $\stepsize_k \asymp
k^{-\steppow}$. With this in mind, we introduce an additional assumption of
smoothness and strong convexity in a neighborhood of the optimal point, and
we consider distributional convergence to establish asymptotic normality of
the averaged iterates, extending results of \citet{PolyakJu92}.  In most
analyses of stochastic convex optimization problems yielding asymptotic
normality, typical assumptions are that the random functions $f(\cdot;
\statval)$ have globally Lipschitz gradients
(e.g.~\cite[Sec.~5]{PolyakJu92}, \cite{GhadimiLa13, Lan12, RyuBo14}), which
is reasonable when $\mc{X}$ is compact, but may fail for non-compact
$\mc{X}$.
In contrast, we consider
\begin{assumption}
  \label{assumption:weak-lipschitz-gradient}
  The functions $F$ and $f$ satisfy the following.
  \begin{enumerate}[label=(\roman*),leftmargin=*]
  \item \label{item:local-strong-convexity}
    The function $F$ is $\mc{C}^2$ in a neighborhood of $x\opt =
    \argmin_{x \in \mc{X}} F(x)$, and $\nabla^2 F(x\opt) \succ 0$.
  \item \label{item:neighbor-smooth}
    There exists $\epsilon > 0$ such that $f(\cdot; \statval)$ is
    $\lipgrad(\statval)$-smooth on the set $\mc{X}\opt_\epsilon \defeq \{x
    \mid \norm{x - x\opt} \le \epsilon\}$, meaning that $x \mapsto \nabla
    f(x; \statval)$ is $\lipgrad(\statval)$ Lipschitz on
    $\mc{X}\opt_\epsilon$, and $\E[\lipgrad(\statrv)^2] = \lipgrad^2 <
    \infty$.
  \end{enumerate}
\end{assumption}

Assumption~\ref{assumption:weak-lipschitz-gradient} says that in
a neighborhood of $x\opt$, the random functions $f$ have Lipschitz gradients
with probability 1. We will apply
Assumption~\ref{assumption:weak-lipschitz-gradient} in conjunction with
Assumption~\ref{assumption:very-weak-moment}, which enforces a type of local
Lipschitz continuity of $f$ and $F$.  Typical Lipschitz conditions on
$\nabla f$ imply
Assumption~\ref{assumption:very-weak-moment}: if
$\nabla f(\cdot; \statval)$ is $\lipgrad_r(\statval)$ Lipschitz on
$\mc{X}_r\opt = \{x \in \mc{X} \mid \norm{x - x\opt} \le r\}$, where the
smoothness constant $\lipgrad_r$ is square integrable for finite $r$,
Assumption~\ref{assumption:very-weak-moment} holds whenever
$\E[\norm{\nabla f(x\opt; \statrv)}^2] < \infty$. Indeed,
\begin{equation*}
  \norm{\nabla f(x; \statval)}
  \le \norm{\nabla f(x\opt; \statval)}
  + \norm{\nabla f(x\opt; \statval) - \nabla f(x;\statval)}
  \le \norm{\nabla f(x\opt; \statval)}
  + \lipgrad_r(\statval) \norm{x - x\opt},
\end{equation*}
so we have the moment bound $\growfunc(r) \le 2
\E[\norm{\nabla f(x\opt; \statrv)}^2] + 2 \E[\lipgrad_r(\statrv)^2] r^2$.
To further motivate Assumption~\ref{assumption:weak-lipschitz-gradient},
we provide a brief example.

\begin{example}[Poisson regression]
  \label{example:poisson-regression}
  In problems with count data $b_1, b_2, \ldots, b_m \in \N$, we may wish to
  predict counts based on a covariate vector $a_i \in \R^n$.  A standard
  model is poisson regression, a generalized linear
  model~\cite{McCullaghNe89, HastieTiFr09}, where we model $b \in \N$
  conditional on $a \in \R^n$ as coming from a poisson distribution
  with parameter $\lambda = e^{\<a, x\>}$ so $p(b \mid a, x) = e^{-\lambda}
  \lambda^b / b!$. The negative log likelihood is $f(x; (a, b)) = -\log
  p(b \mid a, x) = \log (b!) + \exp(\<a, x\>) - b \<a, x\>$, and
  it is easy to compute the truncated model~\eqref{eqn:trunc-model}, as $f$
  satisfies
  \begin{equation*}
    \inf_z f(x; (a, b)) = \log(b!) + \inf_t \{e^t - bt\}
    = \log(b!) + b - b \log b.
  \end{equation*}
  %
  Because $f'(x; (a, b)) = a e^{\<a, x\>} - b a$,
  using $|e^t - e^s| \le e^{\max\{s,t\}} |s - t|$
  shows that $f$ satisfies Assumption~\ref{assumption:weak-lipschitz-gradient}
  as soon as we have the covariance condition $\cov(a) \succ 0$
  and $\E[e^{r \ltwo{a}}] < \infty$ for $r < \infty$. Such moment
  conditions are not sufficient for SGM
  to converge.
\end{example}

We have the following asymptotic normality
theorem; we present the proof in Appendix~\ref{sec:proof-asymptotic-normality}.
\begin{theorem}
  \label{theorem:always-asymptotic-normality}
  Let
  Assumptions~\ref{assumption:convex}--\ref{assumption:weak-lipschitz-gradient}
  hold.  Let the iterates $x_k$ be generated by any method satisfying
  Conditions~\ref{cond:convex-model}--\ref{cond:lower-model} with
  stepsizes $\stepsize_k = \stepsize_0 k^{-\steppow}$ for some $\steppow \in
  (\half, 1)$ and $\stepsize_0 > 0$.  Assume additionally that the iterates
  are bounded: with probability 1, $\sup_k \norm{x_k} < \infty$.  Then
  \begin{equation*}
    \frac{1}{\sqrt{k}}
    \sum_{i = 1}^k (x_i - x\opt) \cd \normal\big(0,
    \nabla^2 F(x\opt)^{-1} \cov(\nabla f(x\opt; \statrv))
    \nabla^2 F(x\opt)^{-1} \big).
  \end{equation*}
\end{theorem}

To prove Theorem~\ref{theorem:always-asymptotic-normality}, we
use two main insights. The first is
that, if the iterates remain bounded, then
Proposition~\ref{proposition:convergence-from-boundedness} guarantees
convergence. The second is a gradient approximation result that shows that
even if the models in the iterations~\eqref{eqn:model-iteration} are
non-smooth, they locally behave as first-order Taylor approximations to the
functions $f$, and thus eventually the iterates approximate the stochastic
gradient method on quadratics. From this, we can apply the
techniques of \citet{PolyakJu92} to guarantee asymptotic normality.

The asymptotic convergence guarantee in
Theorem~\ref{theorem:always-asymptotic-normality} is unimprovable. It
achieves the local asymptotic minimax bound for stochastic
optimization~\cite{DuchiRu19} (the analogue of the standard Fisher
information in classical statistical problems~\cite[Sec.~8.7]{LeCamYa00,
  VanDerVaart98}).  In contrast to stochastic gradient
schemes, however, Theorem~\ref{theorem:always-asymptotic-normality} requires
essentially \emph{only} that the iterates remain bounded.
In the end, then, the important consequence of the \aProx family is that,
with appropriately accurate models, we can guarantee stability
(Definition~\ref{definition:stability}). By leveraging these stability
guarantees, we can then show that model-based iteration
schemes~\eqref{eqn:model-iteration} are (asymptotically) optimal.


\section{Fast convergence for easy problems}
\label{sec:aProx-adv-fast-conv}

The stability and asymptotic results in Section~\ref{sec:stability} provide
some evidence for the benefits of using better models within stochastic
optimization problems: if the models are accurate enough that the iterates
remain bounded, then we obtain asymptotic optimality results
even when standard gradient methods diverge.
In this section, we study a different collection of problems, which
we term \emph{easy} optimization problems.
More precisely, we say that a stochastic minimization problem is easy if
there are \emph{shared} global minimizers:
\begin{definition}
  \label{definition:easy-problems}
  Let $F(x) \defeq \E_P[f(x; \statrv)]$. Then $F$ is \emph{easy
    to optimize} if for each $x\opt \in \mc{X}\opt
  \defeq \argmin_{x \in \mc{X}} F(x)$
  and $P$-almost all $\statval \in \statdomain$ we have
  \begin{equation*}
    \inf_{x \in \mc{X}} f(x; \statval) = f(x\opt; \statval).
  \end{equation*}
\end{definition}

Definition~\ref{definition:easy-problems} places strong restrictions on the
class of functions we consider, but a number of examples satisfy its
conditions.
Other researchers have considered similar
conditions to Definition~\ref{definition:easy-problems}; for example,
\citet{SchmidtLe13} study stochastic optimization problems of the form $F(x)
= \frac{1}{m} \sum_{i = 1}^m f_i(x)$ where $\nabla f_i(x\opt) = 0$ for all
$i$.
\iftoggle{easyexamples}{%
  We briefly enumerate a few examples to motivate what follows,
  detailing them more carefully in the long version~\cite[Sec.~4.3]{AsiDu18}.    
}{%
  We briefly enumerate a few examples to motivate what follows,
  returning in Section~\ref{sec:easy-examples} to flesh them out fully.
} 

\begin{example}[Overdetermined linear systems]
  \label{example:overdetermined-linear}
  In an overdetermined linear system, we have a matrix $A \in \R^{m \times n}$,
  with $m \ge n$, and we wish to solve $Ax = b$, where we assume the system
  of equalities is feasible. Letting $a_i \in \R^n$
  denote the rows of $A$, both objectives $F(x) = \frac{1}{2m} \ltwo{Ax -
    b}^2$ and $F(x) = \frac{1}{m} \lone{Ax - b}$ satisfy
  Definition~\ref{definition:easy-problems}, where we take samples $\statval
  = (a_i, b_i) \in \R^n \times \R$, and any solution
  $x\opt$ to $Ax = b$ satisfies $f(x\opt; (a_i, b_i)) = 0$ for all $i$. 
  \iftoggle{easyexamples}{%
  }{%
    See Section~\ref{sec:kaczmarz}.
  } 
\end{example}

\begin{example}[Finding a point in the intersection of convex sets]
  \label{example:convex-intersection}
  Let $C_1, C_2, \ldots, C_m$ be closed convex sets with
  non-empty intersection
  $\mc{X}\opt \defeq \cap_{i = 1}^m C_i$.  Then the objective
  $F(x) = \frac{1}{m} \sum_{i = 1}^m \dist(x, C_i)$ is convex, and treating
  the sample space $\statdomain = \{1, \ldots, m\}$, we have $F(x) = 0$ and
  $f(x; i) \defeq \dist(x, C_i) = 0$ for all $i$ if and only if $x \in
  \mc{X}\opt$. 
  \iftoggle{easyexamples}{%
  }{%
    See Section~\ref{sec:random-projections}.
  } 
\end{example}

\begin{example}[Data interpolation]
  \label{example:interpolation}
  A more involved example arises out of recent results in statistical
  machine learning. In this area, substantial recent success in deep
  learning~\cite{LeCunBeHi15, ZhangBeHaReVi17} arises out of models that fit
  a training sample of data \emph{perfectly}. In settings more amenable to
  analysis, \citet{BelkinHsMi18, BelkinRaTs18} study statistical algorithms
  that minimize convex losses and perfectly interpolate the data, that is,
  given a sample $\{\statrv_1, \ldots, \statrv_m\}$ drawn i.i.d.\ from an
  underlying population, they find points $x\opt$ satisfying $f(x\opt;
  \statrv_i) = \inf_x f(x; \statrv_i)$ for each $i = 1, \ldots, m$.
  \iftoggle{easyexamples}{%
  }{%
    See Section~\ref{sec:data-interpolation}.
  } 
\end{example}

Given Examples~\ref{example:overdetermined-linear},
\ref{example:convex-intersection}, and \ref{example:interpolation}, it is of
interest to investigate the \aProx family
for problems satisfying
Definition~\ref{definition:easy-problems}.  In this case, we show that any
model satisfying the local approximation
conditions~\ref{cond:convex-model}--\ref{cond:lower-model} and the
additional lower bound condition~\ref{cond:lower-by-optimal} possesses
strong adaptivity and convergence properties. This is in contrast to
subgradient methods, which (given a precise stepsize choice) can
exhibit fast convergence, but are typically non-adaptive.

To highlight the types of models we consider, without loss of generality, we
may assume that $\inf_{x \in \mc{X}} f(x; \statval) = 0$, as given an oracle
that provides the value $\inf f(\cdot; \statval)$ we can replace $f$
with $f(\cdot; \statval) - \inf f(\cdot; \statval)$. As we
discuss in Section~\ref{sec:aProx} and Example~\ref{example:poisson-regression},
it is frequently easy to compute the infimum $\inf_z f(z; \statval)$ for
an individual sample $\statval$.
The
results in this section thus apply to the lower-truncated
model~\eqref{eqn:trunc-model},
\begin{equation*}
  f_x(y; \statval) \defeq \hinge{f(x; \statval) + \<f'(x; \statval),
    y - x\>}.
\end{equation*}
The updates for this model are easy to compute when $\mc{X} = \R^n$;
indeed, in this case, the guarded model~\eqref{eqn:trunc-model}
yields the update
\begin{equation*}
  x_{k + 1} = x_k - \min\left\{\stepsize_k, \frac{f(x_k;\statrv_k)}{
    \ltwo{f'(x_k; \statrv_k)}^2}\right\} f'(x_k ; \statrv_k).
\end{equation*}
This update is reminiscent of the classical Polyak subgradient
method~\cite{Polyak87}, which chooses ``optimal'' stepsizes in the
subgradient method when the value $F(x\opt)$ is known.

In the remainder of this section, we analyze the performance of the \aProx
family of models on easy problems, showing that in a number of settings,
these methods even enjoy linear convergence.
The starting point of each of our results is the following lemma,
whose proof we defer to Section~\ref{sec:proof-shared-min-progress},
that shows that if a problem has shared minimizers as
in Definition~\ref{definition:easy-problems}, iterates of
any method satisfying the lower bound condition~\ref{cond:lower-by-optimal}
are guaranteed to make progress toward the optimal set.
\begin{lemma}
  \label{lemma:shared-min-progress}
  Let $F$ be easy to optimize
  (Definition~\ref{definition:easy-problems}). Let $x_k$ be generated by the
  updates~\eqref{eqn:model-iteration} using a model satisfying
  Conditions~\ref{cond:convex-model}--\ref{cond:lower-by-optimal}. Then for
  any $x\opt \in \mc{X}\opt$,
  \begin{equation*}
    \half \ltwo{x_{k+1} - x\opt}^2
    \le \half \ltwo{x_k - x\opt}^2
    - \half [f(x_k; \statrv_k) - f(x\opt; \statrv_k)]\min\left\{
    \stepsize_k,
    \frac{f(x_k; \statrv_k) - f(x\opt; \statrv_k)}{
      \ltwo{f'(x_k; \statrv_k)}^2}\right\}.
  \end{equation*}
\end{lemma}
\noindent
In the next two sections, we use this lemma to derive conditions on the
growth of $F$ that we can leverage for fast convergence of \aProx models.
\iftoggle{easyexamples}{%
}{%
  We return to our examples in Section~\ref{sec:easy-examples}, demonstrating
  that common problems satisfy the assumptions in
  Sections~\ref{sec:sharp-growth} and~\ref{sec:strong-growth}.
}

\subsection{Sharp growth with shared minimizers}
\label{sec:sharp-growth}

For our first set of problems, we consider objectives that exhibit
\emph{sharp growth} away from the optimal set $\mc{X}\opt$; classical and
recent optimization literature highlights the importance of such conditions
for the convergence of deterministic optimization
methods~\cite{BurkeFe95,DrusvyatskiyLe18}.
As we shall see, these conditions are sufficient to guarantee
linear convergence of \aProx models
in stochastic settings.
We begin with the following assumption, which we use for
its ease of applicability following the structure
of the progress guarantee in Lemma~\ref{lemma:shared-min-progress}.



\begin{assumption}[Expected sharp growth]
  \label{assumption:expected-sharp-growth}
  There exist constants $\lambda_0, \lambda_1 > 0$ such that
  for all $\stepsize \in \R_+$ and $x \in \mc{X}$
  and $x\opt \in \mc{X}\opt$,
  \begin{equation*}
    \E\left[\min\left\{\stepsize [f(x;\statrv) - f(x\opt; \statrv)],
      \frac{(f(x;\statrv) - f(x\opt; \statrv))^2}{\ltwo{f'(x; \statrv)}^2}
      \right\}\right]
    \ge \dist(x, \mc{X}\opt) \min\left\{\lambda_0 \stepsize,
    \lambda_1 \dist(x, \mc{X}\opt)\right\}.
  \end{equation*}
\end{assumption}
\noindent
While Assumption~\ref{assumption:expected-sharp-growth} is somewhat
complex, a quick calculation shows that
a for it to hold, it is sufficient that
there exist constants $\lambda > 0$ and $p > 0$ such that
\begin{equation*}
  \P\big(f(x;\statrv) - f(x\opt; \statrv) \ge \lambda
  \dist(x, \mc{X}\opt)\big) \ge p
\end{equation*}
and $\E[\ltwo{f'(x;\statrv)}^2] \le \lipobj^2$ for $x \in \mc{X}$. That
is, $f(x; \statrv) \ge f(x\opt; \statrv) + \lambda \dist(x, \mc{X}\opt)$ with
non-zero probability, which is
reasonably easy to check (for example, using the Paley-Zygmund inequality
and Mendelson's small-ball conditions~\cite{Mendelson14}).


Under these conditions, we have the following fast convergence guarantee for
the \aProx family of methods; we provide the proof in
Section~\ref{sec:proof-sharp-growth-convex-convergence}.
\begin{proposition}
  \label{proposition:sharp-growth-convex-convergence}
  Let $F$ be easy to optimize,
  Assumption~\ref{assumption:expected-sharp-growth} hold, and $x_k$ be
  generated by the stochastic iteration~\eqref{eqn:model-iteration} using
  any model satisfying
  Conditions~\ref{cond:convex-model}--\ref{cond:lower-by-optimal}, where the
  stepsizes $\stepsize_k = \stepsize_0 k^{-\beta}$ for
  some $\beta \in (-\infty, 1)$. Define
  $K_0 \defeq
  \floor{(\lambda_0 \stepsize_0 / (\lambda_1 \dist(x_1,
      \mc{X}\opt)))^{1/\beta}}$.
  Then
  \begin{equation*}
    \E[\dist(x_{k+1}, \mc{X}\opt)^2]
    \le 
    \begin{cases}
      \exp\left(-\lambda_1 \min\{k, K_0\}
      - \frac{\lambda_0}{\dist(x_1 \mc{X}\opt)} \sum_{i = K_0 + 1}^k \stepsize_i
      \right) \dist(x_1, \mc{X}\opt)^2 & \mbox{if~} \beta \ge 0 \\
      \exp\left(-\lambda_1 \hinge{k - K_0}
      - \frac{\lambda_0}{\dist(x_1, \mc{X}\opt)}
      \sum_{i = 1}^{k \wedge K_0} \stepsize_i \right)
      \dist(x_1, \mc{X}\opt)^2 & \mbox{if}~ \beta < 0 \end{cases}
  \end{equation*}
  and with probability
  $1$, we have the linear convergence
  \begin{equation*}
    \limsup_{k \to \infty}
    \frac{\dist(x_k, \mc{X}\opt)^2}{(1 - \lambda_1)^k}
    < \infty.
  \end{equation*}
\end{proposition}
\noindent
Without careful stepsize choices, even
non-stochastic subgradient methods do not achieve such convergence
guarantees. Indeed, consider the simple objective $F(x) = |x|$, which
certainly satisfies the sharp growth conditions, and apply the subgradient
method $x_{k+1} = x_k - \stepsize_k \sign(x_k)$ (where we treat $\sign(0) =
+1$). Then the convergence can be
no faster than $O(\stepsize_k)$; this is the typical jamming behavior of
subgradient methods. In contrast,
Proposition~\ref{proposition:sharp-growth-convex-convergence} shows that
by leveraging the knowledge that $\inf_x F(x) = 0$, we achieve linear
convergence.

\subsection{Quadratic growth with shared minimizers}
\label{sec:strong-growth}

As an alternative to sharp growth conditions, we also consider optimization
problems that exhibit quadratic growth---strong convexity---away from their
optima. As is the case for the sharpness conditions in
Section~\ref{sec:sharp-growth}, strong convexity conditions play an
important role in the analysis and implementation of methods for convex
optimization~\cite{HiriartUrrutyLe93ab, Polyak87, Nesterov04,
  DrusvyatskiyLe18} as well as stochastic optimization
problems~\cite{HazanKa11, GhadimiLa12}. It is thus of interest to develop an
understanding of the behavior of the \aProx family of methods under strong
convexity conditions, so that we make the following assumption (which is
slightly weaker than strong convexity).

\begin{assumption}[Quadratic growth with shared minimizers]
  \label{assumption:shared-strong-convexity}
  There exist constants $\lambda_0, \lambda_1 > 0$
  such that for all $x \in \mc{X}$ and $\stepsize > 0$,
  \begin{equation*}
    \E\left[(f(x; \statrv) - f(x\opt;\statrv))
      \min\left\{\stepsize, \frac{f(x; \statrv) - f(x\opt; \statrv)}{
        \ltwo{f'(x; \statrv)}^2}\right\}\right]
    \ge \min\left\{\lambda_0 \stepsize, \lambda_1 \right\}
    \dist(x, \mc{X}\opt)^2.
  \end{equation*}
\end{assumption}

Let us give more intuitive conditions sufficient for
\ref{assumption:shared-strong-convexity} to hold. Suppose the standard
conditions that $\nabla F$ is $L$-Lipschitz and that $F$ has quadratic
growth: $F(x) - F(x\opt) \ge c_0 \dist(x, \mc{X}\opt)^2$.
In addition, assume there exist constants $0 < c, C < \infty$, $p > 0$
such that
\begin{equation*}
  \P\left(\ltwo{\nabla f(x;\statrv)}^2 \le C \ltwo{\nabla F(x)}^2
  ~ \mbox{and} ~
  f(x; \statrv) - f(x\opt; \statrv) \ge c(F(x) - F(x\opt))
  \right) \ge p > 0.
\end{equation*}
The Lipschitz
condition implies $F(x\opt) \le F(y) \le F(x) + \<\nabla F(x), y - x\> +
\frac{L}{2} \ltwo{y - x}^2$, and setting $y = x - \frac{1}{L} \nabla F(x)$
gives $F(x) - F(x\opt) \ge \frac{1}{2L} \ltwos{\nabla F(x)}^2$ or $\frac{2
  L}{\ltwos{\nabla F(x)}^2} \ge \frac{1}{F(x) - F(x\opt)}$. Thus,
\begin{align*}
  \lefteqn{\E\left[(f(x; \statrv) - f(x\opt;\statrv))
      \min\left\{\stepsize, \frac{f(x; \statrv) - f(x\opt; \statrv)}{
        \ltwo{f'(x; \statrv)}^2}\right\}\right]} \\
  & \qquad ~
  \ge c p \left(F(x) - F(x\opt)\right)
  \min\left\{\stepsize, \frac{F(x) - F(x\opt)}{C \ltwo{\nabla F(x)}^2}\right\}
  \ge p c c_0 \min\left\{\stepsize,
  \frac{1}{2 C L} \right\} \dist(x, \mc{X}\opt)^2 .
\end{align*}

Under the quadratic growth
Assumption~\ref{assumption:shared-strong-convexity},
whenever the problem is easy (Definition~\ref{definition:easy-problems})
Lemma~\ref{lemma:shared-min-progress} implies the following
proposition, which gives nearly linear convergence of the \aProx family
whenever the lower bound condition~\ref{cond:lower-by-optimal} holds.
\begin{proposition}
  \label{proposition:easy-strong-convex-convergence}	
  Let Assumption~\ref{assumption:shared-strong-convexity} hold and $x_k$ be
  generated by the stochastic iteration~\eqref{eqn:model-iteration} by any
  model satisfying
  Conditions~\ref{cond:convex-model}--\ref{cond:lower-by-optimal}, where
  the stepsizes $\stepsize_k= \stepsize_0 k^{-\beta}$
  for some $\beta \in (-\infty, \infty)$. Define
  $K_0 = \floor{(\lambda_0 \stepsize_0 / \lambda_1)^{1/\beta}}$.
  If $\beta \ge 0$, then
  \begin{equation*}
    \E[\dist(x_{k+1}, \mc{X}\opt)^2]
    \le \exp\left(-\lambda_1 \min\{k, K_0\}
    - \lambda_0 \sum_{i = K_0+1}^k \stepsize_i\right) \dist(x_1, \mc{X}\opt)^2,
  \end{equation*}
  while if $\beta < 0$, then
  \begin{equation*}
    \E[\dist(x_{k+1}, \mc{X}\opt)^2]
    \le \exp\left(-\lambda_1 \hinge{k - K_0}
    - \lambda_0 \sum_{i = 1}^{K_0 \wedge k}
    \stepsize_i\right) \dist(x_1, \mc{X}\opt)^2.
  \end{equation*}
\end{proposition}
\begin{proof}
  Under Assumption~\ref{assumption:shared-strong-convexity}, the distance
  recursion for $D_k \defeq \dist(x_k, \mc{X}\opt)$ in
  Lemma~\ref{lemma:shared-min-progress} then becomes $D_k \le D_{k-1}$ and
  \begin{align*}
    \E[D_{k + 1}^2 \mid \mc{F}_{k-1}]
    & \le \max\{1 - \lambda_0 \stepsize_k,
    1 - \lambda_1 \} D_k^2.
  \end{align*}
  The claim follows by algebraic manipulations and
  that $1 - t \le e^{-t}$ for all $t$.
\end{proof}

Under similar strong convexity assumptions, \citet{SchmidtLe13} and
\citet{MaBaBe18} show that stochastic gradient methods
can achieve linear or near-linear convergence for easy convex optimization
problems. As is typical in the analysis of stochastic gradient methods,
however, this
requires precise stepsize choices that reflect typically
unknown constants, such as global Lipschitz conditions and
the strong convexity parameter $\lambda_0$.  In contrast, the \aProx family
of methods is adaptive to the easiness of the problem---achieving optimal
asymptotic behavior (as in Section~\ref{sec:stability}) while providing
strong finite-sample guarantees and nearly linear convergence
(Proposition~\ref{proposition:easy-strong-convex-convergence}) when problems
satisfy strong growth conditions.

\iftoggle{easyexamples}{%
  
}{%

\subsection{Examples of easy problems with strong growth}
\label{sec:easy-examples}

We now return to the three
examples~\ref{example:overdetermined-linear}--\ref{example:interpolation}
fitting our framework of easy (Definition~\ref{definition:easy-problems})
problems with shared minimizers, exhibiting quantitative growth conditions
for each.

\subsubsection{Overdetermined linear systems and Kaczmarz algorithms}
\label{sec:kaczmarz}

Kaczmarz algorithms~\cite{StrohmerVe09, LeventhalLe10, NeedellWaSr14,
  NeedellTr14} for overdetermined linear systems, as in
Example~\ref{example:overdetermined-linear}, are effective, solving feasible
systems $A x = b$ (where $A \in \R^{m \times n}, m \ge n$) using careful
stochastic gradient steps on the objective $\ltwo{Ax - b}^2$ to achieve fast
convergence.  We consider instead the mean absolute error $F(x) \defeq
\frac{1}{m} \lone{Ax - b}$, which typically satisfies the sharp
growth condition \ref{assumption:expected-sharp-growth}, so that the \aProx
method~\eqref{eqn:model-iteration} using the truncated
model~\eqref{eqn:trunc-model} achieves linear convergence.
Specialized Kaczmarz algorithms typically achieve slightly
better convergence rates~\cite{StrohmerVe09, NeedellWaSr14, NeedellTr14},
but in this case, we have the additional benefit that the \aProx methods are
adaptive: they still converge outside of
linear systems.

We provide conditions sufficient to demonstrate
Assumption~\ref{assumption:expected-sharp-growth}.  Let the vectors $a_i
\in \R^n$ be drawn independently from a distribution with $\ltwo{a_i} \le
\lipobj$, and assume for small $c > 0$ there exists $p_c > 0$ such
that $\P(|\<a_i, v\>| \ge c \ltwo{v}) \ge p_c$ for $v \in \R^n$. Letting
$b_i = \<a_i, x\opt\>$ and $f_i(x) = |\<a_i, x\> - b_i|$, we have the
following lemma;
\iftoggle{SIOPT}{%
  see the long version~\cite{AsiDu18} for a proof.
}{%
  see
  Appendix~\ref{sec:proof-expected-kaczmarz-growth} for its proof.
}
\begin{lemma}
  \label{lemma:expected-kaczmarz-growth}
  Let the preceding conditions on the vectors $a_i$ hold.
  There exists a numerical constant $C < \infty$ such that
  for $c > 0$ and $t \ge 0$, if we define
  \begin{equation*}
    \lambda_0 \defeq c \left(p_c - C \sqrt{\frac{n + t}{m}}\right)
    ~~ \mbox{and} ~~
    \lambda_1 \defeq \frac{c^2}{\lipobj^2}
    \left(p_c - C \sqrt{\frac{n + t}{m}}\right),
  \end{equation*}
  then
  with probability at least $1 - e^{-t}$ over the randomness in the
  $a_i$, simultaneously for all $x$
  \begin{equation*}
    \frac{1}{m} \sum_{i = 1}^m
    \min\left\{\stepsize [f_i(x) - f_i(x\opt)],
    \frac{(f_i(x) - f_i(x\opt))^2}{\ltwo{f'_i(x)}^2}\right\}
    \ge \ltwo{x - x\opt} \min\left\{\lambda_0 \stepsize,
    \lambda_1 \ltwo{x - x\opt}\right\}.
  \end{equation*}
\end{lemma}

That is, with high probability over the choice of $A \in \R^{m \times n}$,
Assumption~\ref{assumption:expected-sharp-growth} holds with parameters
$\lambda_0$ and $\lambda_1$.  As a more concrete example, suppose the
vectors $a_i$ are uniform on $\sqrt{n} \cdot \sphere^{n-1}$, the sphere of
radius $\sqrt{n}$. Then $\P(|\<a_i, v\>| \ge \half \ltwo{v}) \ge \half$, so
that $\lambda_1 \gtrsim 1 / n$ with high probability, and
Proposition~\ref{proposition:sharp-growth-convex-convergence} implies that
$\limsup_k \ltwo{x_k - x\opt}^2 / (1 - C / n)^k < \infty$.  Roughly, then,
$O(1) \cdot n \log \frac{1}{\epsilon}$ iterations of the \aProx
update~\eqref{eqn:model-iteration} are sufficient to achieve
$\epsilon$-accuracy in the solution of $Ax = b$, each of which requires time
$O(n)$, yielding a total operation count of $O(1) \cdot n^2 \log
\frac{1}{\epsilon}$. This is a less precise version of the bound
\citet[Sec.~2.1]{StrohmerVe09} attain for Kaczmarz methods
on well-conditioned problems.

\subsubsection{Finding a point in the intersection of convex sets}
\label{sec:random-projections}

We return now to Example~\ref{example:convex-intersection}, building
connections with randomized projection algorithms~\cite{BauschkeBo96,
  LeventhalLe10, LewisLuMa09}.  Let $C_1, C_2, \ldots, C_m$ be closed convex
sets, where $\mc{X}\opt \defeq \cap_{i = 1}^m C_i$ is non-empty. The
conditioning of the problem of finding a point in this intersection is
related to the ratio $\dist(x, \mc{X}\opt) / \max_i \dist(x, C_i)$
(cf.~\cite{LeventhalLe10,LewisLuMa09}).  As a simple special case, if the
sets $C_i$ are all halfspaces of the form $C_i = \{x \mid \<a_i, x\> \le
b_i\}$, then the Hoffman error bound~\cite{Hoffman52} implies that
$\linfs{\hinge{Ax - b}} \ge c \dist(x, \mc{X}\opt)$ for a constant $c > 0$;
this result also holds in infinite dimensions under a constraint
qualification~\cite{HuWa89}.  Abstracting away the particulars of the
sets $C_i$, consider $F(x) = \frac{1}{m} \sum_{i = 1}^m
\dist(x, C_i)$, and assume the analog of the Hoffman bound that for
some $\lambda > 0$,
\begin{equation*}
  \dist(x, \mc{X}\opt) \ge \max_i \dist(x, C_i)
  \ge \lambda \dist(x, \mc{X}\opt),
\end{equation*}
where the first inequality always holds and  $\lambda$
is a
condition number~\cite{LeventhalLe10}. For example, for
halfspaces $C_i =
\{x \mid \<a_i, x\> \le b_i\}$, $i = 1, 2$,
with $\norm{a_i} = 1$,
then $\lambda = \sqrt{(1 + \<a_1, a_2\>) / 2}$
suffices.

\iftoggle{SIOPT}{%
  To understand the growth properties of $F(x) = \frac{1}{m} \sum_{i = 1}^m
  \dist(x, C_i)$, recall that $\dist(\cdot, C_i)$ is $1$-Lipschitz and
  subdifferentiable~\cite[Ex.~VI.3.3]{HiriartUrrutyLe93ab}, so that
  the component functions
  satisfy $\norm{\partial \dist(x, C_i)} \le 1$.  
}{ To understand
  the growth properties of $F(x) = \frac{1}{m} \sum_{i = 1}^m \dist(x,
  C_i)$, recall~\cite[Ex.~VI.3.3]{HiriartUrrutyLe93ab} that $\nabla \dist(x,
  C_i) = (x - \pi_{C_i}(x)) / \ltwo{\pi_{C_i}(x) - x}$ for $x \not \in C_i$,
  where $\pi_{C_i}(x)$ denotes projection onto $C_i$, and $\partial \dist(x,
  C_i)$ consists of those $v$ in the normal cone to $C_i$ at $x$ with
  $\ltwo{v} \le 1$ if $x \in C_i$, so that the subgradients of the component
  functions $\dist(x, C_i)$ have norm bounded by $1$.  }
Then letting $I$ be
uniform in $\{1, \ldots, m\}$, we obtain for any $\stepsize > 0$ that
\begin{align*}
  \E[\dist(x, C_I) \min\{\stepsize, \dist(x, C_I)\}]
  & = \frac{1}{m} \sum_{i = 1}^m
  \dist(x, C_i) \min\{\stepsize, \dist(x, C_i)\} \\
  & \ge \frac{1}{m} \max_i \dist(x, C_i)
  \min\left\{\stepsize, \max_i \dist(x, C_i)\right\} \\
  & \ge \dist(x, \mc{X}\opt) \min\left\{\frac{\stepsize \lambda}{m},
  \frac{\lambda}{m} \dist(x, \mc{X}\opt) \right\}.
\end{align*}
In particular, Assumption~\ref{assumption:expected-sharp-growth}
holds with constant $\lambda_1 = \lambda/m$,
and thus Proposition~\ref{proposition:sharp-growth-convex-convergence}
implies the convergence
$\limsup_k \dist(x_k, \mc{X}\opt)^2 / (1 - \lambda/m)^k < \infty$.
Roughly $\frac{m}{\lambda} \log \frac{1}{\epsilon}$ steps
are sufficient to find a point $x_k$ that is within distance $\epsilon$
of the set $\mc{X}\opt$ using
the \aProx family.

\subsubsection{Interpolation problems}
\label{sec:data-interpolation}

\newcommand{\hnorm}[1]{\norm{#1}_{\mc{H}}}
\newcommand{\kernel}{\mathsf{K}}

Finally, we consider statistical
machine learning problems in which one interpolates the
data, Example~\ref{example:interpolation}. Belkin and
collaborators~\cite{BelkinHsMi18, BelkinRaTs18, MaBaBe18} suggest that in a
number of statistical machine learning problems, it is possible to achieve
zero error on a training sample while still achieving optimal convergence
rates for the population objective.
To give a simple example, we study
underparameterized least-squares.
Our data comes in $m$ pairs $(a_i, b_i) \in \R^n \times \R$,
$f_i(x) = \half (\<a_i, x\> - b_i)^2$,
where $n > m$. The minimum norm
interpolant
$x\opt \defeq \argmin_{x} \{\ltwo{x} \mid
Ax = b\}$,
or equivalently, the minimizer
\begin{equation*}
  x\opt = \argmin_{x \in \R^n}
  \left\{ \sum_{i = 1}^m \ell(\<a_i, x\> - b_i) \mid x
  \in \mbox{span}\{a_1, \ldots, a_m\} \right\}
\end{equation*}
for any loss $\ell : \R \to \R_+$ uniquely minimized at $0$,
possesses certain statistical
optimality properties while exhibiting strong
empirical prediction performance~\cite[Sec.~5]{BelkinHsMi18, BelkinRaTs18,
  MaBaBe18,ZhangBeHaReVi17}.
In our case, the truncated models~\eqref{eqn:trunc-model} guarantee that the
iterates of the \aProx family lie in the span of the vectors $a_i$.
With our choice $f_i$ above, if we let
$\lipobj \defeq \max_i \ltwo{a_i}$
then
\begin{align*}
  \frac{1}{m} \sum_{i = 1}^m f_i(x)
  \min\left\{\stepsize, \frac{f_i(x)}{\ltwo{\nabla f_i(x)}^2}\right\}
  & \ge \frac{1}{2m} \ltwo{A (x - x\opt)}^2 \min\left\{\stepsize,
  \frac{1}{\lipobj^2}\right\}.
\end{align*}
Let $U \Sigma V^T = A$ be the singular value decomposition of $A$,
so
\begin{equation*}
  \ltwo{A(x - x\opt)}^2
  = \sum_{i = 1}^m \sigma_i(A)^2 \<v_i, x - x\opt\>^2
  \ge \sigma_m(A)^2 \ltwo{VV^T (x - x\opt)}^2.
\end{equation*}
As $x\opt \in \mbox{span}\{a_i\}$, whenever $x \in \mbox{span}\{a_i\}$, we
have $\ltwo{VV^T (x - x\opt)} = \ltwo{x - x\opt}$, and
Assumption~\ref{assumption:shared-strong-convexity} holds for
such $x$: if $I$ is uniform on
$\{1, \ldots, m\}$ then
\begin{equation*}
  \E\left[(f_I(x) - f_I(x\opt)) \min \left\{\stepsize, \frac{f_I(x)
      - f_I(x\opt)}{\ltwos{\nabla f_I(x)}^2}\right\}\right]
  \ge \frac{\sigma_m(A)^2}{m} \min\left\{\stepsize, \lipobj^{-2}\right\}
  \ltwo{x - x\opt}^2.
\end{equation*}
Random matrices $A \in \R^{m \times n}$ with
independent rows typically satisfy an inequality
of the form
$\sigma_m(A) \gtrsim \E[\ltwo{a}^2]^{1/2} (1 - \sqrt{m/n})$
with high probability~\cite{Vershynin12}.
Assuming that $\lipobj^2 \lesssim \E[\ltwo{a}^2]$, we see that in this case,
Assumption~\ref{assumption:shared-strong-convexity} then holds
with constants $\lambda_0 \gtrsim \frac{1}{m} \E[\ltwo{a}^2]$
and
$\lambda_1 \gtrsim \frac{1}{m}$.

Summarizing, under typical scenarios,
Proposition~\ref{proposition:easy-strong-convex-convergence} guarantees
the following
\begin{corollary}
  \label{corollary:interpolation}
  Consider the underdetermined least squares problem, where the matrix
  $A$ has rows with constant norm $\ltwo{a_i} = \sqrt{n}$ and $\sigma_m(A)
  \ge \sqrt{m}$, and let $x_k$ be generated by the
  iteration~\eqref{eqn:model-iteration} with the truncated
  model~\eqref{eqn:trunc-model} and stepsizes $\stepsize_k = \stepsize_0
  k^{-\steppow}$ for some $\steppow > 0$. Then
  there exists a constant $C$ depending on $n, m, \stepsize_0$, and $\steppow$
  such that for all $k \in \N$
  \begin{equation*}
    \E[\ltwo{x_k - x\opt}^2] \le C \max\left\{\exp\left(-\frac{k}{m}\right),
    \exp\left(-\frac{\stepsize_0}{1 - \steppow} k^{1 - \steppow}\right)
    \right\} \ltwo{x_1 - x\opt}^2.
  \end{equation*}
\end{corollary}
\noindent
Corollary~\ref{corollary:interpolation} shows that the iterates exhibit
nearly linear convergence.  With a
more careful stepsize choice, gradient methods achieve linear
convergence~\cite[Section 4]{SchmidtLe13,MaBaBe18},
but these recommended stepsizes cause non-convergence except on the
narrow set of ``easy'' problems considered. In our scenario, however,
\aProx methods achieve these convergence rates
while  enjoying
convergence in other problems.

}


\section{Non asymptotic convergence results}
\label{sec:everything-is-fine}

For our final set of theoretical results, we provide two propositions on the
non-asymptotic convergence of \aProx methods, describing their behavior on
Lipschitz objectives, and then showing that for any functions exhibiting a
type of strong convexity, the stochastic proximal point method achieves
strong non-asymptotic guarantees.

The first result, on convergence for Lipschitzian objectives,
is essentially known~\cite[Thm.~4.1]{DavisDr19}, and it generalizes
the standard results in most treatments of stochastic convex
optimization, where one makes some type of
Lipschitzian assumptions on the random functions $f$
(e.g.~\cite{Zinkevich03, NemirovskiJuLaSh09, Bertsekas11}):
\begin{assumption}
  \label{assumption:lipschitz}
  There exists $\lipobj^2 < \infty$ such that for each $x \in \mc{X}$,
  $\E[\ltwo{f'(x; \statrv)}^2] \le \lipobj^2$.
\end{assumption}
\noindent
Under this Lipschitzian assumption, the following extension of the
well-known results on convergence of the stochastic gradient
method~\cite{Zinkevich03, NemirovskiJuLaSh09} (also generalizing
\citet{Bertsekas11}) holds.
\begin{proposition}
  \label{proposition:convex-lipschitz}
  Let Assumptions~\ref{assumption:convex} and~\ref{assumption:lipschitz}
  hold, and let the iterates $x_k$ be generated by
  algorithm~\eqref{eqn:model-iteration} by any model satisfing
  Conditions~\ref{cond:convex-model}--\ref{cond:lower-model}. Define
  $\wb{x}_k = (\sum_{i = 1}^k \stepsize_i)^{-1} \sum_{i = 1}^k \stepsize_i
  x_i$. Then
  \begin{equation*}
    \E[F(\wb{x}_k)] - F(x\opt)
    \le \frac{1}{2 \sum_{i = 1}^k \stepsize_i} \ltwo{x_0 - x\opt}^2
    + \frac{\lipobj^2}{2 \sum_{i = 1}^k \stepsize_i} \sum_{i = 1}^k \stepsize_i^2.
  \end{equation*}
  If $\mc{X}$ is compact with
  $R \defeq \sup_{x \in \mc{X}} \ltwo{x - x\opt}$, then
  the average $\wb{x}_k \defeq \frac{1}{k} \sum_{i=1}^k x_i$ satisfies
  \begin{equation*}
    \E[F(\wb{x}_k)] - F(x\opt)
    \le \frac{R^2}{2 k \stepsize_k} + \frac{\lipobj^2}{2 k}
    \sum_{i = 1}^k \stepsize_i.
  \end{equation*}
\end{proposition}
\noindent
\iftoggle{SIOPT}{%
  The proof is a more or less standard application of
  Lemma~\ref{lemma:single-recentered-progress} (See Appendix C
  in~\cite{AsiDu18}); the only contribution beyond~\citet{DavisDr19} is the
  claim under compact $\mc{X}$.
}{%
  The proof is a more or less standard application of
  Lemma~\ref{lemma:single-recentered-progress}, so we defer 
  it to Appendix~\ref{sec:proof-convex-lipschitz}; the
  only contribution beyond~\citet{DavisDr19} is the claim
  under compact $\mc{X}$.
} 

We present one final theoretical result, showing how the stochastic proximal
point method (the full model~\eqref{eqn:prox-model}) achieves reasonable
non-asymptotic convergence guarantees even when the functions $f$ may be
non-Lipschitz and non-smooth, as long as they obey a (restricted) strong
convexity condition. The convergence guarantees we present here
are impossible with standard stochastic gradient methods, which can
diverge under the assumptions we consider.

\begin{assumption}[Restricted strong convexity]
  \label{assumption:strong-convexity}
  The functions $f(\cdot;\statval)$
  are strongly convex over $\mc{X}$ with respect to the matrix
  $\Sigma(\statval) \succeq 0$, that is,
  \begin{equation*}
    f(y; \statval) \ge f(x; \statval)
    + \<f'(x; \statval), y - x\> + \half (x - y)^T \Sigma(\statval)
    (x - y)
    ~~ \mbox{for~} x, y \in \mc{X}
  \end{equation*}
  for all $f'(x; \statval) \in \partial f(x; \statval)$.
  The matrix $\Sigma$ satisfies
  $\E[\Sigma(\statrv)] \succeq \lambdamin I_{n \times n}$,
  where $\lambdamin > 0$.
\end{assumption}

We analyze the stochastic proximal point method
under Assumption~\ref{assumption:strong-convexity};
we begin with a technical lemma, which provides a guarantee
on the one-step progress of the method. 
We present the proof of Lemma~\ref{lemma:prox-point-strong-convexity} in Appendix~\ref{sec:proof-prox-point-strong-convexity}.
\begin{lemma}
  \label{lemma:prox-point-strong-convexity}
  Let Assumption~\ref{assumption:strong-convexity} hold and the iterates
  $x_k$ be generated by the stochastic proximal point
  method~\eqref{eqn:sppm}. Define
  \begin{equation*}
    \wb{\Sigma}_k
    \defeq \E\left[\frac{1}{1 + 2 \stepsize_k \lambdamax(\Sigma(\statrv))}
      \Sigma(\statrv)\right].
  \end{equation*}
  Then
  \begin{equation*}
    \half \E\left[\ltwo{x_{k + 1} - x\opt}^2 \mid \mc{F}_{k-1}\right]
    \le \half (x_k - x\opt)^T
    \left(I - \stepsize_k \wb{\Sigma}_k\right)(x_k - x\opt)
    + \stepsize_k^2 \E\left[\ltwo{f'(x\opt; \statrv)}^2\right].
  \end{equation*}
\end{lemma}

Lemma~\ref{lemma:prox-point-strong-convexity} guarantees that the
proximal-point method makes progress whenever the restricted strong
convexity conditions hold, irrespective of smoothness of the objective
functions. First, we have $0 \prec \wb{\Sigma}_0 \preceq \wb{\Sigma}_k$ for
all $k \in \N$, and defining
$\lambda_k \defeq \lambda_{\min}(\wb{\Sigma}_k) > 0$, we have
$\lambda_k \uparrow \lambda_\infty \defeq \lambda_{\min}(\E[\Sigma(\statrv)])$
when $\stepsize_k \downarrow 0$.
Under the conditions of Lemma~\ref{lemma:prox-point-strong-convexity}
and Assumption~\ref{assumption:convex},
we thus have
\begin{equation*}
  \E\left[\ltwo{x_{k + 1} - x\opt}^2 \mid \mc{F}_{k-1} \right]
  \le (1 - \stepsize_k \lambda_k) \ltwo{x_k - x\opt}^2
  + \stepsize_k^2 \sigma^2
  \le (1 - \stepsize_k \lambda_0) \ltwo{x_k - x\opt}^2
  + \stepsize_k^2 \sigma^2.
\end{equation*}
Applying this inequality recursively gives
\begin{equation*}
  \E[\ltwo{x_{k + 1} - x\opt}^2]
  \le \prod_{i = 1}^k (1 - \stepsize_i \lambda_i)
  \ltwo{x_1 - x\opt}^2
  + \sum_{i = 1}^k \stepsize_i^2 \prod_{j = i+1}^k (1 - \stepsize_j \lambda_j)
  \sigma^2,
\end{equation*}
where we note that $\stepsize_j \lambda_j < 1$ for all $j$.  An
inductive argument~\cite{PolyakJu92} 
implies the next proposition.
\begin{proposition}
  \label{proposition:with-strong-convexity}
  Let Assumptions~\ref{assumption:convex}
  and~\ref{assumption:strong-convexity} hold, and let $x_k$ be
  generated by the stochastic proximal point method~\eqref{eqn:sppm}
  with stepsizes $\stepsize_k = \stepsize_0 k^{-\steppow}$ for
  some $\steppow \in (0, 1)$. Then
  for a numerical constant $C < \infty$,
  \begin{equation*}
    \E[\ltwo{x_{k+1} - x\opt}^2]
    \le \exp\left(-\lambda_0 \sum_{i=1}^k \stepsize_i\right)
    \ltwo{x_1 - x\opt}^2
    + C \cdot \frac{\sigma^2}{\lambda_0} \stepsize_k \cdot \log k.
  \end{equation*}
\end{proposition}
\noindent
An asymptotic argument~\cite{PolyakJu92} gives convergence $\E[\ltwo{x_k -
    x\opt}^2] \lesssim \frac{\sigma^2}{\lambda_{\infty}} \stepsize_k$ for
large $k$. For stochastic proximal point methods, if the objective is
strongly convex, choosing $\stepsize_k = C / k$ for a large constant $C$
yields asymptotic convergence bounds on $\ltwo{x_k - x\opt}^2$ of the form
$\frac{\sigma^2}{\lambda_\infty} \frac{1}{k}$.
In comparison to convergence results available for stochastic gradient
methods~\cite{Polyak87, NemirovskiJuLaSh09, BachMo11},
Proposition~\ref{proposition:with-strong-convexity} holds whenever the
subgradient $f'(x\opt; \statrv)$ has finite second moment, and we require
this only at the point $x\opt$; moreover, it holds no matter the
stepsize sequence.



\newcommand{\numTests}{T}
\newcommand{\numIterations}{K}

\newcommand{\dimExp}{n} 
\newcommand{\numSamplesExp}{m} 
\newcommand{\MatExp}{A} 
\newcommand{\matExp}{a} 
\newcommand{\vecExp}{b} 
\newcommand{\noiseMagExp}{\sigma} 
\newcommand{\noiseVecExp}{v} 

\newcommand{\matCompMat}{M} 
\newcommand{\matCompVara}{X} 
\newcommand{\matCompVarb}{Y}
\newcommand{\numRows}{m} 
\newcommand{\numCols}{n} 
\newcommand{\dimVar}{r} 

\newcommand{\varLogReg}{x} 
\newcommand{\vecLogReg}{a} 
\newcommand{\labelLogReg}{b} 

\newcommand{\vecHinge}{a} 
\newcommand{\labelHinge}{\ell} 
\newcommand{\numClassesHinge}{K} 
\newcommand{\flipProp}{p}

\section{Experiments}
\label{sec:exp}

The final component of this paper is an empirical evaluation of
the \aProx methods. Our goal
in the experiments is to evaluate the relative merits of different
approximate models in the iteration~\eqref{eqn:model-iteration},
and accordingly, we consider the four approximations below.

\begin{enumerate}[label=(\roman*),leftmargin=*]
  \setlength{\itemsep}{0pt}
\item Stochastic gradient method (SGM): uses the
  linear model~\eqref{eqn:dumb-linear-model}.
\item Proximal: uses the full
  model~\eqref{eqn:prox-model}.
\item Truncated: uses the lower truncated model~\eqref{eqn:trunc-model}.
\item Bundle: uses the bundle (cutting plane) model~\eqref{eqn:bundling}
  with two lines, that
  is, with $i = 1$.
\end{enumerate}

We present several experiments, each comparing different aspects of the
\aProx models.  Throughout, we consider step size sequences of the form
$\stepsize_k = \stepsize_0 k^{-\steppow}$, where $\steppow \in (1/2, 1)$.
We wish to evaluate the robustness and stability of each of the models
\eqref{eqn:dumb-linear-model}--\eqref{eqn:bundling} for different problems,
investigating both well- and poorly-conditioned instances, as well as
problems satisfying the ``easy'' conditions of
Definition~\ref{definition:easy-problems}.  Let us provide some guidance
toward expected results.  Roughly, for problems with globally Lipschitz
gradients (such as linear regression problems with noise), we expect the
methods to have fairly similar performance---stochastic gradient, proximal
point, truncated, and bundle models are all asymptotically normal with
optimal covariance. However, as problems become either (i) easier---closer
to satisfying Definition~\ref{definition:easy-problems}---or (ii) more
poorly conditioned or harder because of large or unbounded Lipschitz
constants, we expect stochastic gradient methods to (i) converge more slowly
or (ii) be substantially more sensitive to stepsize choices.

Within each of our experiments, we run each model-based
iteration~\eqref{eqn:model-iteration} for $\numIterations$ total iterations
across multiple initial stepsizes $\stepsize_0$. For a fixed accuracy
$\epsilon > 0$, we record the number of steps $k$ required to achieve
$F(x_k) - F(x\opt) \le \epsilon$, reporting these times (where we terminate
each run at iteration $\numIterations$).  We perform $\numTests$ experiments
for each initial stepsize choice, reporting the median of the
time-to-$\epsilon$-accuracy; the shaded areas in each
plot cover the 5th to 95th percentile of convergence times
(90\% coverage sets).
Finally, while
we restrict our experiments to the single-sample batch size setting
(i.e.\ no ``mini-batches''), we note that using multiple samples to
decrease the variance of (sub)gradients does not necessarily
improve the robustness of standard stochastic subgradient methods.  Indeed,
Examples~\ref{example:non-quadratic-divergence}
and~\ref{example:quadratic-instability} show that the
deterministic gradient method can be sensitive to stepsize choice. In
experiments we do not include because of the additional space they require,
we verify this intuition---stochastic gradient methods remain similarly
sensitive
to stepsize choice even using large batch sizes. Additionally,
it is possible to consider the iteration $k$ at which the averaged
iterate $\wb{x}_k$ achieves $F(\wb{x}_k) - F(x\opt) \le \epsilon$; this
does not change the qualitative aspects of our plots in any way.


\subsection{Linear Regression}
\label{sec:LinReg_exp}

\begin{figure}[ht]
  \begin{center}
    \begin{tabular}{cc}
      \begin{overpic}[width=.5\columnwidth]{
          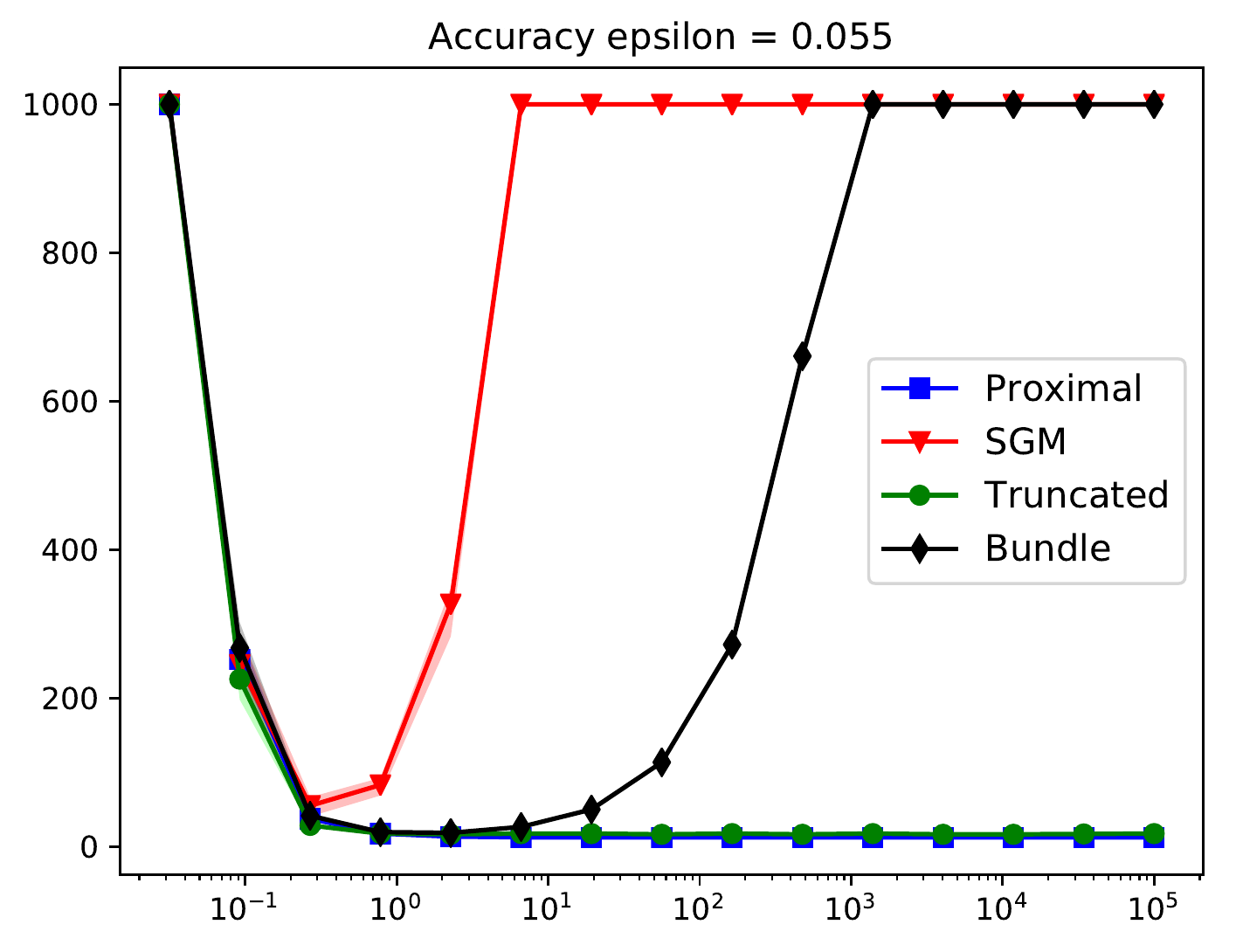}
        \put(35,-1){{\small Initial stepsize $\stepsize_0$}}
        \put(-3,15){
          \rotatebox{90}{{\small Time to accuracy $\epsilon = .05$}}}
        \put(33,72){
          \tikz{\path[draw=white,fill=white] (0, 0) rectangle (3cm,.35cm);}
        }
      \end{overpic} &
      \begin{overpic}[width=.5\columnwidth]{
          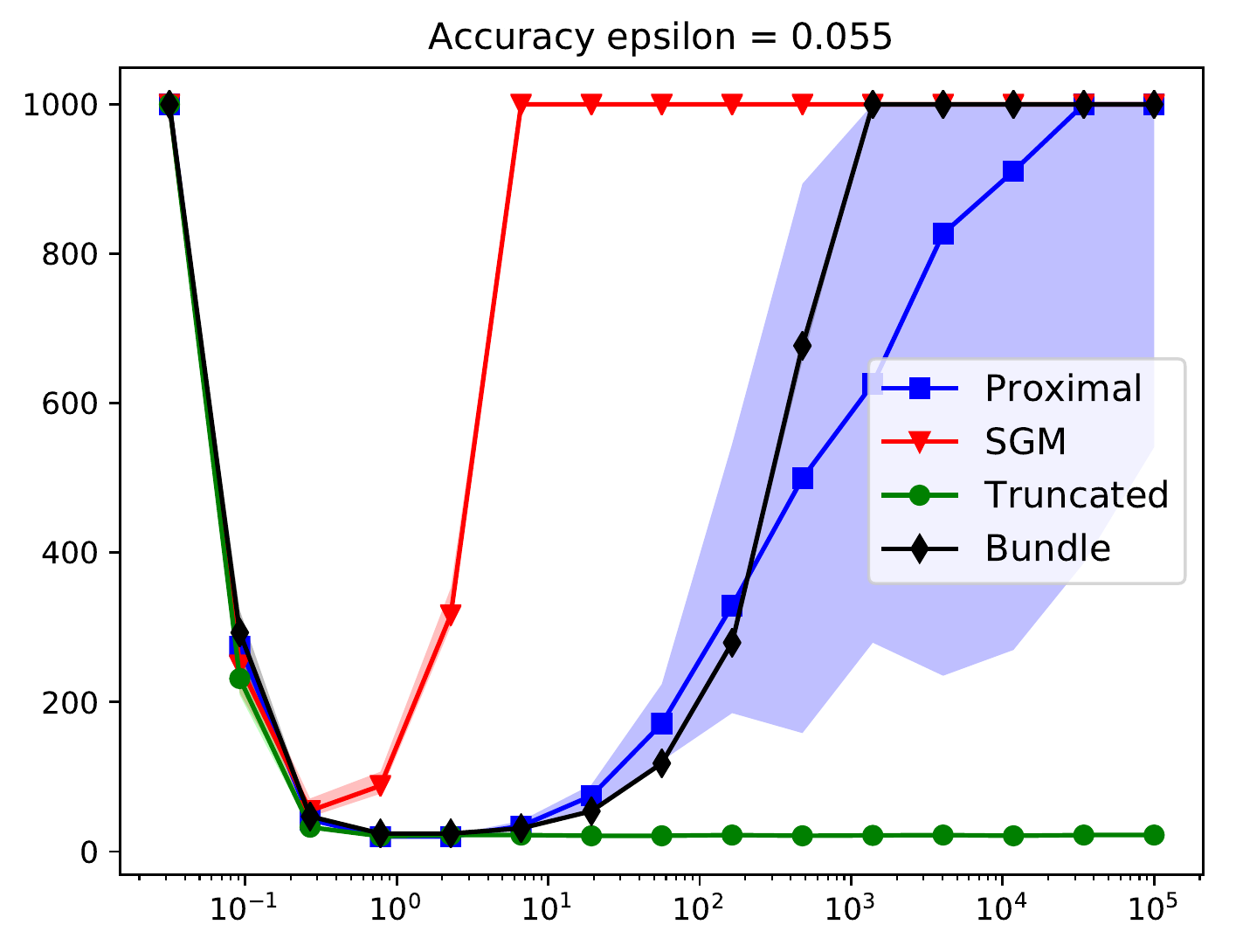}
        \put(35,-1){{\small Initial stepsize $\stepsize_0$}}
        \put(-3,15){
          \rotatebox{90}{{\small Time to accuracy $\epsilon = .05$}}}
        \put(33,72){
          \tikz{\path[draw=white,fill=white] (0, 0) rectangle (3cm,.35cm);}
        }
      \end{overpic} \\
      (a) & (b)
    \end{tabular}
    \caption{\label{fig:linear-regression-well-conditioned}
      The number of iterations to achieve
      $\epsilon$-accuracy versus initial stepsize $\stepsize_0$
      for linear regression with $m = 1000$, $n = 40$, and
      condition number $\kappa(A) = 1$.
      (a) The noiseless setting with $\sigma = 0$.
      (b) Noisy setting with $\sigma = \half$.
    }
  \end{center}
\end{figure}

\begin{figure}[ht]
  \begin{center}
    \begin{tabular}{cc}
      \begin{overpic}[width=.5\columnwidth]{
      		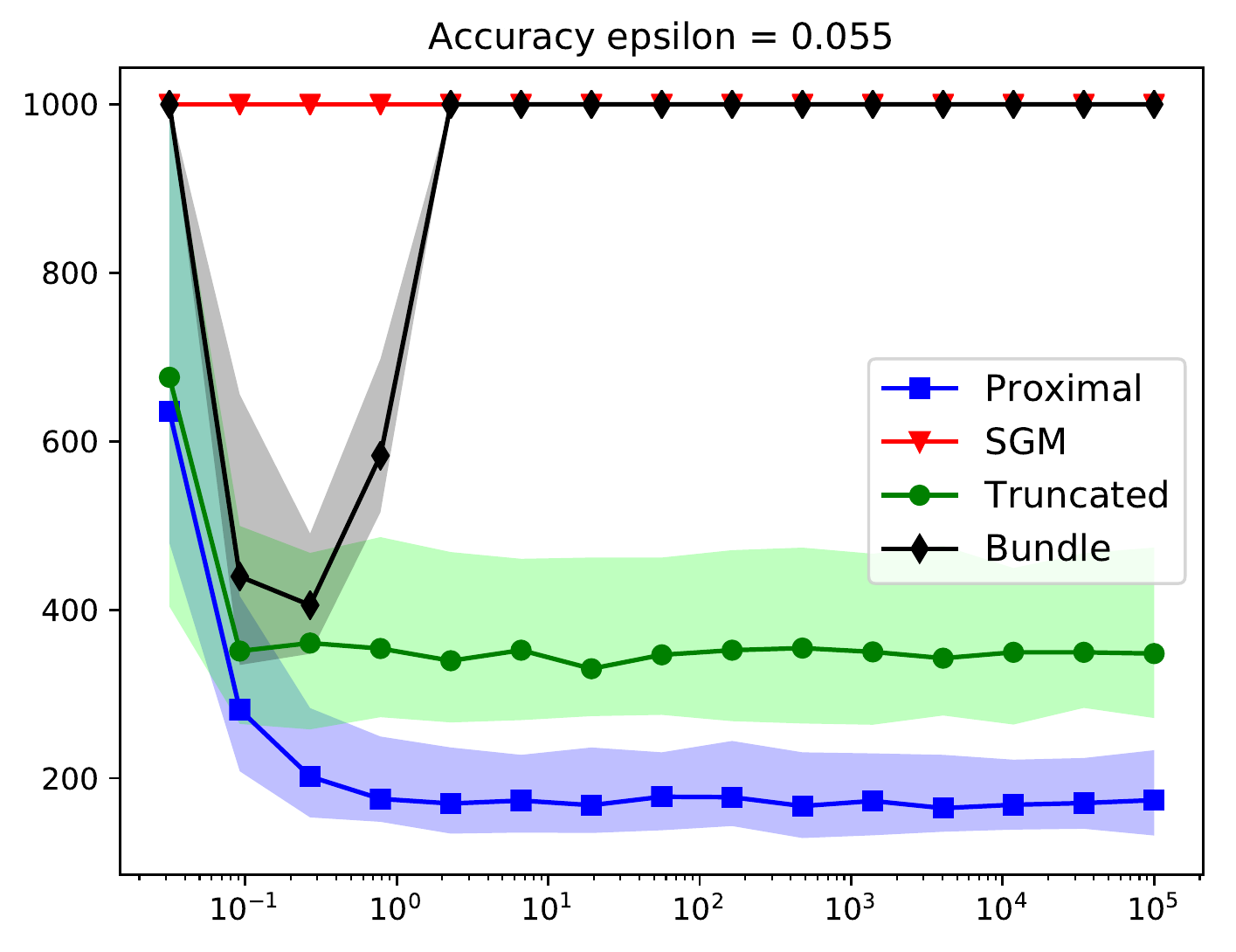}
      	\put(35,-1){{\small Initial stepsize $\stepsize_0$}}
      	\put(-3,15){
      		\rotatebox{90}{{\small Time to accuracy $\epsilon = .05$}}}
      	\put(33,72){
      		\tikz{\path[draw=white,fill=white] (0, 0) rectangle (3cm,.35cm);}
      	}
      \end{overpic} &
      \begin{overpic}[width=.5\columnwidth]{
      		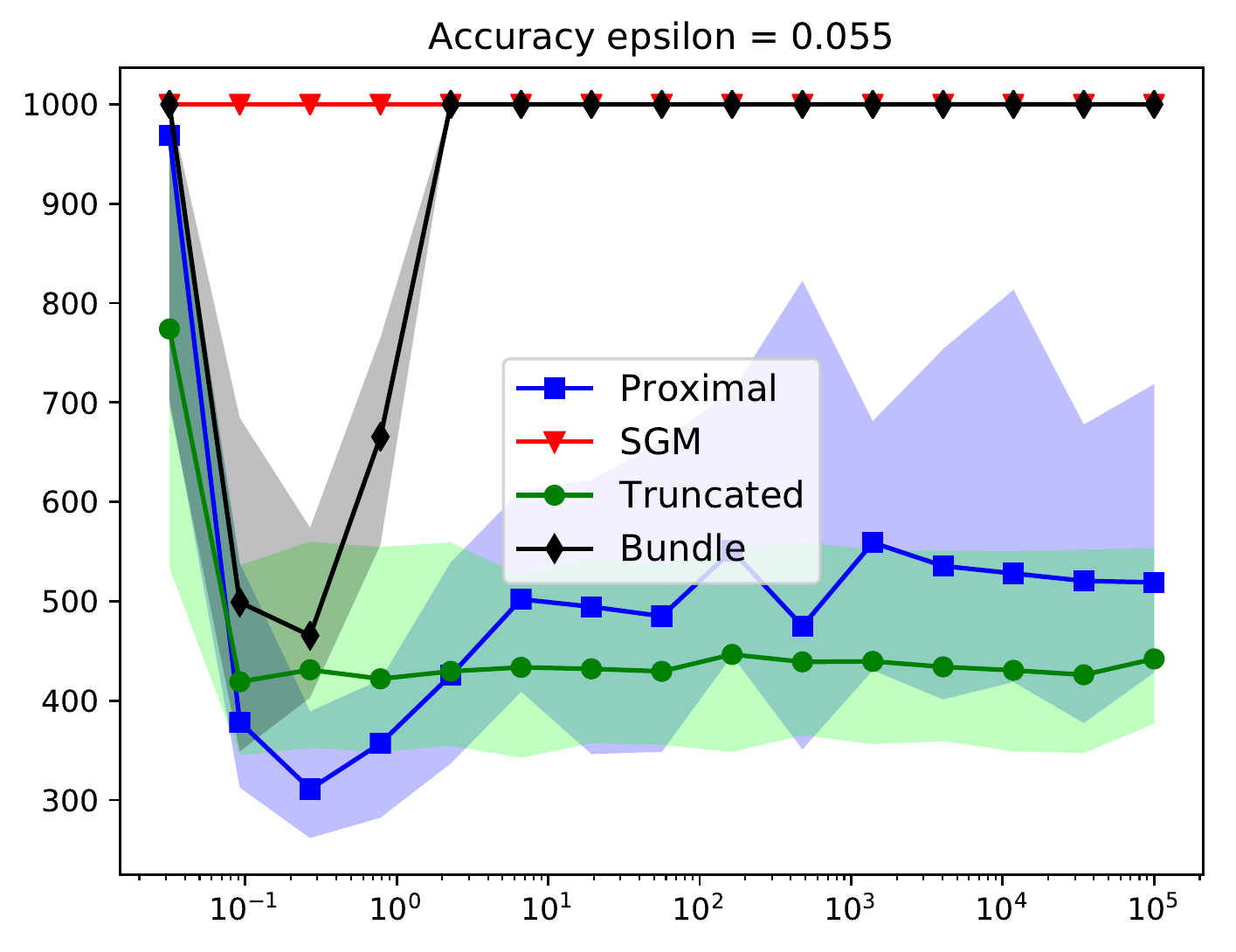}
      	\put(35,-1){{\small Initial stepsize $\stepsize_0$}}
      	\put(-3,15){
      		\rotatebox{90}{{\small Time to accuracy $\epsilon = .05$}}}
      	\put(33,72){
      		\tikz{\path[draw=white,fill=white] (0, 0) rectangle (3cm,.35cm);}
      	}
      \end{overpic} \\
  	  (a) & (b)
    \end{tabular}
  \caption{\label{fig:linear-regression-medium-conditioned}
      The number of iterations to achieve
      $\epsilon$-accuracy versus initial stepsize $\stepsize_0$
      for linear regression with $m = 1000$, $n = 40$, and
      condition number $\kappa(A) = 15$.
      (a) The noiseless setting with $\sigma = 0$.
      (b) Noisy setting with $\sigma = \half$.
  }
  \end{center}
\end{figure}

In our linear regression experiments, we let $A \in \R^{m \times n}$, $b \in
\R^m$, and $F(x) = \frac{1}{2m} \ltwo{Ax - b}^2$, where in each individual
experiment we generate $x\opt \sim \normal(0, I_n) \in \R^n$ and set $b =
Ax\opt + \sigma v$ for $v \sim \normal(0, I_m)$. We choose $\sigma$
differently depending on the experiment, setting $\sigma = 0$ in noiseless
experiments and $\sigma = \half$ otherwise. We generate $A$ as $A = QD$,
where $Q \in \R^{m \times n}$ has uniformly random orthogonal columns, and
$D = \diag(1, 1 + (\kappa-1) / (n - 1), \ldots, \kappa)$ is a diagonal
matrix with linearly spaced entries between $1$ and a desired condition
number $\kappa \ge 1$.

In Figure~\ref{fig:linear-regression-well-conditioned}, we plot the results
of our experiments on well-conditioned problems, which use matrices $A$ with
condition number $\kappa(A) = 1$, while
Figure~\ref{fig:linear-regression-medium-conditioned} shows identical
results except that we use condition number $\kappa(A) = 15$.  Plot (a) of
each figure demonstrates the results for the noiseless setting with $\sigma
= 0$. In Fig.~\ref{fig:linear-regression-well-conditioned}(a), we see the
expected result that the stochastic gradient method has good performance for
a precise range of stepsizes in $[10^{-1}, 1]$, while the better
approximations of the proximal point~\eqref{eqn:prox-model} and the
truncated~\eqref{eqn:trunc-model} models yield better
convergence over a range of stepsizes with six orders of magnitude. The
bundle model~\eqref{eqn:bundling} shows somewhat more robustness than
SGM, but for large stepsizes also exhibits some oscillation.  In the noisy
cases, plots (b) in each figure, the results are similar, except that the
full proximal model is somewhat less robust to stepsize choice; roughly, in
the stochastic proximal point (SPPM) case, the model trusts the
instantaneous function \emph{too} much, and overfits to the noise at each
iteration.

Figure~\ref{fig:linear-regression-medium-conditioned} tells a similar story
to Figure~\ref{fig:linear-regression-well-conditioned}, except that the
stochastic gradient method is essentially not convergent: in no experiment
did it ever achieve accuracy below $\epsilon = .05$ in the noisy or
noiseless settings. This problem is not in any real sense challenging: a
condition number of $\kappa(A) = 15$ is not particularly poorly
conditioned~\cite{TrefethenBa97}, yet stochastic gradient methods exhibit
very poor behavior. These plots suggest that the reliance on stochastic
gradient methods in much of the statistical and machine learning literature
may be misplaced.

\iftoggle{SIOPT}{%

}{%
\subsection{Absolute Loss Regression}
\label{sec:AbsLoss}

\begin{figure}[ht]
  \begin{center}
    \begin{tabular}{cc}
      \includegraphics[width=.5\columnwidth]{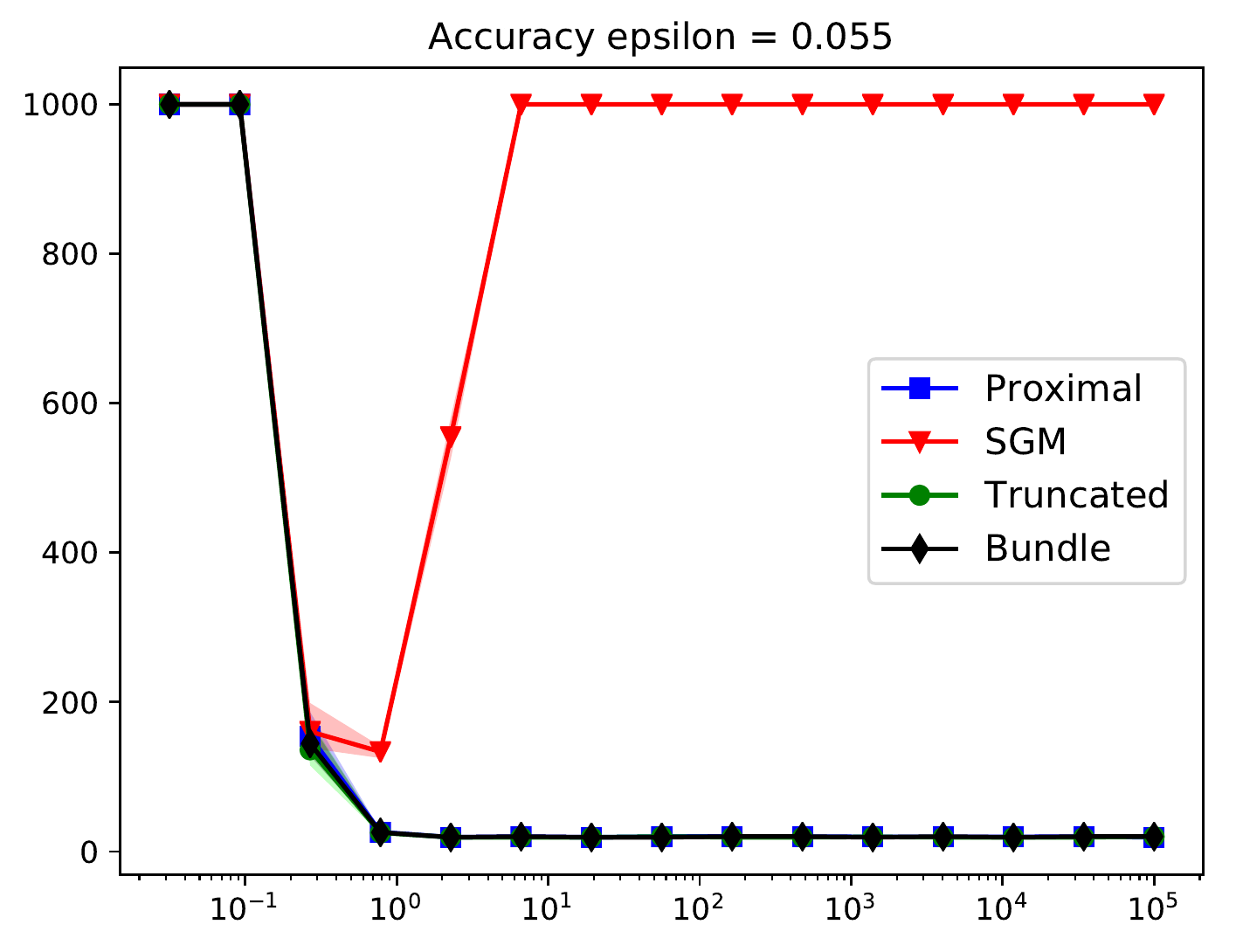}
      &
      \includegraphics[width=.5\columnwidth]{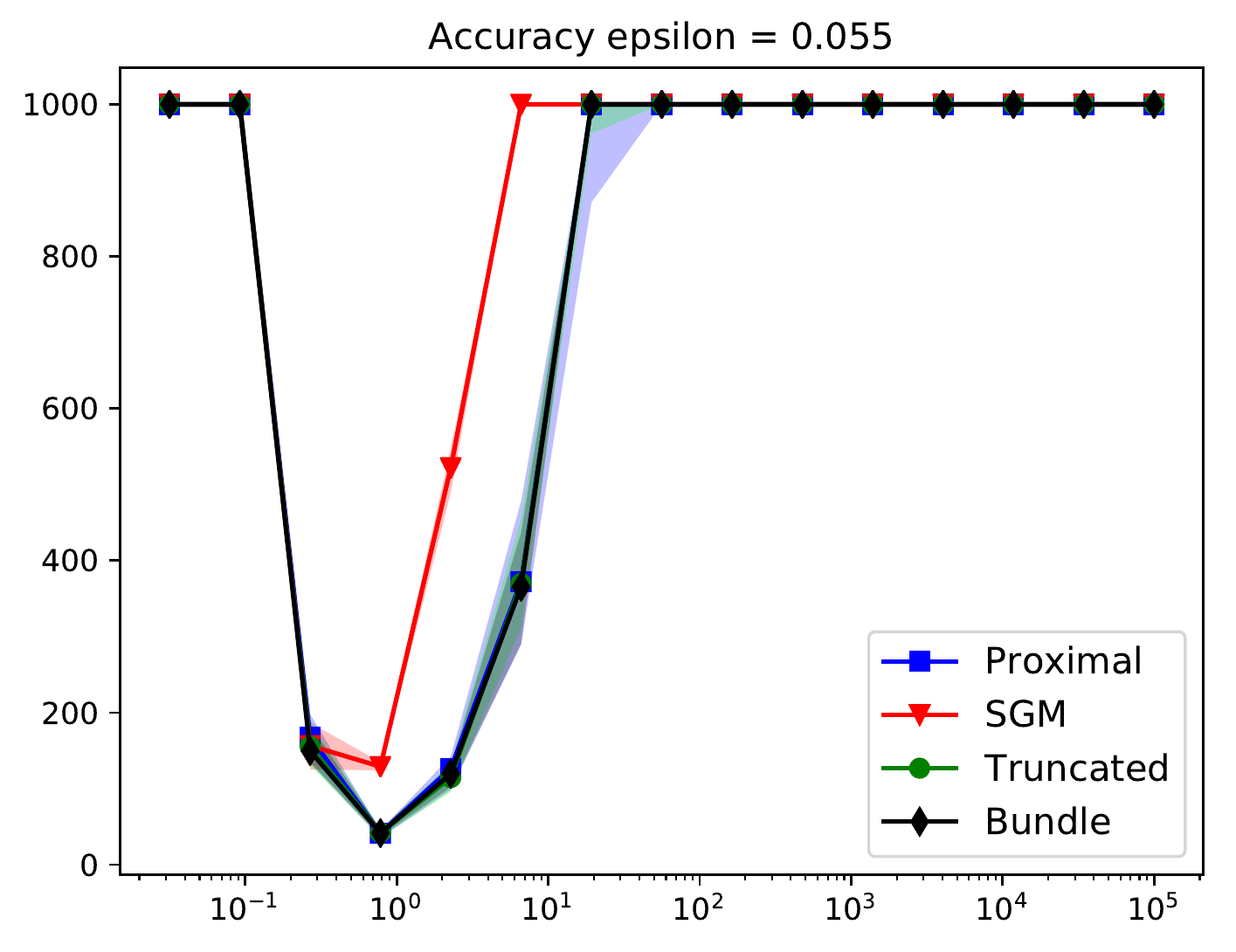}
      \\
      (a) & (b)
    \end{tabular}
    \caption{\label{fig:absolute-regression} The number of iterations to
      achieve $\epsilon$-accuracy as a function of the initial step size
      $\stepsize_0$ for absolute loss regression with $m = 1000$,
      $n = 40$, and condition number $\kappa(A) = 1$. (a) Noiseless setting
      with $\sigma = 0$. (b) Noisy setting with $\sigma = \half$.}
  \end{center}
\end{figure}

For our second set of experiments, we consider regression with an absolute
loss, using $F(x) = \frac{1}{m} \lone{Ax - b}$ and $f(x; (a, b)) = |\<a, x\>
- b|$, which is non-smooth but Lipschitz.  We generate $A$ and $b$
identically to the linear regression experiments in
Sec.~\ref{sec:LinReg_exp}.  Based on our results in
Section~\ref{sec:kaczmarz}, we expect that in the noiseless case, the
proximal-point~\eqref{eqn:prox-model}, truncated~\eqref{eqn:trunc-model},
and multi-line~\eqref{eqn:bundling} methods should have very fast
convergence, which is indeed what we see in
Figure~\ref{fig:absolute-regression}(a). (For this objective,
the multi-line model
functions identically to the stochastic
proximal point method.) When the problems are easy, even when they are
non-smooth, a more careful (even simple truncated) model yields
substantial improvements over naive linear
models~\eqref{eqn:dumb-linear-model}.  In
Figure~\ref{fig:absolute-regression}(b), we display results with noise with
standard deviation $\sigma = \half$; in this case, the better approximations
to $f$ achieve convergence for a slightly larger range of
stepsizes---statistically significant as their confidence regions
exhibit---but the difference is less substantial.
Experiments with
condition numbers $\kappa(A) = 15$ yielded
results completely parallel
to those in Figure~\ref{fig:linear-regression-medium-conditioned}.


}

\subsection{Logistic Regression}
\label{sec:LogReg}

\begin{figure}[ht]
  \begin{center}
    \begin{tabular}{cc}
    \begin{overpic}[width=.5\columnwidth]{
    		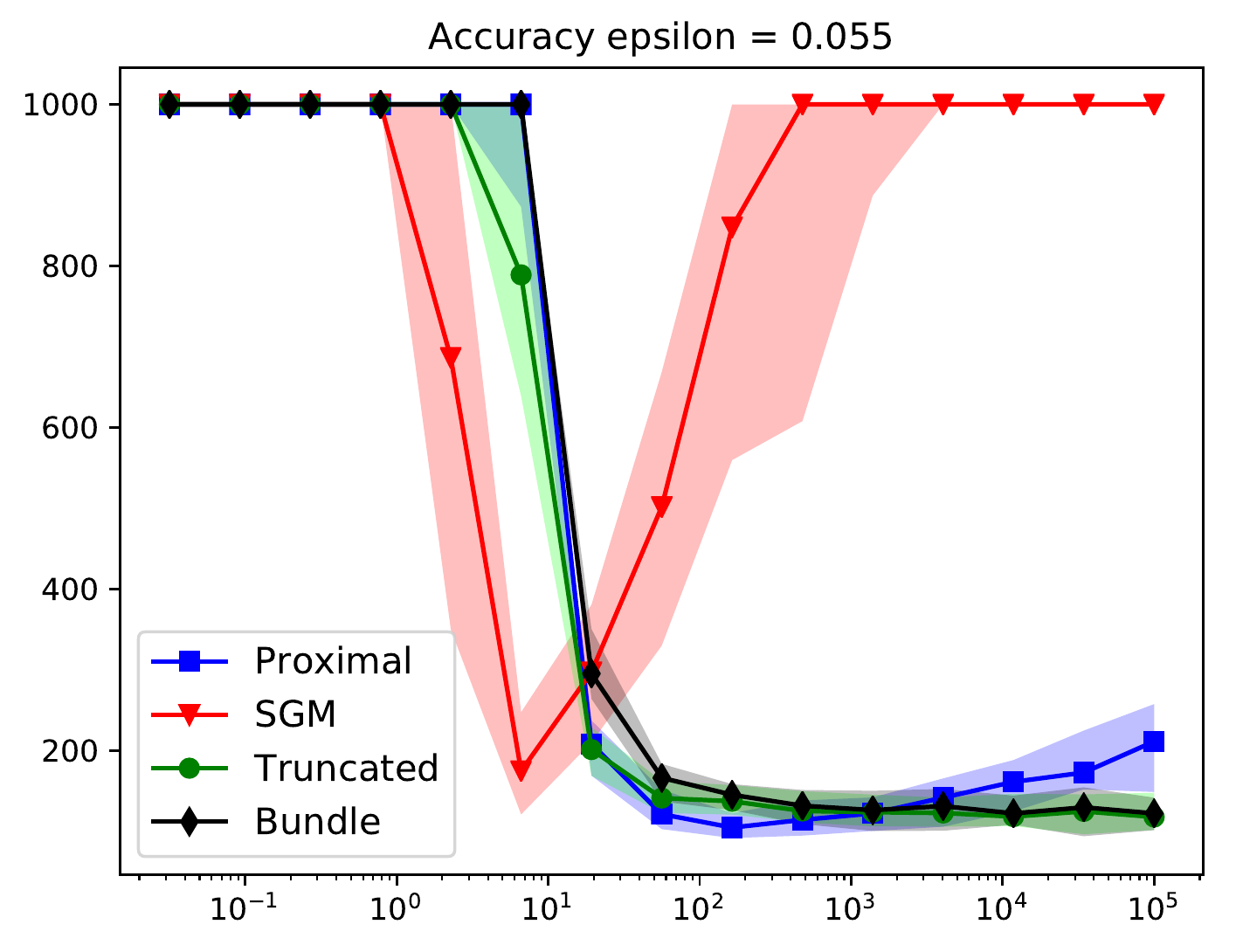}
    	\put(35,-1){{\small Initial stepsize $\stepsize_0$}}
    	\put(-3,15){
    		\rotatebox{90}{{\small Time to accuracy $\epsilon = .05$}}}
    	\put(33,72){
    		\tikz{\path[draw=white,fill=white] (0, 0) rectangle (3cm,.35cm);}
    	}
    \end{overpic} &
    \begin{overpic}[width=.5\columnwidth]{
    		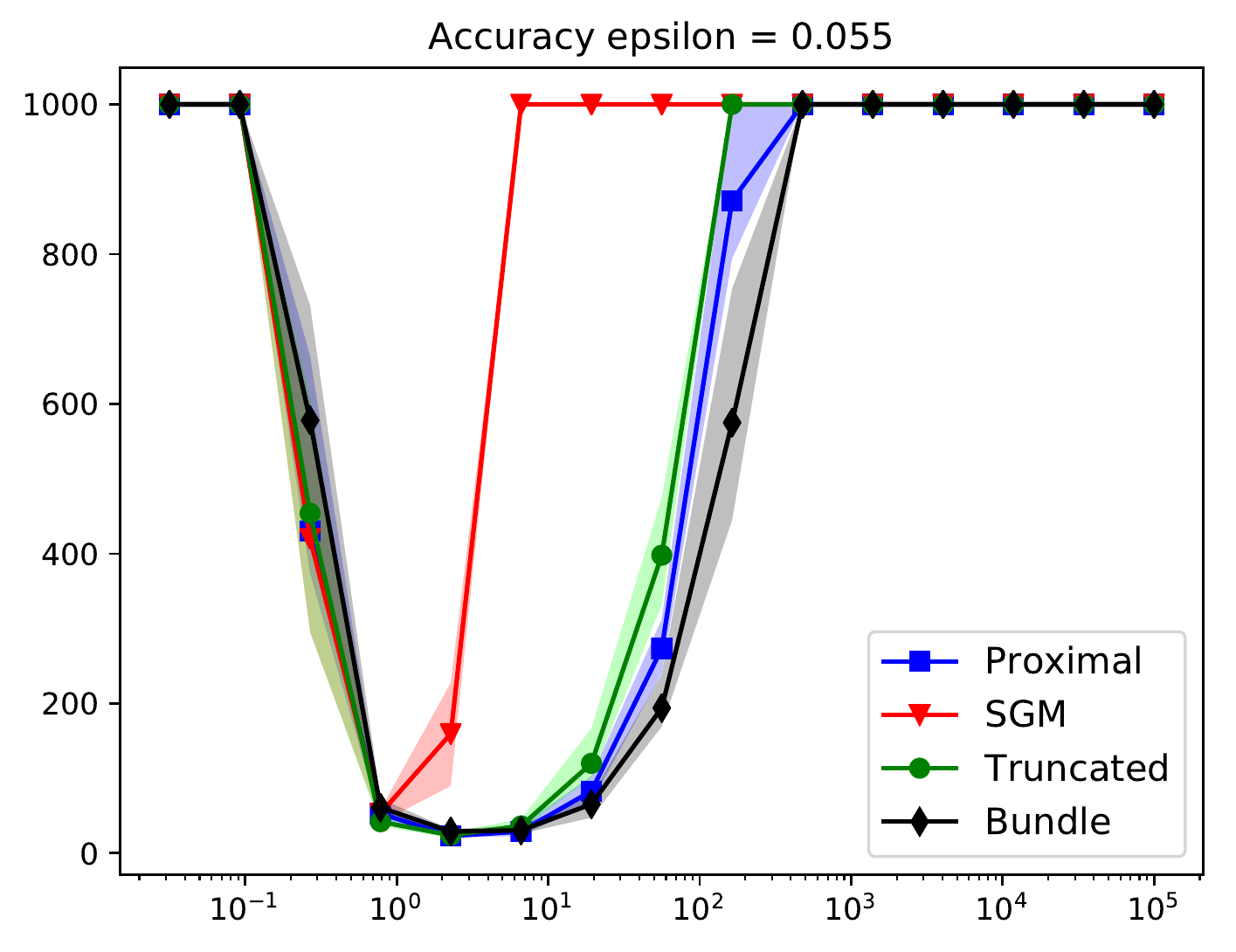}
    	\put(35,-1){{\small Initial stepsize $\stepsize_0$}}
    	\put(-3,15){
    		\rotatebox{90}{{\small Time to accuracy $\epsilon = .05$}}}
    	\put(33,72){
    		\tikz{\path[draw=white,fill=white] (0, 0) rectangle (3cm,.35cm);}
    	}
    \end{overpic} \\
 	(a) & (b)
    \end{tabular}
    \caption{\label{fig:logreg} The number of iterations to achieve
      $\epsilon$-accuracy as a function of the initial step size
      $\stepsize_0$ for logistic regression with $\dimExp = 40$,
      $\numSamplesExp = 1000$, and condition number $\kappa(A) = 1$.
      (a) Noiseless experiment.
      (b) Labels flipped with probability $\flipProp = 0.01$.}
  \end{center}
\end{figure}
\begin{figure}
  \begin{center}
    \begin{tabular}{cc}
        \begin{overpic}[width=.5\columnwidth]{
      		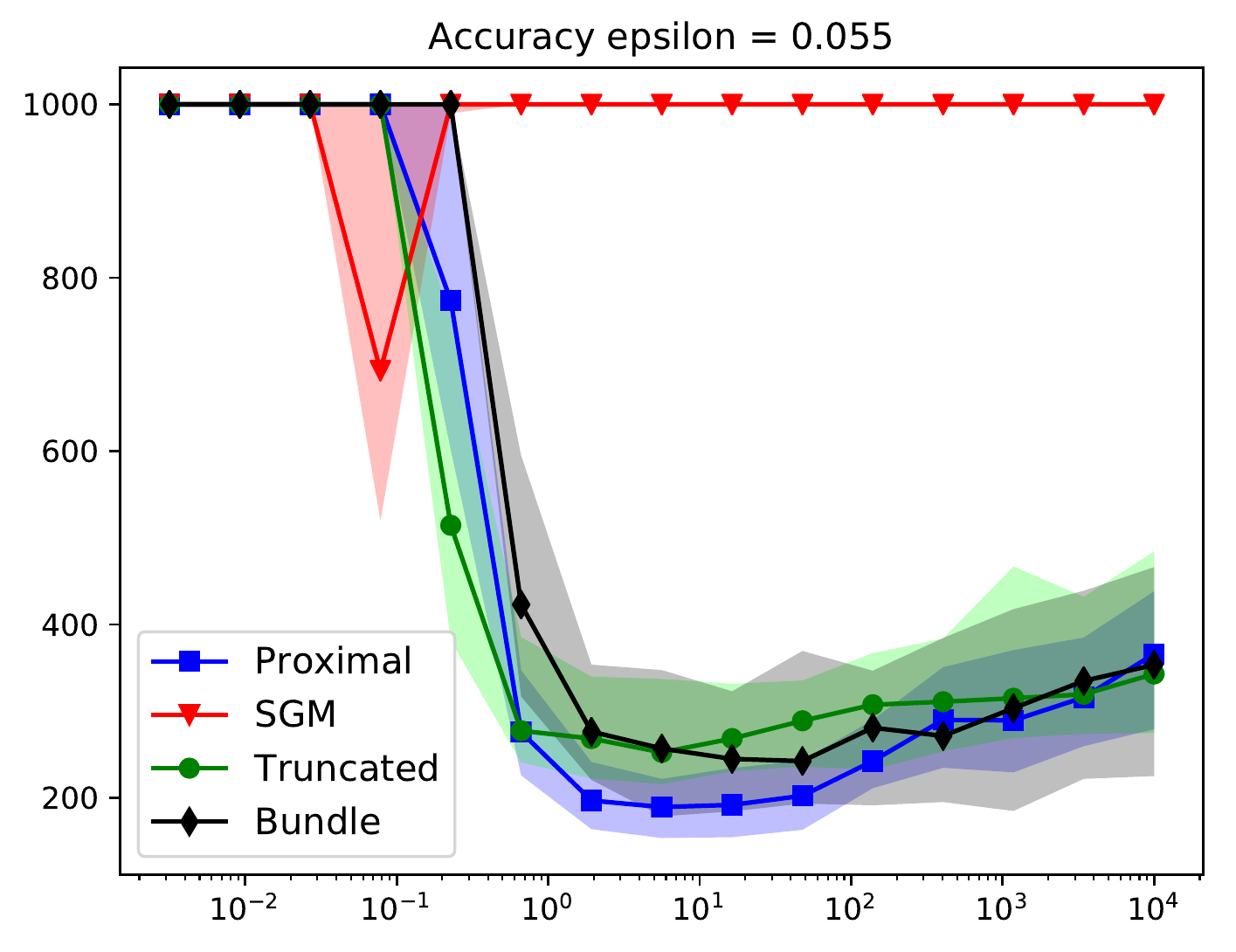}
      	\put(35,-1){{\small Initial stepsize $\stepsize_0$}}
      	\put(-3,15){
      		\rotatebox{90}{{\small Time to accuracy $\epsilon = .05$}}}
      	\put(33,72){
      		\tikz{\path[draw=white,fill=white] (0, 0) rectangle (3cm,.35cm);}
      	}
      \end{overpic} &
      \begin{overpic}[width=.5\columnwidth]{
      		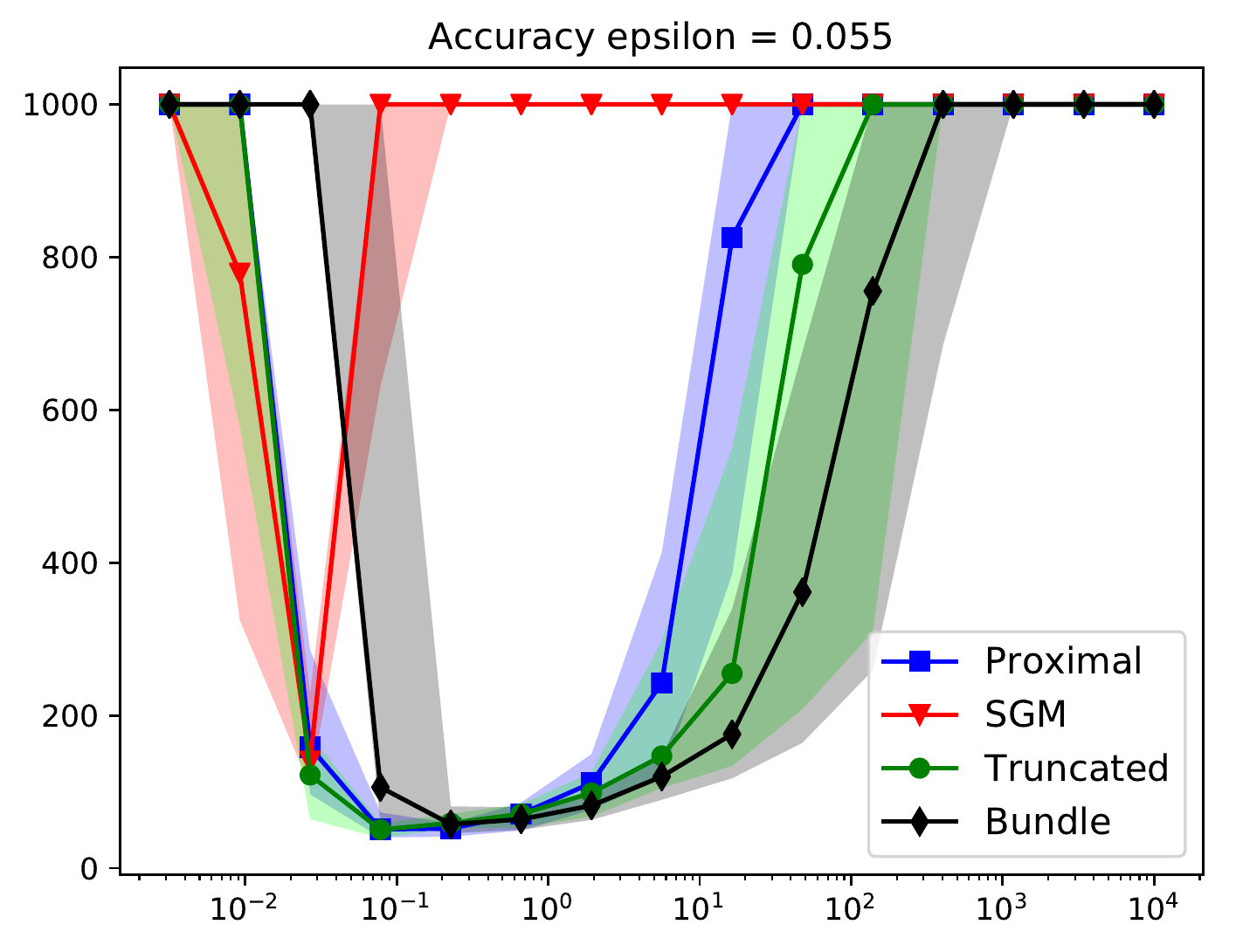}
      	\put(35,-1){{\small Initial stepsize $\stepsize_0$}}
      	\put(-3,15){
      		\rotatebox{90}{{\small Time to accuracy $\epsilon = .05$}}}
      	\put(33,72){
      		\tikz{\path[draw=white,fill=white] (0, 0) rectangle (3cm,.35cm);}
      	}
      \end{overpic} \\
      (a) & (b)
    \end{tabular}
    \caption{
      \label{fig:logreg-condition}
      The number of iterations to achieve $\epsilon$-accuracy as a function
      of the initial step size $\stepsize_0$ for logistic regression with
      parameters $\dimExp = 40$, $\numSamplesExp = 1000$, and
      condition number $\kappa(A) = 15$.
      (a) Noiseless experiment. (b) Labels flipped with probability
      $\flipProp = .01$.}
  \end{center}
\end{figure}
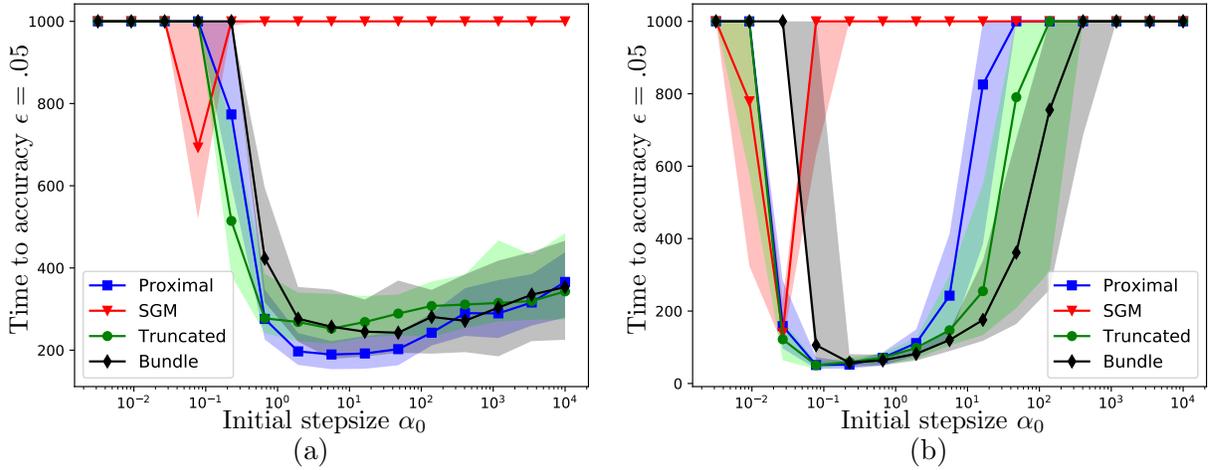

We now turn to classification experiments, beginning with a logistic
regression experiment. In logistic regression, widely used for
fitting models for binary classification in statistics and
machine learning~\cite{HastieTiFr09},
we have
data pairs $(\vecLogReg_i, \labelLogReg_i) \in \R^n \times \{\pm 1\}$,
and we wish to minimize
\begin{equation*}
  F(\varLogReg) \defeq \frac{1}{m} \sum_{i = 1}^m f(\varLogReg;
  (\vecLogReg_i, \labelLogReg_i))
  ~~~ \mbox{where} ~~~
  f(\varLogReg; (\vecLogReg, \labelLogReg))
  = \log\left(1 + e^{-\labelLogReg \<\vecLogReg, \varLogReg\>}\right).
\end{equation*}
We generate the data as follows: we sample $\vecLogReg_i \simiid \normal(0,
I_n)$ and $u\opt \sim \normal(0, I_n)$, labeling
$\labelLogReg_i = \sign(\<\vecLogReg_i, u\opt\>)$; in the noisy setting
we flip each label's sign independently with probability $\flipProp$.

We present the results of this experiment in Figures~\ref{fig:logreg} and
\ref{fig:logreg-condition}, including plots for both the noiseless
(perfectly separated) and noisy cases (plots (a) and (b) in each figure,
respectively), where Fig.~\ref{fig:logreg-condition} displays results when
the condition number of the data matrix $A$ is $\kappa(A) = 15$.  These
plots demonstrate similar results to those in the preceding sections, though
there are a few differences.  First, in the noiseless setting, there is no
optimizer $\varLogReg\opt$, as the optimal value is $\lim_{t \to \infty} F(t
u\opt) = 0$, yet still we see the benefits of the more accurate models in
Figures~\ref{fig:logreg}(a) and \ref{fig:logreg-condition}(a).
Moreover, the truncated and proximal models exhibit a wider range
of convergent stepsizes than the simple stochastic gradient
method does even in the noisy case.

\subsection{Multi Class Hinge Loss}
\label{sec:HingeLoss}

\begin{figure}[ht]
  \begin{center}
    \begin{tabular}{cc}
       \begin{overpic}[width=.5\columnwidth]{
       		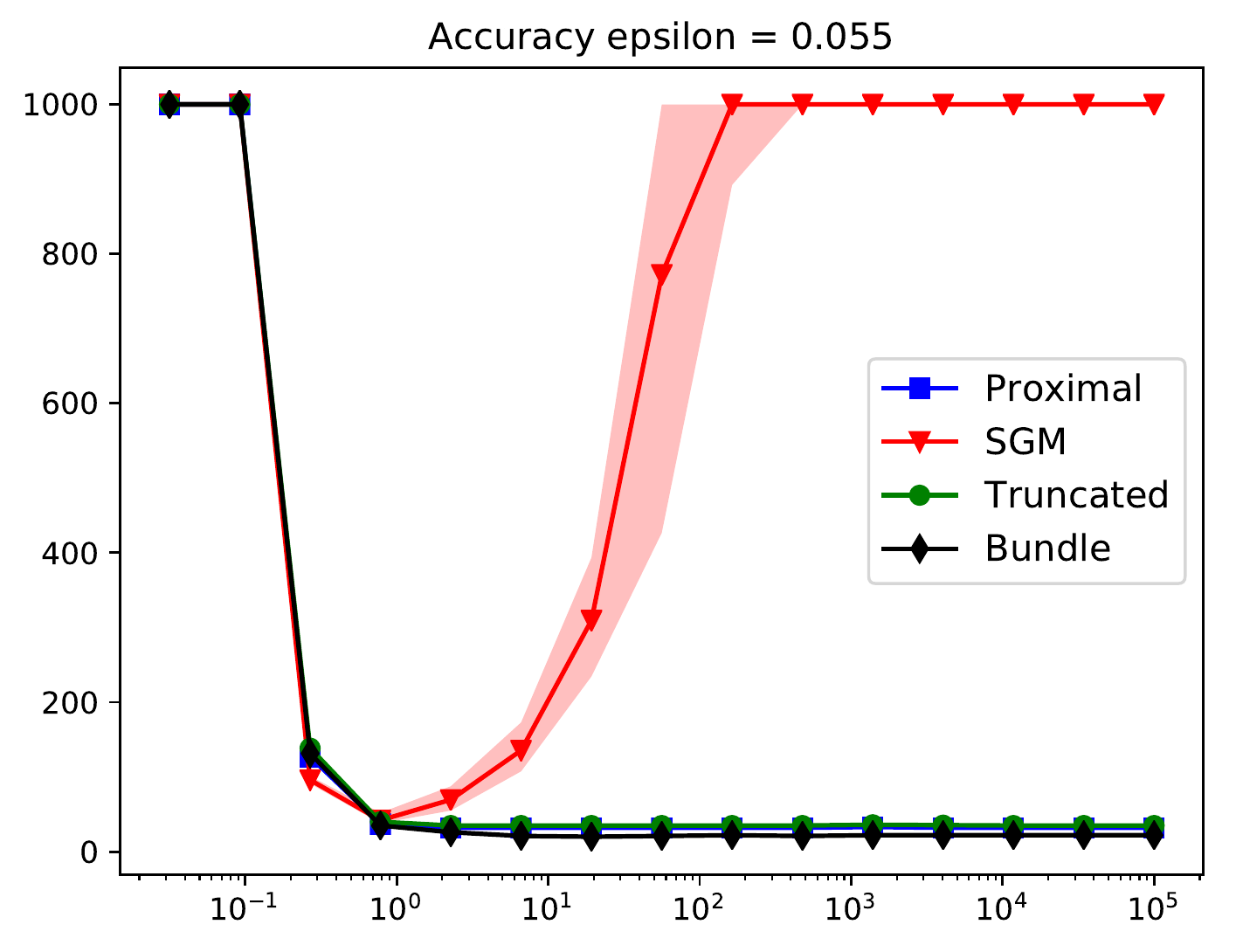}
       	\put(35,-1){{\small Initial stepsize $\stepsize_0$}}
       	\put(-3,15){
       		\rotatebox{90}{{\small Time to accuracy $\epsilon = .05$}}}
       	\put(33,72){
       		\tikz{\path[draw=white,fill=white] (0, 0) rectangle (3cm,.35cm);}
       	}
       \end{overpic} &
       \begin{overpic}[width=.5\columnwidth]{
       		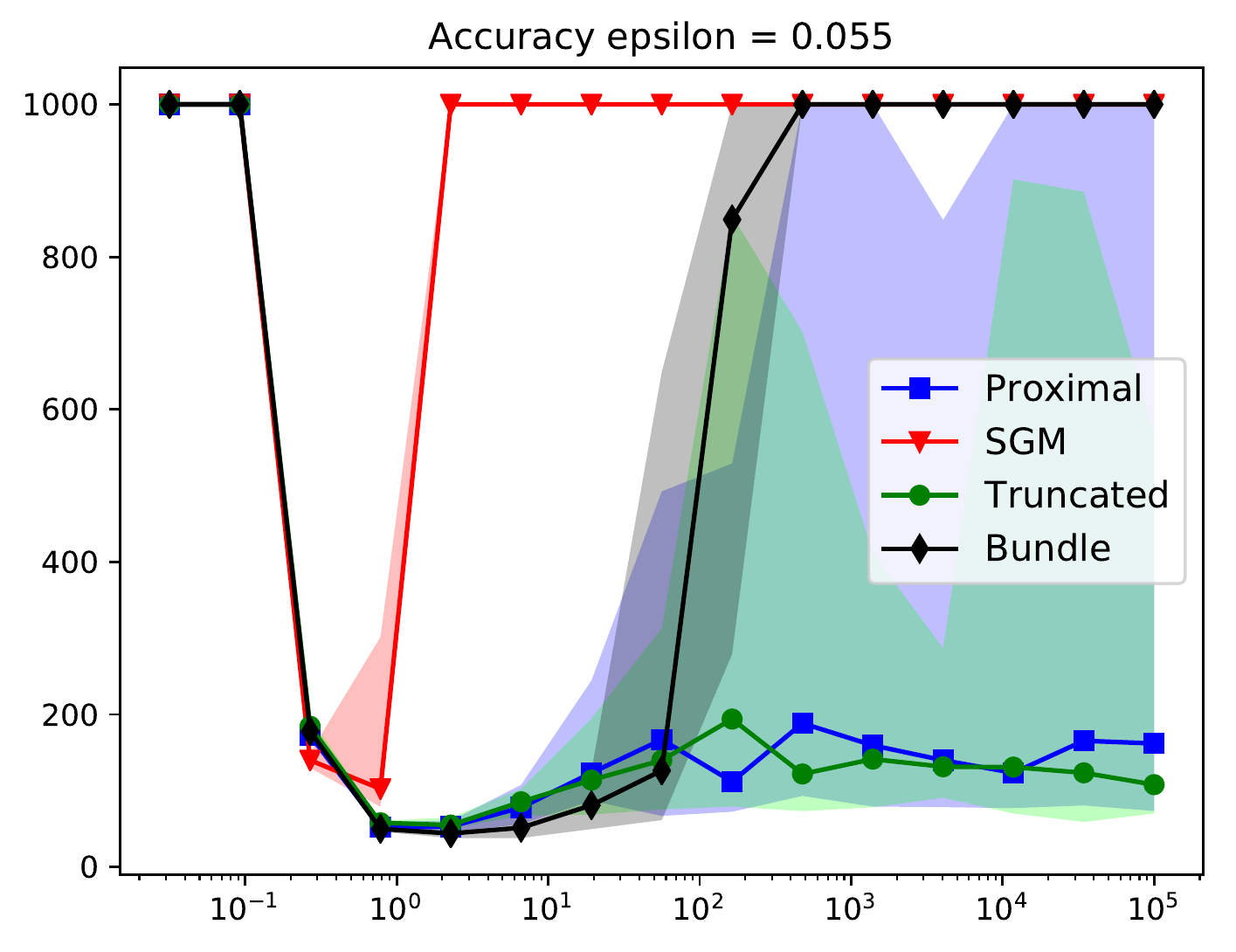}
       	\put(35,-1){{\small Initial stepsize $\stepsize_0$}}
       	\put(-3,15){
       		\rotatebox{90}{{\small Time to accuracy $\epsilon = .05$}}}
       	\put(33,72){
       		\tikz{\path[draw=white,fill=white] (0, 0) rectangle (3cm,.35cm);}
       	}
       \end{overpic} \\
      (a) & (b)
    \end{tabular}
    \caption{\label{fig:HingeLoss} The number of iterations to achieve
      $\epsilon$-accuracy as a function of the initial step size
      $\stepsize_0$ for multi class hinge loss with parameters $\dimExp =
      15$, $\numSamplesExp = 2000$, $\numClassesHinge = 10$ and (a) label
      randomization probability $\flipProp = 0$ and (b) label randomization
      probability $\flipProp = 0.01$.}
  \end{center}
\end{figure}
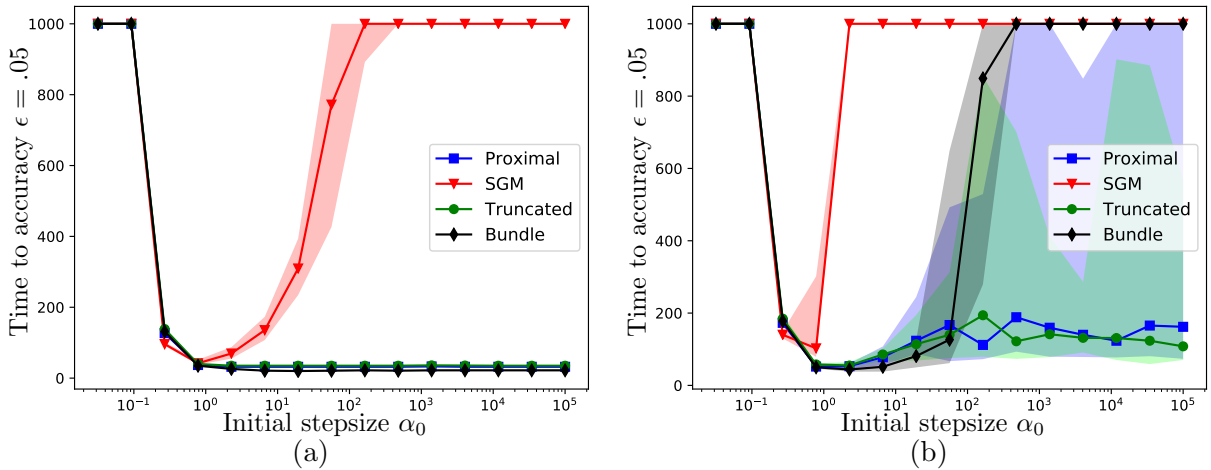

In our second classification experiment, we focus on a somewhat more complex
multi-class setting, using the multi-class hinge
loss~\cite{CrammerSi01a}. In this setting, we receive $\numSamplesExp$
vectors $\vecHinge_i \in \R^{\dimExp}$ and a correct label $\labelHinge_i
\in [\numClassesHinge]$ for each $i$, where $\numClassesHinge$ is the number
of classes. We wish to find a classifier, represented as a collection of $K$
vectors $X = [x_1 ~ \cdots ~ x_\numClassesHinge] \in \R^{\dimExp \times
  \numClassesHinge}$ that minimizes
\begin{equation*}
  F(X) = \frac{1}{\numSamplesExp}
  \sum_{i=1}^{\numSamplesExp} \max_{j \neq \labelHinge_i}
  \hinge{1 + \<\vecHinge_i, x_j - x_{\labelHinge_i}\>}.
\end{equation*}
In the case that the data is separable with a positive margin, meaning that
there exists $X\opt$ such that $\<\vecHinge_i, x_{\labelHinge_i}\> \ge 1 +
\<\vecHinge_i, x_j\>$ for all $j \neq \labelHinge_i$, this problem is
equivalent to finding a point in the intersection of halfspaces
(recall Section~\ref{sec:random-projections}). Accordingly, we expect to see
fast convergence for the truncated~\eqref{eqn:trunc-model} and
proximal-point~\eqref{eqn:prox-model} models for large stepsizes.

To generate the data, we draw vectors $\vecHinge_i \simiid \normal(0,
I_{\dimExp})$, then generate an ``optimal'' classifier $U\opt \in
\R^{\dimExp \times K}$ with i.i.d. $\normal(0, 1)$ entries.  In the
non-noisy setting, we set $\labelHinge_i = \argmax_j \<u\opt_j,
\vecHinge_i\>$, while in the noisy setting, for each $i \in \{1, \ldots,
\numSamplesExp\}$ we resample a value $\labelHinge_i$ uniformly at random
with probability $\flipProp$.  We present the results of this experiment in
Figure~\ref{fig:HingeLoss}. The experiment reinforces the conclusions of the
previous experiments: better models~\eqref{eqn:model-iteration} in the
\aProx family are significantly more robust to the step size values and
achieve generally faster convergence than more naive subgradient methods.

\subsection{Poisson Regression}
\label{sec:Poisson}
\begin{figure}[ht]
  \begin{center}
        \begin{overpic}[width=.5\columnwidth]{
    		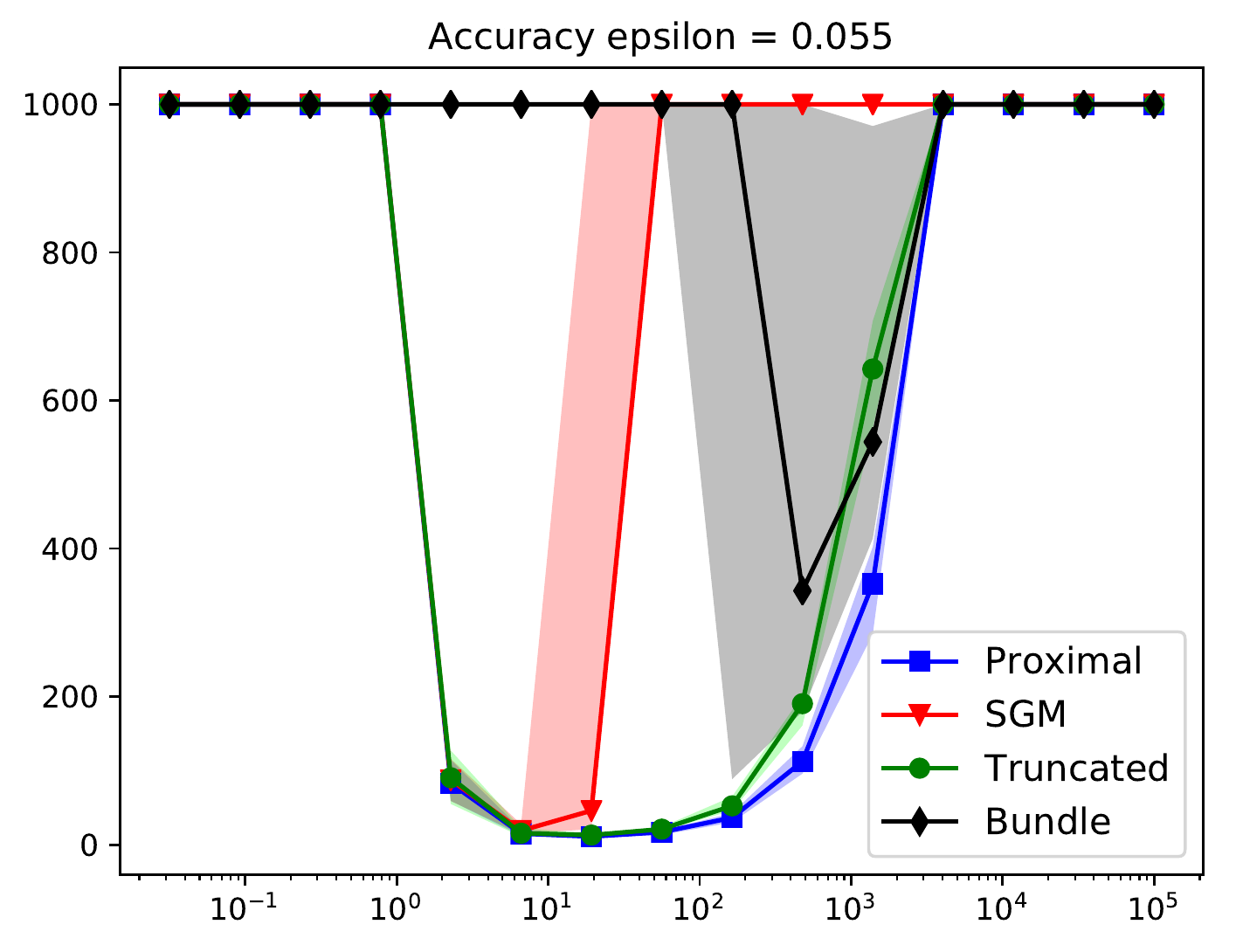}
    	\put(35,-1){{\small Initial stepsize $\stepsize_0$}}
    	\put(-3,15){
    		\rotatebox{90}{{\small Time to accuracy $\epsilon = .05$}}}
    	\put(33,72){
    		\tikz{\path[draw=white,fill=white] (0, 0) rectangle (3cm,.35cm);}
    	}
    	\end{overpic}
    \caption{\label{fig:Poisson} The number of iterations to achieve
      $\epsilon$-accuracy as a function of the initial step size
      $\stepsize_0$ for Poisson regression with parameters $\dimExp = 40$,
      and $\numSamplesExp = 1000$.}
  \end{center}	
\end{figure}
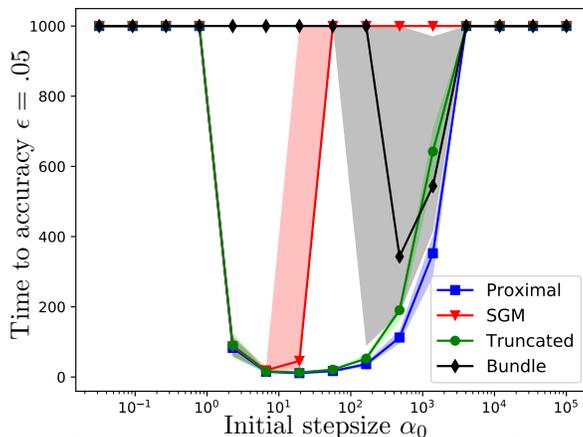

For our final experiment, we consider a poisson regression problem
(Example~\ref{example:poisson-regression}), for which classical results on
stochastic approximation do not apply because of the exponential objective.
In this case, we model counts $b_i \in \N$ as coming from a distribution
$p(b \mid a, x) = \exp(-e^{\<a, x\>}) \exp(b\<a, x\>) / b!$, giving loss
$f(x; (a, b)) = \exp(\<a, x\>) - b \<a, x\>$. We generate the data by first
drawing $u \sim \sqrt{n} \cdot \uniform(\sphere^{n-1})$, then drawing $a_i
\simiid \normal(0, (1/n) I_n)$ and $b_i \sim \poissondist(e^{\<a_i, u\>})$.
In this experiment, the proximal update has no closed-form, as it involves
minimizing a quadratic plus exponential term, but the minimizer of the
proximal update lies on $\{x_k + ta \mid t \in \R\}$, yielding a
1-dimensional convex optimization problem.  As
Example~\ref{example:poisson-regression} shows, however, it is simple to
implement the truncated model~\eqref{eqn:trunc-model} by computing $\inf_x
f(x; (a, b)) = \log(b!) + b - b \log b$.

We present the results in Figure~\ref{fig:Poisson}. It is
surprising to us that the stochastic gradient method converges at all on
this problem, but with low noise scenarios and small enough stepsizes, it
seems that SGM does not leave a region around zero and so is convergent.
Importantly, the truncated
model~\eqref{eqn:trunc-model}, in spite of its substantially easier
calculation, enjoys convergence nearly as robust as that of the stochastic
proximal point method.

\subsection*{Acknowledgments}

We thank two anonymous reviewers for careful reading and feedback on
earlier versions of this paper. In particular, both suggested
strengthenings of Proposition~\ref{proposition:convergence-from-boundedness}
we have implemented.

\appendix


\section{Proof of Theorem~\ref{theorem:always-asymptotic-normality}}
\label{sec:proof-asymptotic-normality}

The proof of the theorem proceeds in a series of lemmas, each of which
requires some work. Roughly, the outline is as
follows: first, we develop a recursion for the iterates $x_{k+1}$, which
parallels a noisy gradient recursion, except that the errors
implicitly depend on the next iterate.  This allows a decomposition (see
Eq.~\eqref{eqn:apply-polyak-juditsky-equality}) of $\frac{1}{k} \sum_{i =
  1}^k (x_i - x\opt)$ into a leading term, which is obviously asymptotically
normal, and several error terms.  Two of these errors are standard
stochastic approximation errors (similar, for example, to those in
\citet{PolyakJu92}), though we require care in showing they are negligible
(see
Lemmas~\ref{lemma:convergence-from-boundedness},
\ref{lemma:an-remainders-zero},
and
\ref{lemma:always-zeta-zero-errors}).  The last error term involves
subgradients of the models $f_{x_k}(\cdot; \statrv_k)$ \emph{at the point
  $x_{k+1}$}, causing an implicit and potentially non-smooth dependence in
the errors.  To address this, we provide a gradient comparison result
(Lemma~\ref{lemma:gradient-comparison}), based on \citet{DavisDrPa17}, which
shows that even if the method generating the iterates $x_k$ uses a
non-smooth approximation to $f(\cdot; \statrv_k)$, near $x_k$ the
subgradients of the model approximate $\nabla f$.
This allows us to adapt the results
of Polyak and Juditsky~\cite{PolyakJu92} on asymptotic optimality of
stochastic gradient methods.

Let $\Delta_k = x_k - x\opt$, and for $H = \nabla^2 F(x\opt)$, define the
remainder $R(x) = \nabla F(x) - H (x - x\opt)$, so that $\norm{R(x)} =
O(\norm{x - x\opt}^2)$ as $x \to x\opt$.  We perform an expansion to rewrite
the implicit iteration $x_{k + 1} = x_k - \stepsize_k f_{x_k}'(x_{k+1};
\statrv_k)$. Defining the localized (sub)gradient errors
\begin{equation}
  \label{eqn:centered-gradient-error}
  \zeta_k = \zeta(x_k, \statrv_k) \defeq \left(\nabla
  f'(x_k;\statrv_k) - \nabla f(x\opt; \statrv_k)\right)
  - \left(F'(x_k) - \nabla F(x\opt)\right),
\end{equation}
we obtain
\begin{align}
  \nonumber
  x_{k + 1}
  & = x_k - \stepsize_k f'_{x_k}(x_{k+1}; \statrv_k) \\
  & = x_k
  - \stepsize_k \Big[
    H (x_k - x\opt)
    + R(x_k)
    + f'(x_k; \statrv_k) - F'(x_k)
    + \underbrace{f'_{x_k}(x_{k+1};\statrv_k) - f'(x_k;\statrv_k)}_{
      =: \varepsilon_k}\Big]\nonumber \\
  & = x_k - \stepsize_k H (x_k - x\opt)
  - \stepsize_k \nabla f(x\opt; \statrv_k)
  - \stepsize_k \left[R(x_k) + \zeta_k + \varepsilon_k\right].
  \label{eqn:implicit-iteration}
\end{align}
Subtracting $x\opt$ to use $\Delta_k = x_k - x\opt$,
the implicit iteration~\eqref{eqn:implicit-iteration} becomes
\begin{equation*}
  \Delta_{k+1} = (I - \stepsize_k H) \Delta_k
  - \stepsize_k \nabla f(x\opt; \statrv_k)
  - \stepsize_k [R(x_k) + \zeta_k + \varepsilon_k].
\end{equation*}
Defining the matrices
\begin{equation*}
  B_i^k \defeq \stepsize_i \sum_{j = i}^k \prod_{l = i + 1}^j
  (I - \stepsize_l H)
  ~~ \mbox{and} ~~
  A_i^k \defeq B_i^k - H^{-1},
\end{equation*}
\citet[Lemma 2]{PolyakJu92} show that
$\wb{\Delta}_k = \frac{1}{k} \sum_{i=1}^k \Delta_i$ satisfies the equality
\begin{equation}
  \begin{split}
    \sqrt{k} \wb{\Delta}_k
    & = \frac{1}{\sqrt{k}}
    \sum_{i = 1}^k H^{-1} \nabla f(x\opt; \statrv_k) \\
    & ~~~
    + \frac{1}{\sqrt{k}} \sum_{i = 1}^k A_i^k \nabla f(x\opt; \statrv_i)
    + \frac{1}{\sqrt{k}}
    \sum_{i = 1}^k B_i^k [R(x_i) + \zeta_i + \varepsilon_i] + O(1/\sqrt{k})
  \end{split}
  \label{eqn:apply-polyak-juditsky-equality}
\end{equation}
and additionally $\sup_{i,k} \norms{B_i^k} < \infty$ and $\lim_k \frac{1}{k}
\sum_{i=1}^k \norms{A_i^k} = 0$.  Evidently,
equality~\eqref{eqn:apply-polyak-juditsky-equality} implies the theorem as
soon as we show each of the terms except $k^{-1/2} H^{-1} \sum_{i=1}^k
\nabla f(x\opt; \statrv_k)$ are $o_P(1)$.  We thus bound each of the terms
in~\eqref{eqn:apply-polyak-juditsky-equality} in turn, which gives
Theorem~\ref{theorem:always-asymptotic-normality}.

\begin{lemma}
  \label{lemma:convergence-from-boundedness}
  Under the conditions of Theorem~\ref{theorem:always-asymptotic-normality},
  $\norm{\Delta_k} \cas 0$,
  $\sum_{k = 1}^\infty \stepsize_k (F(x_k) - F(x\opt)) < \infty$,
  and $\sum_{k = 1}^\infty \stepsize_k \norm{\Delta_k}^2 < \infty$.
\end{lemma}

\begin{lemma}
  \label{lemma:an-remainders-zero}
  Under the conditions of Theorem~\ref{theorem:always-asymptotic-normality},
  $\frac{1}{\sqrt{k}} \sum_{i = 1}^k \norm{R(x_i)} \cas 0$.
\end{lemma}

\begin{lemma}
  \label{lemma:always-zeta-zero-errors}
  Under the conditions of Theorem~\ref{theorem:always-asymptotic-normality},
  $\frac{1}{\sqrt{k}} \sum_{i = 1}^k \zeta_i \cas 0$ and $\frac{1}{\sqrt{k}}
  \sum_{i = 1}^k A_i^k \zeta_i \cas 0$.
\end{lemma}
\noindent
\iftoggle{SIOPT}{%
  \citet{PolyakJu92} prove similar versions of these lemmas,
  so we omit their proofs. (See~\cite[Appendix A]{AsiDu18} for complete proof.)
}{
  See
  Sections~\ref{sec:proof-convergence-from-boundedness},
  \ref{sec:proof-an-remainders-zero}, and
  \ref{sec:proof-always-zeta-zero-errors} for proofs of each of these
  lemmas, respectively.
}

Controlling the implicit modeling errors $\varepsilon_k =
f'_{x_k}(x_{k+1}; \statrv_k) - \nabla f(x_k; \statrv_k)$
is the most important challenge, and for this, we
use the following gradient comparison lemma.

\begin{lemma}[Davis et al.~\cite{DavisDrPa17}, Theorem~6.1]
  \label{lemma:gradient-comparison}
  Let $f$ and $h$ be convex and subdifferentiable on $\R^n$ and $\epsilon >
  0$, $r < \infty$. Assume that on the set $\{x \mid \norm{x - x\opt}
  \le \epsilon\}$, the function $f$ has $\lipgrad$-Lipschitz gradient.
  Assume additionally that $f \ge h$ and at the
  point $x_0$, $h'(x_0) \in \partial f(x_0)$. Then for any $x$ and $h'(x)
  \in \partial h(x)$,
  if $\norm{x - x\opt} \le \epsilon/4$ and $\norm{x_0 - x\opt} \le
  \epsilon/4$, then
  \begin{equation*}
    \norm{\nabla f(x) - h'(x)} \le 2 \lipgrad \norm{x - x_0}.
  \end{equation*}
\end{lemma}
\noindent
Key to the application of Lemma~\ref{lemma:gradient-comparison}
is that individual iterates move very little.

\begin{lemma}
  \label{lemma:single-step-not-far}
  Let Conditions~\ref{cond:convex-model} and~\ref{cond:lower-model}
  hold. Then $\norm{x_k - x_{k+1}} \le \stepsize_k \norm{f'(x_k;\statrv_k)}$
  for some $f'(x_k; \statrv_k) \in \partial f(x_k; \statrv_k)$.
\end{lemma}
\begin{proof}
  Let $g_{k+1} \in \partial f_{x_k}(x_{k+1};\statrv_k)$ satisfy
  $\<\stepsize_k g_{k+1} + (x_{k+1} - x_k), y - x_{k+1}\> \ge 0$ for all $y
  \in \mc{X}$. Then, using $\<g_k - g_{k+1}, x_k - x_{k+1}\> \ge 0$ for any $g_k
  \in \partial f_{x_k}(x_k;\statrv_k) \subset \partial f(x_k;\statrv_k)$ and
  setting $y = x_k$, we have $\stepsize_k \<g_k, x_k - x_{k+1}\> \ge
  \stepsize_k \<g_{k+1}, x_k - x_{k+1}\> \ge \norm{x_k - x_{k+1}}^2$.
  Cauchy-Schwarz gives the result.
\end{proof}

\noindent
With Lemmas~\ref{lemma:gradient-comparison}
and~\ref{lemma:single-step-not-far} in place, we can control the final error
terms in the expansion~\eqref{eqn:apply-polyak-juditsky-equality}.
\iftoggle{SIOPT}{%
}{%
  We provide the proof of the next lemma in
  Section~\ref{sec:proof-implicit-errors-negligible}.
}

\begin{lemma}
  \label{lemma:an-implicit-errors-negligible}
  Under the conditions of Theorem~\ref{theorem:always-asymptotic-normality},
  $\frac{1}{\sqrt{k}} \sum_{i = 1}^k \norm{\varepsilon_i} \cas 0$.
\end{lemma}
\iftoggle{SIOPT}{%
}{%
  This completes the proof of
  Theorem~\ref{theorem:always-asymptotic-normality}.
}

\iftoggle{SIOPT}{%
}{%
  \subsection{Proof of Lemma~\ref{lemma:convergence-from-boundedness}}
  \label{sec:proof-convergence-from-boundedness}

  The first two results are consequences of
  Proposition~\ref{proposition:convergence-from-boundedness}. By the local
  strong convexity of $F$ (Assumption
  \ref{assumption:weak-lipschitz-gradient}\ref{item:local-strong-convexity}),
  for all $r \in \R_+$ there exists $\lambda_r > 0$ such that that $\norm{x -
    x\opt} \le r$ implies $F(x) - F(x\opt) \ge \frac{\lambda_r}{2} \norm{x -
    x\opt}^2$. Thus for some
  (random) $r < \infty$, we have $\sum_k \stepsize_k (F(x_k) - F(x\opt)) \ge
  \frac{\lambda_r}{2} \sum_k \stepsize_k \norm{\Delta_k}^2$, which gives the
  last result of the lemma.

  \subsection{Proof of Lemma~\ref{lemma:an-remainders-zero}}
  \label{sec:proof-an-remainders-zero}

  By Assumption~\ref{assumption:weak-lipschitz-gradient}, $R(x) = \nabla
  F(x) - H (x - x\opt)$ satisfies $R(x) = O(\norm{x - x\opt}^2)$ as $x \to
  x\opt$, and thus for each $r \in \R_+$, there exists some $C_r < \infty$
  such that $\norm{x - x\opt} \le r$ implies $\norm{R(x)} \le C_r
  \norm{x - x\opt}^2$.  Now, for $r \in \R_+$ define the stopping time
  \begin{equation}
    \label{eqn:stopping-time}
    \tau_r \defeq \inf \{k \in \N \mid
    \norm{\Delta_k} \ge r\},
  \end{equation}
  so that $\{\tau_r > k\} \in \mc{F}_{k}$, as $x_{k+1} \in \mc{F}_{k}$. Then
  using Lemma~\ref{lemma:single-recentered-progress} exactly as in the proof
  of Proposition~\ref{proposition:convergence-from-boundedness}, we have
  \begin{align*}
    \lefteqn{
      \E[\norms{\Delta_{k+1}}^2 \indic{\tau_r > k+1} \mid \mc{F}_{k-1}]
      \le \E[\norms{\Delta_{k+1}}^2 \indic{\tau_r > k} \mid \mc{F}_{k-1}]}
    \\
    & \qquad\qquad\qquad ~
    \le \indic{\tau_r > k}
    \left(\norms{\Delta_k}^2 - 2 \stepsize_k (F(x_k) - F\opt)
    + \stepsize_k^2 \growfunc(r)\right).
  \end{align*}
  Again using the local strong convexity of $F$
  (Assumption~\ref{assumption:weak-lipschitz-gradient}), there exists
  $\lambda_r > 0$ such that $F(x_k) - F\opt \ge (\lambda_r/2) \norm{x_k -
    x\opt}^2$ on the event $\{\tau_r > k\}$, so we obtain
  \begin{align}
    \label{eqn:recurse-delta-stops}
    \E[\norms{\Delta_{k+1}}^2 \indic{\tau_r > k+1} \mid \mc{F}_{k-1}]
    \le \indic{\tau_r > k}
    \left((1 - \stepsize_k \lambda_r) \norms{\Delta_k}^2
    + \frac{\stepsize_k^2}{2} \growfunc(r)\right).
  \end{align}
  A technical lemma, whose proof we defer to
  section~\ref{sec:proof-recursive-sum-prod}, helps to control this term.
  \begin{lemma}
    \label{lemma:recursive-sum-prod}
    Let $\stepsize_k = \stepsize_0 k^{-\steppow}$ for some $\steppow \in (0, 1)$
    and $\stepsize_0 > 0$,
    and let $\lambda > 0$ and $\rho > \frac{1}{\beta}$.
    Define
    $p_k \defeq \sum_{i = 1}^k \stepsize_i^\rho \prod_{j = i + 1}^k
    |1 - \lambda \stepsize_j|$.
    Then $\limsup_k p_k / (\stepsize_k^{\rho - 1} \log k) < \infty$. If
    additionally $\lambda \stepsize_1 \le 1$, then
    there exists a numerical constant $C < \infty$ such that
    \begin{equation*}
      p_k \le C \stepsize_0^\rho \left(\frac{1}{k} + \frac{1}{k^{\rho \steppow}}
      \right)
      + C \frac{\log k}{\lambda} \stepsize_k^{\rho - 1}.
    \end{equation*}
  \end{lemma}

  A recursive application of inequality~\eqref{eqn:recurse-delta-stops} and
  (see also~\cite[Pg.~851]{PolyakJu92}) yields
  the bound $\E[\norm{\Delta_k}^2 \indic{\tau_r > k}] \le o(\stepsize_k)
  + C \sum_{i = 1}^k \stepsize_i^2 \prod_{j = i + 1}^k |1 - c
  \stepsize_j|$ for constants $0 < c, C < \infty$, which may
  depend on $r$, and therefore Lemma~\ref{lemma:recursive-sum-prod}
  implies
  \begin{equation}
    \label{eqn:stopping-time-error-stepsize-bound}
    \E[\norm{\Delta_k}^2 \indic{\tau_r > k}] \le C_r \stepsize_k \cdot \log k
  \end{equation}
  for a finite $C_r < \infty$.
  Once we note that $\norm{R(x)} \le C_r \norm{x
    - x\opt}^2$, the remainder of the argument is completely identical to
  that of \citet[Pg.~851]{PolyakJu92}, which demonstrates that $\sum_{k}
  \norm{\Delta_k}^2 / \sqrt{k} < \infty$ with probability 1 once there exists
  a (random) $r < \infty$ such that $\tau_r = \infty$, then applies the
  Kronecker lemma.

  \subsection{Proof of Lemma~\ref{lemma:always-zeta-zero-errors}}
  \label{sec:proof-always-zeta-zero-errors}

  We begin with the first statement. By
  Lemma~\ref{lemma:convergence-from-boundedness}, we have $\sum_k \stepsize_k
  \norm{\Delta_k}^2 < \infty$. Now we apply \citet[Exercise 5.3.35]{Dembo16},
  which states that if $M_k$ is an $\mc{F}_k$-adapted martingale and $b_k
  \uparrow \infty$ are non-random, if $\sum_k b_k^{-2} \E[\norm{M_k -
      M_{k-1}}^2 \mid \mc{F}_{k-1}] < \infty$ then $b_k^{-1} M_k \cas 0$.
  Recall the definition~\eqref{eqn:centered-gradient-error} of $\zeta_k =
  \zeta(x_k, \statrv_k)$ for $\zeta(x,\statval) = (f'(x;\statval) - \nabla
  f(x\opt; \statval)) - (F'(x) - \nabla F(x\opt))$, where $f'(x; \statval) \in
  \partial f(x; \statval)$ and $F'(x) \in \partial F(x)$ (which are both
  singleton sets near $x\opt$ by
  Assumption~\ref{assumption:weak-lipschitz-gradient}).  Then if $\norm{x -
    x\opt} \le \epsilon$,
  Assumption~\ref{assumption:weak-lipschitz-gradient}\ref{item:neighbor-smooth}
  guarantees that $\norm{\zeta(x, \statval)} \le (\lipgrad(\statval) +
  \E[\lipgrad(\statrv)]) \norm{x - x\opt}$, while if $\norm{x - x\opt} \le r$
  we have $\norm{\zeta(x, \statval)} \le \norm{f'(x; \statval)} + \norm{\nabla
    f(x\opt; \statval)} + \growfunc(r)$ by
  Assumption~\ref{assumption:very-weak-moment}. Combining these inequalities,
  we have that whenever $\norm{x - x\opt} \le r$,
  \begin{equation*}
    \E[\norm{\zeta(x, \statrv)}^2]
    \le C_r \norm{x - x\opt}^2
  \end{equation*}
  for a constant $C_r < \infty$.
  Recall the stopping times~\eqref{eqn:stopping-time},
  $\tau_r = \inf\{k \mid \norm{\Delta_k} \ge r\}$,
  and define the truncated variables
  \begin{equation}
    \label{eqn:stopping-time-truncations}
    \zeta_i^{(r)} \defeq \zeta_i \indic{\tau_r > i}
    \in \mc{F}_i.
  \end{equation}
  With probability 1, there is some finite $r$ such that
  $\zeta_i^{(r)} = \zeta_i$ for all $i$ by assumption in
  Theorem~\ref{theorem:always-asymptotic-normality}. Moreover,
  $\E[\zeta_i^{(r)} \mid \mc{F}_{i-1}]
  = \E[\zeta(x_i, \statrv_i) \mid \mc{F}_{i-1}] \indic{\tau_r > i}
  = 0$, and
  \begin{equation*}
    \sum_{i} \frac{1}{i} \E[\norms{\zeta_i^{(r)}}^2 \mid \mc{F}_{i-1}]
    \le C_r \sum_{i} \frac{1}{i}
    \norm{\Delta_i}^2 \le C_r \sum_i \stepsize_i \norm{\Delta_i}^2
    < \infty
  \end{equation*}
  with probability 1 (by Lemma~\ref{lemma:convergence-from-boundedness}).
  This gives $k^{-1/2} \sum_{i = 1}^k \zeta_i \cas 0$ as desired,
  and also shows that $k^{-1/2} \sum_{i = 1}^k H^{-1} \zeta_i \cas 0$.

  We now turn to the second claim of the lemma. We first argue that for all $r
  < \infty$, $k^{-1/2} \sum_{i=1}^k A_i^k \zeta_i^{(r)} \cp 0$, and then use
  continuity to argue that the sequence converges almost surely, which implies
  the result as $k^{-1/2} \sum_{i = 1}^k A_i^k \zeta_i^{(r)} = k^{-1/2}
  \sum_{i = 1}^k A_i^k \zeta_i$ with probability 1 for some $r \in \N$, $r <
  \infty$.  We begin with the convergence in probability.  First, recall
  inequality~\eqref{eqn:stopping-time-error-stepsize-bound} in the proof of
  Lemma~\ref{lemma:an-remainders-zero}, which states that
  $\E[\norm{\Delta_k}^2 \indic{\tau_r > k}] \le C_r \stepsize_k \log k$.
  Then for $r
  < \infty$, we have that
  \begin{equation*}
    \E[\norms{\zeta_i^{(r)}}^2]
    = \E\left[\E[\norms{\zeta(x_i, \statrv_i)}^2 \mid \mc{F}_{i-1}]
      \indic{\tau_r > i}\right]
    \le \E[C_r \norm{x_i - x\opt}^2 \indic{\tau_r > i}]
    \le C_r \stepsize_i \log i
  \end{equation*}
  for a constant $C_r < \infty$ whose value may change from line to line.
  Thus
  \begin{equation*}
    \E\left[\normbigg{\frac{1}{\sqrt{k}}
	\sum_{i = 1}^k A_i^k \zeta_i^{(r)}}^2\right]
    = \frac{1}{k} \sum_{i=1}^k \norms{A_i^k}^2 \E[\norms{\zeta_i^{(r)}}^2]
    \le \frac{C}{k} \sum_{i=1}^k \norms{A_i^k},
  \end{equation*}
  where we have used that
  $\zeta_i^{(r)}$ still form a martingale difference sequence.
  This converges to zero~\cite[Lemma 2]{PolyakJu92}.

  Now, we come to the second claim on the almost sure convergence. Recall
  the definition
  $B_i^k = \stepsize_i \sum_{j=i}^{k} \prod_{l=i+1}^j (I - \stepsize_l H)$,
  so that $A_i^k = B_i^k - H^{-1}$.
  For $r < \infty$, define
  the sequence
  \begin{equation*}
    Z_{k,r} \defeq \sum_{i = 1}^k B_i^k \zeta_i^{(r)}.
  \end{equation*}
  The lemma will be proved if we can show that
  for any finite $r$, we have $k^{-1/2} Z_{k,r} \cas 0$.
  Note that $B_i^{k+1} - B_i^k
  = \stepsize_i \prod_{j = i+1}^{k+1} (I - \stepsize_j H)$,
  and so if we define
  \begin{equation*}
    W_i^k = \prod_{j = i}^k (I - \stepsize_j H),
    ~~~
    V_{k,r} \defeq \sum_{i = 1}^k \stepsize_i W_{i+1}^{k+1} \zeta_i^{(r)},
  \end{equation*}
  then $V_{k,r} \in \mc{F}_k$ and
  \begin{equation}
    Z_{k,r} = Z_{k-1,r} + V_{k-1,r} + \stepsize_{k} \zeta_{k}^{(r)}
    = \sum_{i = 1}^{k-1} V_{i,r} + \sum_{i = 1}^k \stepsize_i
    \zeta_i^{(r)}.
  \end{equation}
  The second sum $\sum_{i = 1}^k \stepsize_i \zeta_i^{(r)}$ is a
  square-integrable martingale with summable squared increments, so it
  converges with probability 1~\cite[Ex.~5.3.35]{Dembo16}, and $1/\sqrt{k}
  \sum_{i = 1}^k \stepsize_i \zeta_i^{(r)} \cas 0$.  It thus suffices to
  show that $\frac{1}{\sqrt{k}} \sum_{i = 1}^k V_{i,r}$ converges almost
  surely, which we do by showing that it is a Cauchy sequence.  Now, we note
  that
  \begin{align*}
    \E[\norm{V_{k,r}}^2]
    & = \sum_{i = 1}^k \stepsize_i^2 \norms{W_{i+1}^{k+1}}^2
    \E[\norms{\zeta_i^{(r)}}^2]
    \le C_r \sum_{i=1}^k \stepsize_i^2 \exp\left(-c
    \sum_{j = i+1}^{k+1} \stepsize_j\right)
    \stepsize_i \log i,
  \end{align*}
  where we have used that
  $\E[\norms{\zeta_i^{(r)}}^2] \lesssim \stepsize_i \log i$ as above.
  As $\stepsize_i^3 \log i \ll \stepsize_i^{3 - \epsilon}$ for all
  $\epsilon > 0$,
  Lemma~\ref{lemma:recursive-sum-prod} implies that for
  all $\epsilon > 0$,
  \begin{equation}
    \label{eqn:intermediate-v-bound}
    \E[\norms{V_{k,r}}^2] \le \frac{C_r}{k^{2 \beta - \epsilon}}.
  \end{equation}

  We use inequality~\eqref{eqn:intermediate-v-bound} to demonstrate that the
  sequence $T_{k,r} \defeq \frac{1}{\sqrt{k}} \sum_{i = 1}^k V_{i,r}$ is
  Cauchy.  Indeed, we have for all $\epsilon > 0$ that
  \begin{align*}
    \E\left[\norm{T_{k,r} - T_{k+1,r}}\right]
    & \le \left|\frac{1}{\sqrt{k+1}} - \frac{1}{\sqrt{k}}\right|
    \sum_{i = 1}^k \E[\norm{V_{i,r}}]
    + \frac{1}{\sqrt{k+1}} \E[\norm{V_{k+1,r}}] \\
    & \le \frac{C}{k^{3/2}}
    \sum_{i=1}^k \frac{1}{i^{\beta - \epsilon}}
    + \frac{1}{\sqrt{k}} \frac{1}{k^{\beta - \epsilon}}
    + \frac{1}{\sqrt{k}} \frac{1}{k^{\beta - \epsilon}}
    \le C k^{-\beta - 1/2 + \epsilon}.
  \end{align*}
  As $\beta \in (\half, 1)$, we have $\sum_k \E[\norm{T_{k,r} - T_{k+1,r}}]
  \lesssim \sum_k k^{-\beta - 1/2 + \epsilon} < \infty$ for sufficiently
  small $\epsilon > 0$.  Thus for all $r < \infty$, the sequence $T_{k,r}$
  is Cauchy with probability 1, so that $k^{-1/2} \sum_{i=1}^k
  V_{k,r}$ converges. As we know that $k^{-1/2}
  \sum_{i=1}^k V_{k,r} \cp 0$ by the previous arguments, the limit point
  must be zero.

  \subsection{Proof of Lemma~\ref{lemma:gradient-comparison}}
  \label{sec:proof-gradient-comparison}

  The result is a consequence of \citet{DavisDrPa17}. A simplified
  variant of their
  theorem follows:
  \begin{lemma}[Theorem 6.1~\cite{DavisDrPa17}]
    \label{lemma:davis-gradient-comparison}
    Let $f$ and $h$ by convex and assume there exists a function $u$ such that
    $0 \le f(y) - h(y) \le u(y)$ for all $y$. Then for any $x$ and $\gamma >
    0$, there exists $\what{x}$ such that for all $h'(x) \in \partial h(x)$,
    there is $f'(\what{x}) \in \partial f(\what{x})$ such that and
    \begin{equation*}
      \norms{x - \what{x}} \le 2 \gamma
      ~~ \mbox{and} ~~
      \norms{h'(x) - f'(\what{x})} \le \frac{u(x)}{\gamma}.
    \end{equation*}
  \end{lemma}

  In our context, we note that
  by assumption,
  \begin{equation*}
    f(y) \le f(x_0) + \<\nabla f(x_0), y - x_0\> +
    \frac{\lipgrad}{2} \norm{x_0 - y}^2
  \end{equation*}
  if $\norm{x_0 - x\opt} \le \epsilon$ and $\norm{y - x\opt} \le \epsilon$,
  while otherwise certainly $f(y) < \infty$.
  Noting that our function $h$ satisfies $h(y) \ge h(x_0)
  + \<h'(x_0), y - x_0\> = f(x_0) + \<f'(x_0), y - x_0\>$,
  we may take the upper bound function
  \begin{equation*}
    u(y) = 
    \begin{cases} \frac{\lipgrad}{2}
      \norm{y - x_0}^2 & \mbox{if~} \norm{y - x\opt} \le \epsilon,
      \norm{x_0 - x\opt} \le \epsilon \\
      +\infty & \mbox{otherwise},
    \end{cases}
  \end{equation*}
  so that $0 \le f(y) - h(y) \le u(y)$ for
  all $y$.

  When $\norm{x - x\opt} \le \epsilon/4$ and $\norm{x_0 - x\opt} \le
  \epsilon/4$. In this case, we choose $\gamma = \half \norm{x - x_0}
  \le \epsilon/4$, so that $\norms{x - \what{x}}
  \le 2 \gamma \le \epsilon/2$ implies that
  $\norms{\what{x} - x\opt} \le 3\epsilon / 4 < \epsilon$. Thus,
  we have
  \begin{align*}
    \norm{h'(x) - \nabla f(x)}
    & \le \norms{h'(x) - \nabla f(\what{x})}
    + \norms{\nabla f(\what{x}) - \nabla f(x)} \\
    & \le \frac{\lipgrad}{2 \gamma} \norms{x - x_0}^2
    + \lipgrad \norms{x - \what{x}}
    \le \frac{\lipgrad}{2 \gamma}
    \norm{x - x_0}^2
    + 2 \lipgrad \gamma
  \end{align*}
  by Lemma~\ref{lemma:davis-gradient-comparison} and our definition of $u$.
  The choice $\gamma = \half \norm{x - x_0}$ gives
  the lemma.

}

\iftoggle{SIOPT}{%
  \begin{proof}
}{%
  \subsection{Proof of Lemma~\ref{lemma:an-implicit-errors-negligible}}
  \label{sec:proof-implicit-errors-negligible}
}
In the implicit iteration~\eqref{eqn:implicit-iteration},
if $x_k, x_{k+1} \in \{x : \norm{x - x\opt} \le \epsilon/4\}$,
Lemma~\ref{lemma:gradient-comparison} shows
that $\varepsilon_k = f'_{x_k}(x_{k+1}; \statrv_k) - \nabla
f(x_k; \statrv_k)$ satisfies
\begin{equation}
  \label{eqn:bound-subgradient-errors}
  \begin{split}
    \norm{\varepsilon_k} \le
    2 \lipgrad(\statrv_k) \norm{x_k - x_{k+1}}
    & \stackrel{(*)}{\le}
    2 \stepsize_k \lipgrad(\statrv_k) \norm{\nabla f(x_k; \statrv_k)} \\
    & \le \stepsize_k \lipgrad(\statrv_k)^2
    + \stepsize_k \norm{\nabla f(x_k; \statrv_k)}^2,
  \end{split}
\end{equation}
where inequality~$(*)$ follows by the single step guarantee in
Lemma~\ref{lemma:single-step-not-far}.

We show that each of the two terms in
inequality~\eqref{eqn:bound-subgradient-errors} has small
sum. We have $\E[\lipgrad^2(\statrv)] < \infty$ and $\E[\sum_k \stepsize_k
  k^{-1/2} \lipgrad^2(\statrv_k)] < \infty$, so the Kronecker
lemma implies $k^{-1/2} \sum_{i=1}^k \stepsize_i
\lipgrad^2(\statrv_i) \cas 0$.
For the second term, let $\epsilon > 0$ be such that
$\sum_k \stepsize_k k^{-1/2 + \epsilon} < \infty$, which must
exist as $\stepsize_k = \stepsize_0 k^{-\beta}$ for $\beta \in (1/2, 1)$.
Define
\begin{equation*}
  Z_k \defeq \frac{1}{k^{1/2 - \epsilon}} \sum_{i = 1}^k \stepsize_i
  \norm{f'(x_i; \statrv_i)}^2,
\end{equation*}
which is adapted to $\mc{F}_k$. Then
\begin{align*}
  \E[Z_{k+1} \mid \mc{F}_k]
  & \le \frac{k^{1/2 - \epsilon}}{(k+1)^{1/2 - \epsilon}}
  Z_k + (k+1)^{-1/2 + \epsilon} \stepsize_{k+1}
  \E[\norm{f'(x_{k+1}; \statrv_{k+1})}^2 \mid \mc{F}_k] \\
  & \le Z_k + (k+1)^{-1/2 + \epsilon}
  \stepsize_{k+1} \growfunc(\norm{\Delta_{k+1}}),
\end{align*}
the last inequality following by
Assumption~\ref{assumption:very-weak-moment}.
On the event that $\sup_k \norm{\Delta_k} < \infty$, we have $\sum_k
k^{-1/2+\epsilon} \stepsize_k \growfunc(\norm{\Delta_k}) < \infty$, so that
the Robbins-Siegmund Lemma~\ref{lemma:robbins-siegmund} applies, and thus
$Z_k \cas Z_\infty$ for some finite random variable $Z_\infty$.
Thus $k^{-\epsilon} Z_k = \frac{1}{\sqrt{k}} \sum_{i = 1}^k
\stepsize_i \norm{f'(x_i; \statrv_i)}^2 \cas 0$.

The definition~\eqref{eqn:model-iteration} of the iteration for
$x_{k+1}$ implies that $f'_{x_k}(x_{k+1}; \statrv_k) =
\stepsize_k^{-1} (x_k - x_{k+1})$, and Lemma~\ref{lemma:single-step-not-far}
gives $\norms{f'_{x_k}(x_{k+1}; \statrv_k)} = \stepsize_k^{-1}
\norms{x_k - x_{k+1}} \le \norms{\partial f(x_k; \statrv_k)}$.
Thus, we obtain
\begin{align*}
  \frac{1}{\sqrt{k}}
  \sum_{i = 1}^k \norm{\varepsilon_i}
  & \le \frac{2}{\sqrt{k}}
  \sum_{i = 1}^k \indic{
    \norm{x_i - x\opt} \ge \epsilon / 4 ~ \mbox{or} ~
    \norm{x_{i+1} - x\opt} \ge \epsilon/4
  }
  \norm{f'(x_i; \statrv_i)} \\
  & \qquad\qquad ~ + \frac{1}{\sqrt{k}}
  \sum_{i = 1}^k (\stepsize_i \lipgrad(\statrv_i)^2
  + \stepsize_i \norm{f'(x_i; \statrv_i)}^2),
\end{align*}
where the second term is a consequence of
inequality~\eqref{eqn:bound-subgradient-errors}. Because $x_k \cas x\opt$
(Lemma~\ref{lemma:convergence-from-boundedness}), both of these terms
converge to zero with probability 1.
\iftoggle{SIOPT}{%
  \end{proof}
}{}
\iftoggle{SIOPT}{%
}{
  \subsection{Proof of Lemma~\ref{lemma:recursive-sum-prod}}
  \label{sec:proof-recursive-sum-prod}

  Because $\stepsize_k$ is decreasing in $k$, there
  exists some $K_0$ such that $k > K_0$ implies
  $\lambda \stepsize_k \le 1$, so that defining
  $C_0 = \prod_{i = 1}^{K_0} (|1 - \lambda \stepsize_i| \vee 1)
  \exp(\lambda \sum_{i = 1}^{K_0} \stepsize_i)$, we have
  \begin{equation}
    \label{eqn:get-the-products}
    \prod_{j = i + 1}^k |1 - \lambda \stepsize_j|
    \le C_0 \exp\left(-\lambda \sum_{j = i+1}^k \stepsize_j\right).
  \end{equation}
  Using that $\stepsize_j = \stepsize_0 j^{-\steppow}$, we have
  $\sum_{j = i+1}^k \stepsize_j
  \ge \stepsize_0 \frac{c}{1 - \steppow} (k^{1 - \steppow}
  - (i + 1)^{1 - \steppow})$ for a numerical constant $c > 0$.
  Consequently, by our definition of $p_k$ we have
  \begin{equation*}
    p_k \le C_0 \sum_{i = 1}^k \stepsize_i^\rho
    \exp\left(- c \frac{\lambda \stepsize_0}{1 - \steppow} (k^{1 - \steppow}
    - (i + 1)^{1 - \steppow})\right).
  \end{equation*}
  Now, for $B > 0$ and $k_0 = B k^\steppow \log k$, we have for a numerical
  constant $c$ that $k^{1 -
    \steppow} - (i + 1)^{1 - \steppow} \ge c B \log k$ for
  all $i \le k_0$. Taking
  $B \ge c(1 - \steppow) / (\lambda \stepsize_0)$,
  where $c$ is a numerical constant,
  we obtain
  \begin{equation*}
    \lambda \sum_{j = i+1}^k \stepsize_j
    \ge c \frac{\lambda \stepsize_0}{1 - \steppow}
    (k^{1 - \steppow} - (i + 1)^{1 - \steppow})
    \ge \log k
  \end{equation*}
  for all $i \le k_0$. Returning to the definition of $p_k$, we thus
  obtain
  \begin{align*}
    p_k \le C_0 \sum_{i = 1}^{k_0} \stepsize_i^\rho \exp(-\log k)
    + C_0 \sum_{i = k_0 + 1}^k \stepsize_i^\rho
    & \le c \frac{C_0 \stepsize_0^\rho}{k} \int_1^k t^{-\rho \steppow} dt
    + \frac{c C_0}{\stepsize_0 \lambda}
    \stepsize_k^\rho \cdot k^\steppow \log k  \\
    & = c \left[\frac{C_0\stepsize_0^\rho}{k}
      \frac{1}{1 - \rho \steppow}
      \cdot \left(k^{1 - \rho \steppow} - 1\right)
      + \frac{C_0 \log k}{\lambda}
      \stepsize_k^{\rho - 1}\right].
  \end{align*}

  Finally, if $\lambda \stepsize_1 \le 1$, then in
  inequality~\eqref{eqn:get-the-products} we may take $C_0 = 1$, giving the
  second result of the lemma.

}


\section{Proofs of fast convergence on easy problems}	

\subsection{Proof of Lemma~\ref{lemma:shared-min-progress}}
\label{sec:proof-shared-min-progress}

We assume without loss of generality that $f(x\opt;\statval) = 0$ for all
$x\opt \in \mc{X}\opt$, as we may replace $f$ with $f - \inf f$.  By
Lemma~\ref{lemma:single-step-progress}, the
update~\eqref{eqn:model-iteration} satisfies
\begin{equation*}
  \half \ltwo{x_{k+1} - x\opt}^2
  \le \half \ltwo{x_k - x\opt}^2
  + \stepsize_k \left[f_{x_k}(x\opt;\statrv_k)
    - f_{x_k}(x_{k+1}; \statrv_k)\right] - \half \ltwo{x_{k+1} - x_k}^2.
\end{equation*}
For shorthand, let $g_k = f'(x_k;\statrv_k)$ and $f_k = f(x_k;\statrv_k)$.
As
$f_{x_k}(x\opt; \statrv_k) \le f(x\opt; \statrv_k) = 0$, and
by Condition~\ref{cond:lower-by-optimal} we have
$f_{x_k}(x_{k+1}; \statrv_k)
\ge \hinge{f_k + \<g_k, x_{k+1} - x_k\>}$, we have
\begin{equation}
  \label{eqn:single-step-with-hinge}
  \ltwo{x_{k+1} - x\opt}^2
  \le \ltwo{x_k - x\opt}^2
  + 2 \stepsize_k \left[f(x\opt;\statrv_k)
    - \hinge{f_k + \<g_k, x_{k+1} - x_k\>} \right]
  - \ltwo{x_{k+1} - x_k}^2.
\end{equation}

If we let $\wt{x}_{k+1}$ denote the unconstrained minimizer
\begin{equation*}
  \wt{x}_{k+1} = \argmin_x\left\{\hinge{f_k + \<g_k, x - x_k\>}
  + \frac{1}{2\stepsize_k} \ltwo{x - x_k}^2\right\}
  = x_k - \lambda_k g_k
  ~~ \mbox{for}~~
  \lambda_k \defeq \min\left\{\stepsize_k, \frac{f_k}{
    \ltwo{g_k}^2}\right\},
\end{equation*}
then because $x_{k+1} \in \mc{X}$ we have
\begin{equation*}
  -\stepsize_k f_{x_k}(x_{k+1};\statrv_k)
  - \half \ltwo{x_{k+1} - x_k}^2
  \le -\stepsize_k f_{x_k}(\wt{x}_{k+1}; \statrv_k)
  - \half \ltwo{\wt{x}_{k+1} - x_k}^2.
\end{equation*}
By inspection, the guarded stepsize $\lambda_k$ guarantees that
$\hinge{f_k + \<g_k, \wt{x}_{k+1} - x_k\>} = f_k - \lambda_k
\ltwo{g_k}^2$, and thus inequality~\eqref{eqn:single-step-with-hinge}
(setting $f(x\opt;\statrv_k) = 0$) yields
\begin{align*}
  \ltwo{x_{k+1} - x\opt}^2
  & \le \ltwo{x_k - x\opt}^2
  - 2 \stepsize_k f_{x_k}(\wt{x}_{k+1}; \statrv_k)
  - \ltwos{\wt{x}_{k+1} - x_k}^2 \\
  & \le \ltwo{x_k - x\opt}^2
  - 2 \lambda_k f_k
  + \lambda_k^2 \ltwo{g_k}^2.
\end{align*}
We have two possible cases: whether
$f_k / \ltwo{g_k}^2 \lessgtr \stepsize_k$.
In the case that $f_k / \ltwo{g_k}^2 \le \stepsize_k$,
we have $\lambda_k = f_k / \ltwo{g_k}^2$ and so
$-2 \lambda_k f_k + \lambda_k^2 \ltwo{g_k}^2 =
- f_k^2 / \ltwo{g_k}^2$. In the alternative
case that $f_k / \ltwo{g_k}^2 > \stepsize_k$, we have
$\lambda_k = \stepsize_k$ and
$\stepsize_k^2 \ltwo{g_k}^2 \le
\stepsize_k f_k$. Combining these cases
gives
\begin{equation*}
  \ltwo{x_{k+1} - x\opt}^2
  \le \ltwo{x_k - x\opt}^2
  - \min\left\{\stepsize_k f_k,
  \frac{f_k^2}{\ltwo{g_k}^2}\right\}.
\end{equation*}

\subsection{Proof of
  Proposition~\ref{proposition:sharp-growth-convex-convergence}}
\label{sec:proof-sharp-growth-convex-convergence}

We adopt a bit of shorthand notation. Let $D_k = \dist(x_k, \mc{X}\opt)$, so
$D_k \in \mc{F}_{k-1}$.  Then Lemma~\ref{lemma:shared-min-progress} implies
that under Assumption~\ref{assumption:expected-sharp-growth}, we have
\begin{align}
  \E[D_{k+1}^2 \mid \mc{F}_{k-1}]
  \le D_k^2 - \min\left\{\lambda_0 \stepsize_k D_k,
  \lambda_1 D_k^2\right\}
  & = \max\left\{(1 - \lambda_1),
  (1 - \lambda_0 \stepsize_k / D_k)\right\} D_k^2 \nonumber \\
  & \le \max\left\{(1 - \lambda_1), (1 - \lambda_0 \stepsize_k / D_1)\right\}
  D_k^2,
  \label{eqn:simple-sharp-recurse}
\end{align}
where we have used that $D_1 \ge D_k$ for all $k \ge 1$ by
Lemma~\ref{lemma:shared-min-progress}.
Inequality~\eqref{eqn:simple-sharp-recurse} immediately implies
that if $\beta \ge 0$, then
\begin{equation*}
  K_0 = \sup\{k \in \N \mid \lambda_0 \stepsize_k > \lambda_1 D_1\}
  = \floor{\left(\frac{\lambda_0 \stepsize_0}{\lambda_1 D_1}\right)^{1/\beta}}
\end{equation*}
is the index
for which
$k \ge K_0$ implies $\lambda_0 \stepsize_k / D_1 \le \lambda_1$. This
gives the first result of the proposition,
$\E[D_{k+1}^2] \le
  \exp(-\lambda_1 \min\{k, K_0\}
  - \frac{\lambda_0}{D_1} \sum_{i = K_0+1}^k \stepsize_i ) D_1^2$
if $\beta \ge 0$. For $\beta < 0$, the same choice of $K_0$ gives
$\E[D_{k+1}^2] \le
  \exp(-\lambda_1 \hinge{k - K_0}
  - \frac{\lambda_0}{D_1} \sum_{i = 1}^{k \wedge K_0} \stepsize_i
  ) D_1^2$.

For the second result, we prove only the case that $\beta > 0$,
as the other case is similar. Note that
$\sum_{i=1}^k \stepsize_i \gtrsim k^{1 - \beta}$, so for any $\epsilon >
0$ we have for constants $0 < c, C < \infty$ depending on $\stepsize_0$,
$\beta$, $\lambda_0$ and $\lambda_1$ that
\begin{equation*}
  \sum_{k = 1}^\infty \P(D_k > \epsilon \stepsize_k)
  \le \frac{1}{\epsilon} \sum_{k = 1}^\infty
  \exp\left(C - c k^{1 - \beta}
  + \log \frac{1}{\stepsize_k} \right)
  < \infty.
\end{equation*}
The Borel Cantelli lemma
implies that $D_k / \stepsize_k \cas 0$. The
first inequality~\eqref{eqn:simple-sharp-recurse} implies that
if $V_k = D_{k+1}^2 / (1 - \lambda_1)^{k+1}$, then
\begin{equation*}
  \E[V_k \mid \mc{F}_{k-1}] \le
  \frac{D_k^2}{(1 - \lambda)^k}
  \max\left\{1, \frac{1 - \stepsize_k / D_k}{1 - \lambda_1}\right\}
  \le \left(1 + \hinge{\frac{1 - \stepsize_k / D_k}{1 - \lambda_1}}\right)
  V_{k-1}.
\end{equation*}
As $\hinge{(1 - \stepsize_k / D_k)/(1 - \lambda_1)} = 0$ eventually
with probability 1 and is $\mc{F}_{k-1}$-measurable,
the Robbins-Siegmund Lemma~\ref{lemma:robbins-siegmund} implies that
$V_k \cas V_\infty$ for some $V_\infty \in \R_+$.

\iftoggle{SIOPT}{%
}{
\subsection{Proof of Lemma~\ref{lemma:expected-kaczmarz-growth}}
\label{sec:proof-expected-kaczmarz-growth}

Let $\Delta = x - x\opt$ for shorthand, noting
that $f_i(x\opt) = 0$, and note that for
any $c > 0$ we have
\begin{align*}
\frac{1}{m} \sum_{i = 1}^m
\min\left\{\stepsize f_i(x),
\frac{f_i(x)^2}{\ltwo{f_i'(x)}^2}\right\}
& \ge
\frac{1}{m} \sum_{i=1}^m
\min\left\{\stepsize |\<a_i, \Delta\>|,
\frac{|\<a_i, \Delta\>|^2}{\ltwo{a_i}^2}\right\} \\
& \ge \frac{1}{m} \sum_{i = 1}^m c \ltwo{\Delta}
\indic{|\<a_i, \Delta\>| \ge c \ltwo{\Delta}}
\min\left\{\stepsize, \frac{c \ltwo{\Delta}}{\lipobj^2}\right\},
\end{align*}
where the final inequality uses that $\ltwo{a_i} \le \lipobj$ by assumption.
Now, the $a_i \in \R^n$ are independent from some distribution on $\R^n$
satisfying $\P(|\<a_i, \Delta\>| \ge c \ltwo{\Delta}) \ge p_c > 0$ by
assumption.  Using that the VC-dimension of linear functions $v \mapsto
\<a_i, v\>$ is $O(n)$ a standard VC argument~\cite[Ch.~2.6]{VanDerVaartWe96}
implies that for a numerical constant $C < \infty$, for any $t \ge 0$ with
probability at least $1 - e^{-t}$ over the random draw of the $a_i$, we have
\begin{equation*}
\frac{1}{m} \sum_{i = 1}^m \indic{|\<a_i, \Delta\>| \ge c \ltwo{\Delta}}
\ge p_c - C \sqrt{\frac{n + t}{m}}
\end{equation*}
simultaneously for all $\Delta \in \R^n$.
This implies the lemma.
}

\iftoggle{SIOPT}{
  \section{Proof of Lemma~\ref{lemma:prox-point-strong-convexity}}
  \label{sec:proof-prox-point-strong-convexity}
}{
  \section{Proofs of results from Section~\ref{sec:everything-is-fine}}
}

\iftoggle{SIOPT}{%
}{%
  \subsection{Proof of Proposition~\ref{proposition:convex-lipschitz}}
  \label{sec:proof-convex-lipschitz}
  
  The proposition follows nearly directly from
  Lemma~\ref{lemma:single-recentered-progress}.  Indeed, applying that lemma,
  for any $x\opt \in \mc{X}\opt$ we have
  \begin{equation*}
    \half \E[\ltwo{x_{k + 1} - x\opt}^2 \mid \mc{F}_{k-1}]
    \le \half \ltwo{x_k - x\opt}^2
    - \stepsize_k [F(x_k) - F(x\opt)]
    + \frac{\stepsize_k^2}{2} \lipobj^2,
  \end{equation*}
  where we have used Assumption~\ref{assumption:lipschitz} and
  that $x_{k+1} \in \mc{F}_{k-1}$. Rearranging,
  we have
  \begin{equation*}
    \stepsize_k \E[F(x_k) - F(x\opt)]
    \le \half \E\left[\ltwo{x_k - x\opt}^2 - \ltwo{x_{k+1} - x\opt}^2\right]
    + \frac{\stepsize_k^2}{2} \lipobj^2.
  \end{equation*}
  Summing and telescoping yields $\sum_{i = 1}^k \stepsize_i \E[F(x_i) -
    F(x\opt)] \le \half \ltwo{x_1 - x\opt}^2 + \half \sum_{i = 1}^k
  \stepsize_i^2 \lipobj^2$, and dividing by $\sum_{i = 1}^k \stepsize_i$ and
  using convexity gives the first result of the proposition.
  
  For the second result, we rearrange the first display above
  to see that
  \begin{equation*}
    \E[F(x_k) - F(x\opt)]
    \le \frac{1}{2 \stepsize_k}
    \E[\ltwo{x_k - x\opt}^2 - \ltwo{x_{k + 1} - x\opt}^2]
    + \frac{\stepsize_k}{2} \lipobj^2.
  \end{equation*}
  Summing this quantity, we obtain
  \begin{equation*}
    \sum_{i = 1}^k \E[F(x_i) - F(x\opt)]
    \le \sum_{i = 2}^k \left(\frac{1}{2 \stepsize_i}
    - \frac{1}{2 \stepsize_{i - 1}}\right) \E[\ltwo{x_i - x\opt}^2]
    + \frac{1}{2 \stepsize_1} \ltwo{x_1 - x\opt}^2
    + \frac{\lipobj^2}{2} \sum_{i = 1}^k \stepsize_i.
  \end{equation*}
  Noting that $\E[\ltwo{x_i - x\opt}^2] \le R^2$ by assumption,
  dividing by $k$ and applying Jensen's inequality
  to $F(\wb{x}_k) \le \frac{1}{k} \sum_{i = 1}^k F(x_i)$ gives the result.
}

\iftoggle{SIOPT}{}{
  \subsection{Proof of Lemma~\ref{lemma:prox-point-strong-convexity}}
  \label{sec:proof-prox-point-strong-convexity}
}

Let $\Sigma_k = \Sigma(\statrv_k)$, so that
for all $g_k \in \partial f(x_{k+1};\statrv_k)$ and $y \in \mc{X}$
we have
\begin{equation*}
  f(y; \statrv_k) \ge
  f(x_{k+1}; \statrv_k) + \<g_k, y - x_{k+1}\>
  + \half (y - x_{k+1})^T \Sigma_k (y - x_{k+1}).
\end{equation*}
Using this inequality in place of the last step of the proof of
Lemma~\ref{lemma:single-step-progress} yields
\begin{align}
  \label{eqn:single-step-strongly-convex}
  \lefteqn{\half\ltwo{x_{k + 1} - x\opt}^2
    + \frac{\stepsize_k}{2}
    (x_{k + 1} - x\opt)^T \Sigma_k (x_{k + 1} - x\opt)} \\
  & \le \half \ltwo{x_k - x\opt}^2
  - \stepsize_k \left[f(x_{k + 1}; \statrv_k)
    - f(x\opt; \statrv_k)\right]
  - \half \ltwo{x_k - x_{k+1}}^2 \nonumber.
\end{align}
Using that
\begin{align*}
  \lefteqn{\half (x_{k + 1} - x\opt)^T \Sigma_k (x_{k + 1} - x\opt)} \\
  & = \half (x_k - x\opt)^T \Sigma_k (x_k - x\opt)
  + \half (x_{k + 1} - x_k)^T \Sigma_k (x_{k+1} - x_k)
  + (x_{k + 1} - x_k)^T \Sigma_k (x_k - x\opt) \\
  & \ge \frac{1 - \eta}{2} (x_k - x\opt)^T \Sigma_k (x_k - x\opt)
  + \frac{\eta - 1}{2 \eta} (x_{k + 1} - x_k)^T \Sigma_k (x_{k+1} - x_k)
\end{align*}
for all $\eta > 0$, where the inequality follows from Young's inequality.
Inequality~\eqref{eqn:single-step-strongly-convex} then implies
\begin{align*}
  \half \ltwo{x_{k + 1} - x\opt}^2
  & \le \half (x_k - x\opt)^T \left(I - \stepsize_k (1 - \eta) \Sigma_k\right)
  (x_k - x\opt)
  - \stepsize_k \left[f(x_{k + 1}; \statrv_k)
    - f(x\opt; \statrv_k)\right] \\
  & \qquad ~ - \frac{1}{2} (x_k - x_{k+1})^T \left(I +
  \frac{\stepsize_k (\eta - 1)}{\eta} \Sigma_k\right)
  (x_k - x_{k + 1})
\end{align*}
for all $\eta > 0$.
The choice
$\eta_k = \frac{2 \stepsize_k \lambdamax(\Sigma_k)}{
  1 + 2 \stepsize_k\lambdamax(\Sigma_k)}
\in (0, 1)$ is sufficient to guarantee
$I + \frac{\stepsize_k (\eta_k - 1)}{\eta_k} \Sigma_k \succeq \half I$.
%
%
Substituting this choice of $\eta_k$ above, we obtain that
\begin{align*}
  \half \ltwo{x_{k + 1} - x\opt}^2
  & \le \half (x_k - x\opt)^T \left(I - \stepsize_k
  (1 - \eta_k) \Sigma_k\right)
  (x_k - x\opt) \\
  & \qquad ~ - \stepsize_k \left[f(x_{k + 1}; \statrv_k)
    - f(x\opt; \statrv_k)\right]
  - \frac{1}{4} \ltwo{x_k - x_{k+1}}^2.
\end{align*}
Applying Lemma~\ref{lemma:conv-bound} with 
$y = x\opt$, $x_1 = x_{k + 1}$, $x_0 = x_k$, $\beta = 2 \stepsize_k$,
and $g(\cdot) = f(\cdot; \statrv_k)$
implies
\begin{align*}
  \lefteqn{\half \ltwo{x_{k+1} - x\opt}^2} \\
  & \le (x_k - x\opt)^T \left(I - \stepsize_k (1 - \eta_k) \Sigma_k\right)
  (x_k - x\opt)
  + 2 \stepsize_k\<f'(x\opt; \statrv_k), x\opt - x_k\>
  + \stepsize_k^2 \ltwo{f'(x\opt; \statrv_k)}^2.
\end{align*}
Taking expectations and using
that $\E[\<f'(x\opt; \statrv), x\opt - x\>] \le 0$ for all $x$
(as in the proof of Theorem~\ref{theorem:prox-point-bounded})
gives the desired result.

\iftoggle{SIOPT}{
  {\small
    \setlength{\bibsep}{3pt}
    \bibliography{bib}
    \bibliographystyle{abbrvnat}
  }
}{
  \setlength{\bibsep}{3pt}
  \bibliography{bib}
  \bibliographystyle{abbrvnat}
}

\end{document}